\documentclass[a4paper,11pt,reqno]{amsart}

\usepackage[utf8]{inputenc}
\usepackage[margin=1in]{geometry}

\usepackage{amsmath}
\usepackage{amssymb}
\usepackage{amscd}
\usepackage{mathtools}
\usepackage{mathrsfs}

\usepackage{bbm}

\usepackage{graphicx}
\usepackage{tikz}
\usetikzlibrary{positioning,arrows.meta,shapes.geometric,calc}

\usepackage{enumitem}

\usepackage{color}

\usepackage[colorlinks=true, urlcolor=blue, citecolor=blue, linkcolor=blue, linktocpage, pdfpagelabels, bookmarksnumbered, bookmarksopen]{hyperref}
\usepackage[margin=1in]{geometry}
\usepackage{tikz}
\usetikzlibrary{positioning,arrows.meta,shapes.geometric,calc}
\usepackage{caption}

\newtheorem{theorem}{Theorem}[section]
\newtheorem{lemma}[theorem]{Lemma}
\newtheorem{proposition}[theorem]{Proposition}
\newtheorem{corollary}[theorem]{Corollary}

\theoremstyle{definition}
\newtheorem{definition}[theorem]{Definition}
\newtheorem{remark}[theorem]{Remark}

\numberwithin{equation}{section}

\allowdisplaybreaks

\renewcommand{\d}{\mathrm{d}}
\renewcommand{\div}{\operatorname{div}}

\newcommand{\loc}{\mathrm{loc}}

\newcommand{\R}{\mathbb{R}}
\newcommand{\N}{\mathbb{N}}
\newcommand{\M}{\mathbb{M}}
\newcommand{\Po}{\mathbb{P}}
\newcommand{\T}{\mathcal{T}}

\title{$H$-convergence and $\Gamma$-convergence in the Riesz fractional setting: the nonlinear case}

\author[G.\ C.\ Brusca]{Giuseppe C.\ Brusca}
\address[G.\ C.\ Brusca]{
\newline\indent Scuola Internazionale Superiore di Studi Avanzati (SISSA)
\newline\indent via Bonomea 265, Trieste, Italy}
\email{gbrusca@sissa.it}

\author[M.\ Caponi]{Maicol Caponi}
\address[M.\ Caponi]{
\newline\indent Dipartimento di Ingegneria e Scienze dell'Informazione e Matematica (DISIM)
\newline\indent Università degli Studi dell'Aquila 
\newline\indent via Vetoio 1, Coppito, 67100 L'Aquila, Italy}
\email{maicol.caponi@univaq.it}

\author[A.\ Carbotti]{Alessandro Carbotti}
\address[A.\ Carbotti]{
\newline\indent Dipartimento di Matematica e Fisica ``Ennio De Giorgi''
\newline\indent Università del Salento
\newline\indent via Per Arnesano, 73100 Lecce, Italy}
\email{alessandro.carbotti@unisalento.it}

\author[A.\ Maione]{Alberto Maione}
\address[A.\ Maione]{
\newline\indent Dipartimento di Matematica
\newline\indent Politecnico di Milano
\newline\indent Campus Leonardo, Edificio 14 ``Nave", via Edoardo Bonardi 9, 20133 Milano,  Italy}
\email{alberto.maione@polimi.it}

\author[F.\ Paronetto]{Fabio Paronetto}
\address[F.\ Paronetto]{
\newline\indent Dipartimento di Matematica ``Tullio Levi Civita''
\newline\indent Università di Padova
\newline\indent via Trieste 63, 35121 Padova, Italy}
\email{fabio.paronetto@unipd.it}



\begin{document}


\begin{abstract}
This paper concerns the $H$-convergence of nonlinear nonlocal monotone operators defined through the Riesz fractional gradient and divergence. We show that the $H$-convergence in this nonlocal framework is equivalent to the $H$-convergence of the corresponding local one. As a consequence, we obtain a $H$-compactness result for a suitable class of nonlocal monotone operators. We then study the $\Gamma$-convergence of nonlocal energy functionals associated with the subclass of conservative monotone operators, proving that it is equivalent to the $\Gamma$-convergence of the corresponding local energies. A key ingredient is a new uniqueness result for the integral representation of both local and nonlocal functionals. As a by-product, we obtain the $\Gamma$-compactness of the class of nonlocal energies under consideration. Finally, we show the equivalence between the $H$-convergence of nonlocal conservative monotone operators and the $\Gamma$-convergence of the associated energy functionals.
\end{abstract}

\maketitle




\section{Introduction}
$H$-convergence is a mathematical tool for the homogenisation of composite materials and dates back to the late 1960s with the pioneering works of Spagnolo~\cite{Spagnolo, Spagnolo67, Spagnolo68} and De Giorgi-Spagnolo~\cite{DGS}, where the linear symmetric case was studied under the name of $G$-convergence, extended by Tartar~\cite{Tartar0, tartar-murat} to the possibly non-symmetric setting; see also~\cite{MuratI,Tartar} for a modern account.
The theory considers sequences of measurable matrix-valued functions $B_h$ satisfying the following standard uniform ellipticity and boundedness assumption
\begin{equation}\label{eq:standard_growth}
    \lambda |\xi|^2 \le B_h(x)\xi\cdot\xi \le \Lambda |\xi|^2,
\end{equation}
and focuses on the asymptotic behaviour of the solutions of the family of Dirichlet problems
\begin{equation}\label{eq:Lh}
\begin{cases}
-\div(B_h(x)\nabla w_h(x)) = g(x) & \text{in } \Omega, \\
w_h(x) = 0 & \text{on } \partial\Omega,
\end{cases}
\end{equation}
by identifying a suitable measurable matrix-valued function $B_\infty$, called the $H$-limit, such that the solutions $w_h$ converge to the solution $w_\infty$ of the Dirichlet problem
\begin{equation*}
\begin{cases}
-\div(B_\infty(x)\nabla w_\infty(x)) = g(x) & \text{in } \Omega, \\
w_\infty(x) = 0 & \text{on } \partial\Omega.
\end{cases}
\end{equation*}
A central question was whether the class of matrix-valued functions satisfying~\eqref{eq:standard_growth} is relatively closed under $H$-convergence, namely that every $H$-limit belongs to the same class.

After the linear case had been thoroughly understood, 
several extensions to nonlinear settings were considered in~\cite{CPDMDF,DASC,FM86,FM87,Maione,MPV,Tartar}. 
The authors introduced different classes of monotone operators, depending on the structural assumptions adopted to extend~\eqref{eq:standard_growth} to the nonlinear setting.
Some of these classes are closed under $H$-convergence (see, e.g.,~\cite{CPDMDF,Tartar}), whereas for others the $H$-limits are known only to belong to a larger class (see e.g.~\cite{DASC}).

In the fractional setting, several works have recently addressed the study of $H$-convergence of nonlocal linear operators; we refer in particular to the contributions~\cite{bbcejw,bpw,CCM25,FBRS17,Waurick,Waurick2}.
In~\cite{FBRS17}, the authors consider scalar perturbations of the Gagliardo seminorm, while in~\cite{Waurick,Waurick2} a topological technique is introduced to approach the nonlocal setting in the Hilbertian case.
In~\cite{CCM25}, instead, the authors study $H$-convergence for the following sequence of nonlocal Dirichlet problems in fractional divergence form
\begin{equation*}
\begin{cases}
-\div^s(B_h(x)\nabla^s u_h(x)) = f(x) & \text{in } \Omega, \\
u_h(x) = 0 & \text{in } \mathbb{R}^n \setminus \Omega.
\end{cases}
\end{equation*}
The operators $\nabla^s$ and $\div^s$ are the so-called Riesz fractional gradient and fractional divergence, introduced in~\cite{ShiehSpector15,ShiehSpector18} and extensively studied in, e.g.,~\cite{BCCS22,ComiStefani19,ComiStefani22,ComiStefani23,Stefani}. 
They make it possible to formulate the problems above in fractional divergence form and satisfy the identity $-\operatorname{div}^s\nabla^s=(-\Delta)^{s}$, so that the operators $-\div^s(B_h\nabla^s\,)$ may be viewed as perturbations of the fractional Laplacian.
The main result of~\cite{CCM25}, Theorem~3.1, shows that the class of possibly non-symmetric matrices satisfying~\eqref{eq:standard_growth} in $\Omega$ and sharing a common extension outside $\Omega$ is closed under $H$-convergence. 
More recently, in~\cite{bbcejw}, the authors address some of the open questions raised in~\cite[Section~6]{CCM25}.
In particular, they obtain $H$-compactness without prescribing the matrices outside $\Omega$, which was a technical requirement in~\cite{CCM25}, and include the one-dimensional case $n=1$.
They also observe that, in this Hilbertian setting, the results obtained in~\cite{CCM25} can be alternatively deduced from the topological notion of nonlocal $H$-convergence, introduced in~\cite{Waurick}.

In this paper, we address another open problem raised in~\cite{CCM25}, which is the $H$-convergence of nonlinear nonlocal monotone operators in fractional divergence form. 
More precisely, we identify a class of monotone maps that is closed under $H$-convergence, in the spirit of~\cite{CPDMDF,Tartar}.
We consider sequences of maps $A_h\colon\mathbb{R}^n\times \mathbb{R}^n\to \mathbb{R}^n$ characterised by the following lower bound
\begin{align*}
    (A_h(x,\xi) - A_h(x,\eta)) \cdot (\xi - \eta) &\ge \lambda |\xi - \eta|^p
\end{align*}
together with the growth condition, introduced following~\cite{CPDMDF},
\begin{align*}
    |A_h(x,\xi) - A_h(x,\eta)| &\le \Lambda \left[1 + A_h(x,\xi)\cdot\xi + A_h(x,\eta)\cdot\eta\right]^{\frac{p-2}{p}}\left[(A_h(x,\xi) - A_h(x,\eta)) \cdot (\xi - \eta)\right]^{\frac{1}{p}}.
\end{align*}
In the linear case $A_h(x,\xi)=B_h(x)\xi$, these assumptions are consistent with those considered in~\cite{CCM25}, and in the nonlinear Hilbertian setting $p=2$ they reduce to the bounds introduced in~\cite[Chapter~11]{Tartar}.
Instead, for $p>2$, they differ from the commonly adopted assumptions, such as those in~\cite{DASC}, under which the $H$-limits belong only to a larger class of monotone maps. 
The nonlinear nonlocal Dirichlet problems considered in this paper are of fractional divergence type and take the form
\begin{equation*}
\begin{cases}
-\div^s(A_h(x,\nabla^s u_h(x))) = f(x) & \text{in } \Omega, \\
u_h(x) = 0 & \text{in } \mathbb{R}^n \setminus \Omega.
\end{cases}
\end{equation*}
Unlike the linear case treated in~\cite{bbcejw}, we still assume the maps $A_h$ to be prescribed outside $\Omega$, as in~\cite{CCM25}.
Indeed, in the linear setting, for every weakly converging sequence $u_h$ in $H^{s,p}_0(\Omega)$ \cite[Lemma~2.12]{KrSc22} ensures the strong convergence of the fractional gradients $\nabla^s u_h$ outside $\Omega$.
This allows one to pass to the limit in the products $B_h\nabla^s u_h$ under the weak* convergence of the matrices $B_h$.
In contrast, in the nonlinear case, the weak* convergence of the maps $A_h$ is not sufficient to pass to the limit in the nonlinear compositions $A_h(\,\cdot,\,\nabla^s u_h)$, even in the presence of strong convergence of the fractional gradients.
The problem of removing the exterior prescription therefore remains open.

The first main result of this paper is Theorem~\ref{thm:Hs}, which establishes the equivalence between local and nonlocal $H$-convergence, the latter in the sense of Definition~\ref{def:nonlocal_H-convergence}, for the sequences of monotone operators
\[
\mathcal{A}^s_h\coloneqq-\div^s(A_h(\,\cdot\,,\nabla^s\,)),\quad \mathcal{A}_h\coloneqq-\div(A_h(\,\cdot\,,\nabla\,)),
\]
respectively associated with the local and nonlocal Dirichlet problems
\begin{equation*}
\begin{cases}
-\div^s(A_h(x,\nabla^s u_h(x))) = f(x) & \text{in } \Omega, \\
u_h(x) = 0 & \text{in } \mathbb{R}^n \setminus \Omega,
\end{cases}\quad 
\begin{cases}
-\div(A_h(x,\nabla w_h(x))) = f(x) & \text{in $\Omega$}, \\
w_h(x) = 0 & \text{on $\partial\Omega$}.
\end{cases}
\end{equation*}
As a consequence, the $H$-compactness of the class of monotone maps in the local setting, recalled in Proposition~\ref{prop:Hcomploc}, immediately yields the nonlocal $H$-compactness result, stated in Corollary~\ref{cor:nonlocal_H-compactness}.

From a technical viewpoint, the proof of Theorem~\ref{thm:Hs} is inspired by some ideas developed in~\cite{CuKrSc23,KrSc22}, which exploit the relation between local and fractional gradients; see Proposition~\ref{prop:equiv}.
This allows one to transfer local $H$-convergence results to the nonlocal setting and to construct a nonlocal $H$-limit candidate.
The main difficulty lies in handling mixed local–nonlocal terms involving both $\nabla^s u_h$ and $\nabla w_h$, and in passing to the limit. 
To overcome this issue, we rely on an integration-by-parts formula together with compactness arguments. Similar local-to-nonlocal arguments also appear in the study of div-curl formulas in the fractional setting; see~\cite[Section~3.2]{Caponi26}. 

The proof of Theorem~\ref{thm:Hs} is completed by establishing a uniqueness result for nonlocal $H$-limits (Lemma~\ref{lem:uniqueness}), analogous to the classical local one.
As a consequence, the nonlocal $H$-limit candidate obtained from the local theory is uniquely determined. We emphasise that Lemma~\ref{lem:uniqueness} is not entirely straightforward.
In fact, in the local setting, uniqueness is obtained through localization techniques, which do not naturally extend to the nonlocal framework. 
Moreover, the classical construction of affine test functions $\varphi_{x_0,\xi}$ satisfying
\[
\nabla \varphi_{x_0,\xi} = \xi
\]
in a neighbourhood of a point $x_0\in\Omega$, for a prescribed vector $\xi\in\mathbb{R}^n$, becomes considerably more delicate in the nonlocal setting; see Remark~\ref{rem:uniqueness}.

In the second part of this work, we investigate the relation between $H$- and $\Gamma$-convergence in the nonlinear nonlocal setting. 
From a historical perspective, 
this connection already emerged in the pioneering work of De Giorgi and Spagnolo~\cite{DGS}, where the energy functionals play a central role in the analysis of homogenisation problems.
As observed in~\cite{degiorgi2}, these ideas anticipated several aspects of the theory of $\Gamma$-convergence, which was introduced by De Giorgi and Franzoni in~\cite{degiorgi-fran} a few years later as a unified variational theory.
The link between $H$- and $\Gamma$-convergence lies in the fact that each PDE in~\eqref{eq:Lh}, associated with \emph{symmetric} matrices, can be interpreted as the Euler-Lagrange equation of the energy functional
\begin{equation*}
\frac{1}{2}\int_{\Omega}B_h(x)\nabla w_h(x)\cdot \nabla w_h(x)\, \d x-\int_\Omega g(x) w_h(x)\,\d x.
\end{equation*}
Under suitable assumptions, $\Gamma$-convergence implies convergence of minimisers, that is solutions to~\eqref{eq:Lh}; conversely, $H$-convergence implies $\Gamma$-convergence.
We refer the interested reader to~\cite[Theorem~13.12]{DalMaso} for a proof of the equivalence between $H$- and $\Gamma$-convergence in the linear local case, and to~\cite[Theorem~5.1]{CCM25} for the recent extension to the linear nonlocal case.

In the nonlinear setting, the situation becomes more delicate.
Indeed, one first needs to identify a natural subclass of monotone maps that can be associated with energy functionals, in the same way that symmetric matrices do in the linear theory.
For this reason, in Definition~\ref{def:class_conservative_fields} we introduce the class of \emph{conservative} monotone maps $A_h\colon\mathbb{R}^n\times\mathbb{R}^n\to\mathbb{R}^n$, i.e., monotone maps satisfying the former growth conditions and for which there exist potentials $\varphi_h\colon \mathbb{R}^n\times\R^n\to\R$ such that
\[
A_h(x,\xi)=\nabla_\xi \varphi_h(x,\xi).
\]
We emphasise that, in the linear case, conservativity is equivalent to symmetry; see Remark~\ref{rem:conservative_symmetric}.
We then study the $\Gamma$-convergence of the nonlocal energies defined as
\begin{equation*}
F^s_h(u)\coloneqq 
\begin{cases}
\displaystyle \int_{\R^n}\varphi_h(x,\nabla^s u(x))\, \d x&\quad\text{if $u\in H^{s,p}_0(\Omega)$},\\
\infty&\quad\text{if $u\in L^p(\R^n)\setminus H^{s,p}_0(\Omega)$}.
\end{cases}
\end{equation*}
In analogy with the previous result for the $H$-convergence, in Theorem~\ref{thm:Gammas} we show that the $\Gamma$-convergence of $F^s_h$ is equivalent to the $\Gamma$-convergence of the associated local energies 
\begin{equation*}
F_h(u)\coloneqq 
\begin{cases}
\displaystyle \int_\Omega\varphi_h(x,\nabla u(x))\,\d x & \quad \text{if } u\in W_0^{1,p}(\Omega),\\
\infty & \quad \text{if $u\in L^p(\Omega)\setminus W_0^{1,p}(\Omega)$}.
\end{cases}
\end{equation*}
As a by-product, we obtain the $\Gamma$-compactness of the class of nonlocal energies associated with conservative monotone operators; see Corollary~\ref{cor:Gamma-compactness}.

We emphasise that Theorem~\ref{thm:Gammas} is not straightforward. 
The implication \emph{from local to nonlocal $\Gamma$-convergence} (Proposition~\ref{prop:gamma1}) is inspired by~\cite{CuKrSc23}, where an analogous result was established for nonlocal energies associated with fractional gradients of finite interaction horizon.
However, since that framework differs substantially from the present one, where the Riesz fractional gradient involves long-range interactions, our result cannot be deduced directly from~\cite{CuKrSc23}. 
Our proof is based on revisiting the argument of~\cite{CuKrSc23} and combining it with the techniques developed in~\cite{KrSc22} for the Riesz fractional gradient in the context of relaxation. We refer the interested reader to~\cite{ABSS,ACFS,arroyo,BCM21} for other results on $\Gamma$-convergence of functionals involving Riesz fractional gradients. In addition, the proof of the converse implication \emph{from nonlocal to local $\Gamma$-convergence} (Proposition~\ref{prop:gamma2}) is based on the local $\Gamma$-compactness result stated in Corollary~\ref{cor:local_Gamma-compactness}, which can be deduced from the results in~\cite{ADMZ}, and a uniqueness result for the integral representation of nonlocal functionals, which is established in Proposition~\ref{prop:uniq_rapp}, as a consequence of Lemma~\ref{lem:uniq_rappr}. 
To the best of our knowledge, this uniqueness result does not appear in the literature, even in the classical local setting.

In Theorem~\ref{thm:Hs-Gammas}, we finally show the equivalence between the $H$-convergence of nonlocal conservative monotone operators and the $\Gamma$-convergence of the associated energy functionals.
The proof combines Theorem~\ref{thm:Hs}, Theorem~\ref{thm:Gammas}, and Proposition~\ref{prop:H-Gamma}.
The latter is a consequence of some results in~\cite{ADMZ}, and shows this equivalence in the local setting. 
As a further consequence, the class of conservative monotone operators is also closed under $H$-convergence.

Summarising, the results of this paper show that both $H$- and $\Gamma$-convergence in the nonlinear nonlocal setting can be entirely characterised by their local counterparts, and vice versa.
In the spirit of~\cite{CCM25}, such relations are summarised in the following commutative diagram.

\begin{figure}[ht]
\centering

\begin{tikzpicture}[
    scale=0.85,
    transform shape,
    >=stealth,
    thick,
    box/.style={
        rectangle,
        draw,
        rounded corners,
        fill=gray!20,
        text width=4cm,
        align=center,
        minimum height=1.2cm
    }
]

\node (HLoc) [box]
{$H$-convergence of local operators};

\node (HNonloc) [box, right=4.5cm of HLoc]
{$H$-convergence of nonlocal operators};

\node (GammaLoc) [box, below=3cm of HLoc]
{$\Gamma$-convergence of local energies};

\node (GammaNonloc) [box, below=3cm of HNonloc]
{$\Gamma$-convergence of nonlocal energies};

\draw[double distance=1pt,<->]
(HLoc) -- node[above]
{Thm.~\ref{thm:Hs}}
(HNonloc);

\draw[double distance=1pt,<->]
(GammaLoc) -- node[below]
{Thm.~\ref{thm:Gammas}}
(GammaNonloc);

\draw[double distance=1pt,<->]
(HLoc) -- node[left, align=center]
{Prop.~\ref{prop:H-Gamma}\\ 
\footnotesize (conservative maps)}
(GammaLoc);

\draw[double distance=1pt,<->]
(HNonloc) -- node[right, align=center]
{Thm.~\ref{thm:Hs-Gammas}\\ 
\footnotesize (conservative maps)}
(GammaNonloc);

\end{tikzpicture}

\caption*{Relations between $H$- and $\Gamma$-convergence in the local and nonlocal settings.}
\end{figure}
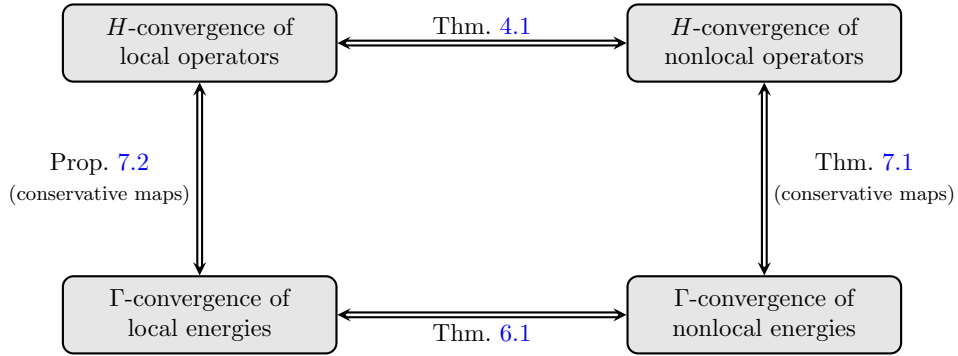

The paper is structured as follows. Section~\ref{sec_2} introduces the nonlocal functional framework. 
In Section~\ref{sec:H-def}, we discuss in detail the class of monotone maps under consideration, comparing it with the most commonly used classes in the literature. 
We also extend the notion of $H$-convergence to the nonlinear nonlocal setting.
Section~\ref{sec:H-equiv} is devoted to the proof of Theorem~\ref{thm:Hs} and the resulting $H$-compactness of the class of nonlocal monotone operators, shown in Corollary~\ref{cor:nonlocal_H-compactness}.
In Section~\ref{sec_4}, we address the second main topic of the paper, namely the $\Gamma$-convergence of nonlocal energy functionals associated with conservative monotone maps.
Finally, Section~\ref{sec:Gamm-equiv} and Section~\ref{sect:equivGammaH} are respectively devoted to the proofs of Theorem~\ref{thm:Gammas} and Theorem~\ref{thm:Hs-Gammas}.


\section{Notation and preliminaries}\label{sec_2}

Let us introduce the notation considered in this paper. 
The set of positive natural numbers is denoted by $\N$, while $\N_0$ denotes the set of natural numbers including zero. 
For $x\in \R^n$ and $r>0$, we let $B_r(x)\subset \R^n$ denote the open ball of radius $r$ centered at $x$. 
Given a measurable set $E \subset\R^n$, we denote by $\mathbf{1}_E\colon\R^n\to\{0,1\}$ its characteristic function. 
Given two open sets $U, V \subset \R^n$, we write $V \subset\subset U$ if there exists a compact set $K$ such that $V \subset K \subset U$. 
Given an open set $U\subset\R^n$, the collection of all open subset $V\subset U$ is denoted by $\mathscr{A}(U)$. 
We use standard notation for Lebesgue and Sobolev spaces. For $p \in [1,\infty]$, we let $p'$ denote the Hölder's conjugate exponent of $p$, namely $p' \coloneqq \frac{p}{p-1}$.



Let $n\ge 1$ be a natural number and let $s\in(0,1)$.
Let us recall the notions of \emph{Riesz $s$-fractional gradient} and \emph{Riesz $s$-fractional divergence}, introduced in~\cite{ComiStefani19,ShiehSpector15,Silhavy20}. 
Following~\cite{Silhavy20}, we first introduce the space
\begin{align*}
\mathcal{T}(\R^n)\coloneqq\left\{\psi \in C^\infty(\R^n):\text{$\int_{\R^n} |\nabla^k\psi(x)|\,\d x < \infty$, $\nabla^k \psi(x) \to 0$ as $|x|\to\infty$ for every $k \in \N_0$}\right\}.
\end{align*}
Note that
\[
C_c^\infty(\R^n) \subset \T(\R^n)
\subset L^1(\R^n)\cap L^\infty(\R^n).
\]

\begin{definition}
Let $s\in (0,1)$. For $\psi \in \T(\R^n)$ and $\Psi \in \T(\R^n;\R^n)$ the \emph{Riesz $s$-fractional gradient}
$\nabla^s\psi \colon \R^n \to \R^n$ and the
\emph{Riesz $s$-fractional divergence}
$\div^s\Psi \colon \R^n \to \R$ are defined, respectively, by
\begin{align*}
\nabla^s \psi(x) &\coloneqq \mu_{n,s} \int_{\R^n} \frac{(\psi(y)-\psi(x))(y-x)}{|y-x|^{n+s+1}} \,\d y\quad\text{for $x \in \R^n$},\\
\div^s \Psi(x) &\coloneqq \mu_{n,s} \int_{\R^n} \frac{(\Psi(y)-\Psi(x))\cdot (y-x)}{|y-x|^{n+s+1}} \,\d y\quad\text{for $x \in \R^n$}.
\end{align*}
The positive normalization constant $\mu_{n,s}$, whose explicit value only plays a minor role in this paper, depends only on $n$ and $s$. 
We refer to~\cite[Eq.~(4)]{Silhavy20} for its precise definition.  
\end{definition}

The nonlocal operators $\nabla^s$ and $\div^s$ are well defined, in the sense that the above integrals converge for every $x \in \R^n$. 
Moreover, as observed in~\cite[Proposition~5.2]{Silhavy20}, 
\[
\nabla^s \psi \in \T(\R^n;\R^n)\quad\text{and}\quad\div^s \Psi \in \T(\R^n).
\]
We stress that, even if $\psi \in C_c^\infty(\R^n)$ and $\Psi \in C_c^\infty(\R^n;\R^n)$, the functions $\nabla^s \psi$ and $\div^s \Psi$ do not have compact support.

The fractional gradient and divergence are closely related to the classical fractional Laplacian, which we recall below.

\begin{definition}
Let $s\in (0,1)$. For $\psi \in \T(\R^n)$, the {\em $\frac{s}{2}$-fractional Laplacian} $(-\Delta)^\frac{s}{2}\psi\colon\R^n\to\R$ is defined by
\begin{equation*}
(-\Delta)^\frac{s}{2}\psi(x)\coloneqq \displaystyle c_{n,s}\int_{\mathbb{R}^n}\frac{\psi(x)-\psi(y)}{|x-y|^{n+s}}\,{\rm d}y\quad\text{for $x\in\R^n$},
\end{equation*}
where $c_{n,s}$ is a positive normalization constant depending only on $n$ and $s$; see~\cite{Silhavy20} for its explicit definition (note that $c_{n,s}$ coincides with $-\nu_{n,s}$ in Eq.~(6) therein).  
\end{definition}

The fractional Laplacian $(-\Delta)^\frac{s}{2}\psi$ is well defined, in the sense that the above integral converges for every $x \in \R^n$. 
Moreover, as observed in~\cite[Proposition~5.2]{Silhavy20}, we have
\[
(-\Delta)^\frac{s}{2}\psi \in \T(\R^n).
\]
By~\cite[Theorem~5.3]{Silhavy20}, the following relation holds for every $\psi \in \T(\mathbb{R}^n)$ and $s,r\in (0,1)$
\[
-\operatorname{div}^s\nabla^r\psi=(-\Delta)^{\frac{s+r}{2}}\psi\quad\text{in $\mathbb{R}^n$}.
\]

Let us introduce the functional framework.

\begin{definition}
Let $\Omega \subset \R^n$ be an open set, $s \in (0,1)$, and $p \in [1,\infty)$. We define the \emph{Riesz fractional Sobolev space} $H_0^{s,p}(\Omega)$ as the closure of  $C_c^\infty(\Omega)$ with respect to the norm
\[
\|\cdot\|_{H_0^{s,p}(\Omega)}\coloneqq\left(\|\cdot\|_{L^p(\Omega)}^p+\|\nabla^s\cdot\|_{L^p(\R^n)}^p\right)^\frac{1}{p}.
\]
We denote by $H^{-s,p'}(\Omega)$ the dual space of $H_0^{s,p}(\Omega)$, that is $H^{-s,p'}(\Omega)\coloneqq (H^{s,p}_0(\Omega))'$. When $\Omega = \R^n$, we simply write $H^{s,p}(\R^n) \coloneqq H_0^{s,p}(\R^n)$.
\end{definition}

The space $H_0^{s,p}(\Omega)$ endowed with the norm $\|\,\cdot\,\|_{H_0^{s,p}(\Omega)}$ is a separable Banach space, and it is reflexive when $p \in (1,\infty)$. Moreover, when $\Omega = \R^n$, it coincides with the fractional Sobolev space $S^{s,p}(\R^n)$ introduced in~\cite[Definition~2.5]{KrSc22}; see~\cite[Theorem~2.7]{KrSc22}. 
Furthermore, as observed in~\cite[Definition~2.9]{KrSc22}, a sequence $(u_j)_j$ converges weakly in $H_0^{s,p}(\Omega)$ if and only if $(u_j)_j$ converges weakly in $L^p(\Omega)$ and $(\nabla^s u_j)_j$ converges weakly in $L^p(\R^n;\R^n)$.

The following result provides a useful connection between the Sobolev spaces $W^{1,p}_{\loc}(\R^n)$ and $H^{s,p}(\R^n)$. For the proof, we refer the interested reader to~\cite[Proposition~3.1]{KrSc22}.

\begin{proposition}\label{prop:equiv}
Let $s\in (0,1)$ and $p\in [1,\infty)$. Then the following two statements hold.
\begin{itemize}
\item[(i)] For every $u \in H^{s,p}(\R^n)$, there exists $v \in W^{1,p}_{\rm loc}(\R^n)$ such that
\[
\nabla v = \nabla^s u\quad\text{in $\R^n$}.
\]
\item[(ii)] For every $v \in W^{1,p}(\R^n)$, the function $u\coloneqq(-\Delta)^\frac{1-s}{2}v\in H^{s,p}(\R^n)$ satisfies
\[
\nabla^s u = \nabla v\quad\text{in $\R^n$}
\] 
and
\begin{equation}\label{eq:frac-lap-est}
\|u\|_{L^p(\R^n)} \le C \|v\|_{L^p(\R^n)}^s  \|\nabla v\|_{L^p(\R^n)}^{1-s},
\end{equation}
for a constant $C>0$ depending only on $n$ and $s$.
\end{itemize}
\end{proposition}

\begin{remark}
When $p<\frac{n}{1-s}$, the function $v$ appearing in Proposition~\ref{prop:equiv}(i) is given by
\[
v = I_{1-s}u,
\]
where $I_{1-s}$ denotes the Riesz potential operator, i.e.\ the inverse of the fractional Laplacian $(-\Delta)^{\frac{1-s}{2}}$. In contrast, for $p \ge \frac{n}{1-s}$, the Riesz potential may fail to be well defined, and thus $v$ is constructed via an approximation procedure; see~\cite[Proposition~3.1]{KrSc22} for further details.
\end{remark}

In order to deal with the nonlocal elliptic PDEs and functionals considered in this paper, we will frequently use Proposition~\ref{prop:equiv}. In particular, an important feature that we need is that weak convergence is preserved when passing from $H^{s,p}(\R^n)$ to $W^{1,p}_{\mathrm{loc}}(\R^n)$. This can be derived by slightly modifying the function $v$ in Proposition~\ref{prop:equiv}(i). 
We have the following result.

\begin{corollary}\label{coro:equiv}
Let $\Omega \subset \R^n$ be a bounded open set, $s \in (0,1)$, and $p \in [1,\infty)$. Let $u_j \in H^{s,p}(\R^n)$ for $j\in\N\cup\{\infty\}$ satisfy
\[
u_j \to u_\infty \quad \text{weakly in $H^{s,p}(\R^n)$ as $j \to \infty$}.
\]
Then, there exist a (not relabeled) subsequence and functions $w_j\in W^{1,p}(\Omega)$ for $j\in\N\cup\{\infty\}$ satisfying
\[
\nabla w_j = \nabla^s u_j \quad \text{in $\Omega$ for every $j \in \N\cup\{\infty\}$},
\]
and
\[
w_j \to w_\infty \quad \text{weakly in $W^{1,p}(\Omega)$ and strongly in $L^p(\Omega)$ as $j \to \infty$}.
\]
\end{corollary}

\begin{proof}
Let $(v_j)_j\in W^{1,p}_{\rm loc}(\R^n)$ and $v_\infty\in W^{1,p}_{\rm loc}(\R^n)$ be the functions given in Proposition~\ref{prop:equiv}. 
Consider a ball $B$ containing $\Omega$.
By Poincaré's inequality, there exists a constant $C$, depending only on $n$, $p$, and $B$, and a sequence $(c_j)_j\subset\R$ satisfying
\[
\|v_j-c_j\|_{L^p(B)}\le C\|\nabla v_j\|_{L^p(B)}=C\|\nabla^s u_j\|_{L^p(B)}\le C\quad\text{for every $j\in\N$}.
\]
By Rellich's theorem there exist a (not relabeled) subsequence and a function $w_\infty\in W^{1,p}(B)$ such that 
\[
w_j\coloneqq v_j-c_j\to w_\infty\quad\text{weakly in $W^{1,p}(B)$ and strongly in $L^p(B)$ as $j\to\infty$}.
\]
Note that for every $\Psi\in C_c^\infty (B;\R^n)$ we have
\begin{align*}
\int_B \nabla w_\infty(x)\cdot\Psi(x)\, \d x&=\lim_{j\to\infty}\int_B \nabla w_j(x)\cdot\Psi(x)\, \d x=\lim_{j\to\infty}\int_B \nabla v_j(x)\cdot\Psi(x)\, \d x\\
&=\lim_{j\to\infty}\int_B \nabla^s u_j(x)\cdot\Psi(x)\, \d x=\int_B \nabla^s u_\infty(x)\cdot\Psi(x)\, \d x.
\end{align*}
Hence, $\nabla w_\infty=\nabla^s u_\infty$ in $B$ and, in particular, in $\Omega$. 
\end{proof}

We conclude this section by recalling several results in the fractional setting that will be used throughout this paper, namely a Leibniz-type rule for the fractional gradient~\cite[Lemma~2.11 and Eq.~(2.11)]{KrSc22}, a Poincaré-type inequality~\cite[Theorem~3.3]{ShiehSpector15} and~\cite[Theorem~2.2]{BCM20}, and a Rellich-type theorem~\cite[Theorem~2.2]{ShiehSpector18} and~\cite[Theorem~2.3]{BCM20}.

\begin{proposition}[Leibniz rule]\label{prop:Leibniz}
Let $\Omega\subset\R^n$ be an open set, $s\in (0,1)$, and $p\in [1,\infty)$. Let $\varphi\in C_c^1(\Omega)$ and $u\in H^{s,p}(\R^n)$. Then $\varphi u\in H^{s,p}_0(\Omega)$ and
\[
\nabla^s(\varphi u)=\varphi \nabla^s u +u \nabla^s\varphi +\nabla^s_{\rm NL}(\varphi,u)\quad\text{in $\R^n$},
\]
where the nonlocal remainder term $\nabla^s_{\rm NL}(\varphi,u)\colon\R^n\to\R^n$ is defined by
\[
\nabla^s_{\rm NL}(\varphi,u)(x)\coloneqq\mu_{n,s}\int_{\R^n}\frac{(\varphi(x)-\varphi(y))(u(x)-u(y))(y-x)}{|y-x|^{n+s+1}}\,\d y\quad\text{for a.e.\ $x\in\R^n$}.
\] 
Moreover, there exists a constant $C>0$, depending only on $n$ and $s$, such that
\begin{equation}\label{eq:nabla_NL}
\|\nabla^s_{\rm NL}(\varphi,u)\|_{L^p(\R^n)}\le C\|u\|_{L^p(\R^n)}\|\varphi\|_{L^\infty(\Omega)}^{1-s}\|\nabla \varphi\|_{L^\infty(\Omega)}^s.
\end{equation}
\end{proposition}

\begin{proposition}[Poincaré inequality]\label{prop:poincare}
Let $\Omega\subset\R^n$ be a bounded open set, $s\in (0,1)$, and $p\in (1,\infty)$. Then there exists a constant $C$, depending only on $n$, $s$, $p$, and $\Omega$, such that 
\[
\|u\|_{L^p(\Omega)}\le C\|\nabla^s u\|_{L^p(\R^n)}\quad\text{for every $u\in H^{s,p}(\R^n)$}.
\]
\end{proposition}

\begin{proposition}[Rellich theorem]\label{prop:rellich}
Let $\Omega\subset\R^n$ be a bounded open set, $s\in (0,1)$, and let $p\in (1,\infty)$. Then the space $H^{s,p}_0(\Omega)$ is compactly embedded into $L^p(\Omega)$.
\end{proposition}

As a consequence of the previous three propositions, one can deduce that if $u_j\in H^{s,p}_0(\Omega)$ for $j\in\N\cup\{\infty\}$ and $u_j \to u_\infty$ weakly in $H^{s,p}_0(\Omega)$ as $j\to\infty$, then the weak convergence of the fractional gradients improves to strong convergence outside $\Omega$. More precisely, we have the following result; see~\cite[Lemma~2.12]{KrSc22} for a proof.

\begin{proposition}\label{prop:strong-gradient}
Let $\Omega\subset\R^n$ be a bounded open set, $s\in (0,1)$, and $p\in (1,\infty)$. Let $u_j \in H^{s,p}_0(\Omega)$ for $j\in\N\cup\{\infty\}$ satisfy
\[
u_j\to u_\infty\quad\text{weakly in $H^{s,p}_0(\Omega)$ as $j\to\infty$}.
\]
Then
\[
\nabla^s u_j \to \nabla^s u_\infty \quad \text{strongly in $L^p_{\rm loc}(\R^n\setminus\overline\Omega; \R^n)$ as $j\to\infty$}.
\]
\end{proposition}


\section{The \texorpdfstring{$H$}{H}-convergence problem}\label{sec:H-def}

Let us introduce the local and nonlocal $H$‑convergence problems involving nonlocal monotone operators. 
From now on, and throughout the rest of the paper, we assume that: 
\begin{equation*}
s\in (0,1),\qquad p \in [2,\infty),\qquad \text{$0<\lambda \le \Lambda$ are two fixed constants}.
\end{equation*}
We start by defining the following classes of monotone maps. 

\begin{definition}\label{def:classes}
Let $\Omega \subset \R^n$ be an open set. Given a Carathéodory function $A \colon \Omega \times \R^n \to \R^n$, we introduce the following conditions for a.e.\ $x \in \Omega$ and for every $\xi,\eta \in \R^n$:
\begin{itemize}
\item[(M1)] $A(x,0) = 0$,
\item[(M2)] $(A(x,\xi) - A(x,\eta)) \cdot (\xi - \eta) \ge \lambda |\xi - \eta|^p$,
\item[(M3)] 
$|A(x,\xi) - A(x,\eta)|
\le \Lambda \left[1 + A(x,\xi)\cdot\xi + A(x,\eta)\cdot\eta\right]^{\frac{p-2}{p}}
\left[(A(x,\xi) - A(x,\eta)) \cdot (\xi - \eta)\right]^{\frac{1}{p}}$,
\item[(M3)$'$]
$|A(x,\xi) - A(x,\eta)|
\le \Lambda \left[1 + |\xi|^p + |\eta|^p\right]^{\frac{p-2}{p}} |\xi - \eta|$,
\item[(M3)$''$]
$|A(x,\xi) - A(x,\eta)|
\le \Lambda \left[1 + |\xi|^p + |\eta|^p\right]^{\frac{p-2}{p-1}} |\xi - \eta|^{\frac{1}{p-1}}$.
\end{itemize}
We define the following classes of maps:
\begin{align*}
\M_\Omega(p,\lambda,\Lambda)
&\coloneqq
\{A \colon \Omega \times \R^n \to \R^n\text{ Carathéodory} : A \text{ satisfies (M1), (M2), (M3)}\},\\
\M_\Omega'(p,\lambda,\Lambda)
&\coloneqq
\{A \colon \Omega \times \R^n \to \R^n\text{ Carathéodory} : A \text{ satisfies (M1), (M2), (M3)$'$}\},\\
\M_\Omega''(p,\lambda,\Lambda)
&\coloneqq
\{A \colon \Omega \times \R^n \to \R^n\text{ Carathéodory} : A \text{ satisfies (M1), (M2), (M3)$''$}\}.
\end{align*}
\end{definition}

\begin{remark}
In Definition~\ref{def:classes}, condition (M2) ensures monotonicity of the function $A$, and together with (M1), implies coercivity.
Finally, conditions (M3), (M3)$'$, and (M3)$''$ control the continuity of $A$ with respect to the second variable, and correspond to different regularity regimes.
\end{remark}

In the next result, we clarify the relations between $\M_\Omega(p,\lambda,\Lambda)$, $\M_\Omega'(p,\lambda,\Lambda)$, and $\M_\Omega''(p,\lambda,\Lambda)$.

\begin{lemma}\label{lem:classes}
Let $\Omega \subset \R^n$ be an open set. The following inclusions hold
\[
\M_\Omega'(p,\lambda,\Lambda')\subset\M_\Omega(p,\lambda,\Lambda)\subset\M_\Omega''(p,\lambda,\Lambda''),
\]
where
\begin{equation}\label{eq:lambda}
\Lambda' \coloneqq \min\{1,\lambda\}^{\frac{p-2}{p}}\lambda^{\frac1p}\Lambda,
\qquad
\Lambda'' \coloneqq \max\{3,2^{p-2}\Lambda^p\}^{\frac{p-2}{p-1}}\Lambda^{\frac{p}{p-1}}.
\end{equation}
\end{lemma}

\begin{proof}
We start by proving the first inclusion. Let $A \in \M_\Omega ' (p,\lambda, \Lambda')$ be fixed. 
By (M1) and (M2), it holds that
\begin{gather*}
|\xi|^p \le \lambda^{-1} A(x,\xi)\cdot \xi, \quad
|\eta|^p \le \lambda^{-1} A(x,\eta)\cdot \eta, \quad
|\xi-\eta|
\le \lambda^{-\frac{1}{p}}
\big[(A(x,\xi)-A(x,\eta))\cdot(\xi-\eta)\big]^{\frac{1}{p}},
\end{gather*}
a.e.\ $x \in \Omega$ and for every $\xi,\eta\in\R^n$, which yields by (M3)$'$
\[
|A(x,\xi)-A(x,\eta)|
\le \lambda^{-\frac{1}{p}}\Lambda'
(1+\lambda^{-1}A(x,\xi)\cdot\xi+\lambda^{-1}A(x,\eta)\cdot\eta)^{\frac{p-2}{p}}
[(A(x,\xi)-A(x,\eta))\cdot(\xi-\eta)]^{\frac{1}{p}}.
\]
Hence, we conclude that
\[
A \in \M_\Omega \left(p,\lambda, \max\{1, \lambda^{-1}\}^{\frac{p-2}{p}}\lambda^{-\frac{1}{p}} \Lambda'\right) = \M_\Omega (p,\lambda, \Lambda)
\]
for the corresponding choice of $\Lambda'$ in~\eqref{eq:lambda}.

Let us show the second inclusion. 
Let $A \in \M_\Omega (p,\lambda,\Lambda)$ be fixed.
By (M1), (M3) and Cauchy-Schwarz inequality, we have
\begin{equation*}
A(x,\xi) \cdot \xi \le |A(x,\xi)| |\xi| \le \Lambda [1 + A(x,\xi) \cdot \xi]^{\frac{p-2}{p}} (A(x,\xi) \cdot \xi)^{\frac{1}{p}} |\xi|
\end{equation*}
a.e.\ $x \in \Omega$ and for every $\xi \in \R^n$.
We distinguish two cases. On the one hand, if $x \in \Omega$ and $\xi \in \R^n$ are such that $A(x,\xi) \cdot \xi \ge 1$, then the previous inequality gives
\[
(A(x,\xi) \cdot \xi)^{\frac{p-1}{p}} \le 2^{\frac{p-2}{p}} \Lambda (A(x,\xi) \cdot \xi)^{\frac{p-2}{p}} |\xi|,
\]
which implies
\begin{equation}\label{eq:A2-2}
A(x,\xi) \cdot \xi \le 2^{p-2} \Lambda^p |\xi|^p \le 1 + 2^{p-2} \Lambda^p |\xi|^p.
\end{equation}
On the other hand, if $x \in \Omega$ and $\xi \in \R^n$ are such that $A(x,\xi) \cdot \xi < 1$, then
\begin{equation}\label{eq:A2-1}
A(x,\xi) \cdot \xi \le 1 + 2^{p-2} \Lambda^p |\xi|^p.
\end{equation}
By combining~\eqref{eq:A2-2}--\eqref{eq:A2-1}, we deduce
\begin{equation}\label{eq:A3}
A(x,\xi) \cdot \xi \le 1 + 2^{p-2} \Lambda^p |\xi|^p, \quad A(x,\eta) \cdot \eta \le 1 + 2^{p-2} \Lambda^p |\eta|^p
\end{equation}
for a.e.\ $x \in \Omega$ and for every $\xi,\eta \in \R^n$.
Then, by~\eqref{eq:A3} and (M3), it follows that
\begin{align*}
|A(x,\xi) - A(x,\eta)| 
& \le \Lambda  [3 + 2^{p-2} \Lambda^p |\xi|^p + 2^{p-2} \Lambda^p |\eta|^p]^{\frac{p-2}{p}} 
\big[(A(x,\xi) - A(x,\eta)) \cdot (\xi - \eta)\big]^{\frac{1}{p}} \\
& \le \max\{3, 2^{p-2} \Lambda^p\}^{\frac{p-2}{p}}  \Lambda  [1 + |\xi|^p + |\eta|^p]^{\frac{p-2}{p}} 
|A(x,\xi) - A(x,\eta)|^{\frac{1}{p}} |\xi - \eta|^{\frac{1}{p}},
\end{align*}
which gives
\begin{equation*}
|A(x,\xi) - A(x,\eta)| \le \max\{3, 2^{p-2} \Lambda^p\}^{\frac{p-2}{p-1}}  \Lambda^{\frac{p}{p-1}}  [1 + |\xi|^p + |\eta|^p]^{\frac{p-2}{p-1}} |\xi - \eta|^{\frac{1}{p-1}}.
\end{equation*}
Thus
\[
A \in \M_\Omega'' \left(p,\lambda,\max\{3, 2^{p-2} \Lambda^p\}^{\frac{p-2}{p-1}}\Lambda^{\frac{p}{p-1}}\right) = \M_\Omega'' (p,\lambda, \Lambda'')
\]
for the corresponding choice of $\Lambda''$ in~\eqref{eq:lambda}.
\end{proof}

\begin{remark}\label{rem:classes}
The following remarks are in order.
\begin{itemize}
\item[(a)] In the literature concerning local monotone operators, the class $\M_\Omega'(p,\lambda,\Lambda)$ is the one most commonly employed, as its definition emphasises uniform continuity properties; see, e.g.,~\cite{DASC}. 
However, this class is not closed under $H$‑convergence, since $H$-limits generally belong only to the larger class $\M_\Omega''(p,\lambda,\Lambda)$. 
Instead, the intermediate class $\M_\Omega(p,\lambda,\Lambda)$ is stable under $H$-convergence and is therefore the natural setting for the theory developed in this paper. 

\item[(b)] We stress that for every $\lambda>0$ there exists $\Lambda=\Lambda(\lambda)\ge\lambda$ sufficiently large such that the classes $\M_\Omega(p,\lambda,\Lambda)$, $\M_\Omega'(p,\lambda,\Lambda)$, and $\M_\Omega''(p,\lambda,\Lambda)$ are non-empty.
Indeed, the $p$-Laplace operator, namely
\[
A(\xi) \coloneqq |\xi|^{p-2}\xi
\quad \text{for every $\xi \in \R^n$},
\]
satisfies for every $\xi,\eta\in\R^n$
\begin{align*}
(A(\xi)-A(\eta))\cdot(\xi-\eta)&\ge c_p|\xi-\eta|^p,\\
|A(\xi)-A(\eta)|&\le C_p(|\xi|^p+|\eta|^p)^\frac{p-2}{p}|\xi-\eta|
\end{align*}
for suitable constants $c_p, C_p > 0$; see, e.g.,~\cite[Eq.~(2.2)]{simon}. 
\item[(c)] When $p = 2$, the class $\M_\Omega(p,\lambda,\Lambda)$ reduces to the one introduced in~\cite[Definition~11.1]{Tartar}.
In the linear case $A(x,\xi)=B(x)\xi$, assumptions (M1), (M2), (M3) coincide with those considered in~\cite{CCM25}.
\item[(d)] As a consequence of conditions (M1) and (M3)$''$, every function $A \in \M_\Omega''(p,\lambda,\Lambda)$ (and therefore every $A \in \M_\Omega(p,\lambda,\Lambda)$ and $A \in \M_\Omega'(p,\lambda,\Lambda)$) satisfies the growth condition
\begin{equation}\label{eq:A-est}
|A(x,\xi)| \le C(1 + |\xi|^{p-1})
\quad \text{for a.e.\ $x \in \Omega$ and every $\xi \in \R^n$},
\end{equation}
for some constant $C>0$ depending only on $p$, $\lambda$, and $\Lambda$.
\end{itemize}
\end{remark}

\subsection{Local \texorpdfstring{$H$}{H}-convergence} 

Let us recall the $H$-convergence problem for nonlinear elliptic PDEs driven by monotone operators in the local setting. Let $\Omega\subset\R^n$ be a bounded open set. Let $A\in\M_\Omega (p,\lambda,\Lambda)$. For every $f\in W^{-1,p'}(\Omega)$ we consider the following Dirichlet problem
\begin{equation}\label{eq:localproblem}
\begin{cases}
-\div(A(\,\cdot\,,\nabla u))=f&\text{in $\Omega$},\\
u=0&\text{on }\partial\Omega.
\end{cases}
\end{equation}

\begin{lemma}\label{lem:existence}
Let $\Omega \subset \R^n$ be an open set and $A\in \M_\Omega (p,\lambda,\Lambda)$.
For every $f\in W^{-1,p'}(\Omega)$ there exists a unique function $u\in W^{1,p}_0(\Omega)$ weak solution to problem~\eqref{eq:localproblem}, i.e.\ satisfying
\begin{equation*}
\int_\Omega A(x,\nabla u(x))\cdot \nabla v(x)\,\d x=\langle f,v\rangle_{W^{-1,p'}(\Omega)\times W^{1,p}_0(\Omega)}\quad\text{for every $v\in W^{1,p}_0(\Omega)$}.
\end{equation*}
\end{lemma}

\begin{proof}
We define the differential operator $\mathcal A\colon W^{1,p}_0(\Omega)\to W^{-1,p'}(\Omega)$ as 
\begin{equation*}
\langle\mathcal A(u),v\rangle_{W^{-1,p'}(\Omega)\times W^{1,p}_0(\Omega)}\coloneqq \int_\Omega A(x,\nabla u(x))\cdot \nabla v(x)\,\d  x\quad\text{for every $u,v\in W^{1,p}_0(\Omega)$}.
\end{equation*}
As a consequence of Poincaré's inequality and assumptions (M1)--(M3), we derive that $\mathcal A$ is strictly monotone, coercive, and continuous. Then, by Browder-Minty theorem, the operator $\mathcal A$ is a bijection and the conclusion follows.
\end{proof}

Let $\Omega\subset\R^n$ be a bounded open set, and let $A_h\in \M_\Omega (p,\lambda,\Lambda)$ for $h\in\N\cup\{\infty\}$. 
As observed in Lemma~\ref{lem:existence}, the operators $\mathcal A_h\colon W^{1,p}_0(\Omega)\to W^{-1,p'}(\Omega)$ defined by 
\begin{equation}\label{eq:Ah}
\langle\mathcal A_h(u),v\rangle_{W^{-1,p'}(\Omega)\times W^{1,p}_0(\Omega)}\coloneqq \int_\Omega A_h(x,\nabla u(x))\cdot \nabla v(x)\,\d  x\quad\text{for every $u,v\in W^{1,p}_0(\Omega)$}
\end{equation}
are invertible. 
We set
\begin{equation}\label{eq:Bh}
\mathcal B_h\coloneqq \mathcal A_h^{-1}\colon W^{-1,p'}(\Omega)\to W^{1,p}_0(\Omega),
\end{equation}
and define the operators $\mathcal M_h\colon W^{-1,p'}(\Omega)\to L^{p'}(\Omega;\R^n)$ by
\begin{equation}\label{eq:Mh}
\mathcal M_h(f)(x)\coloneqq A_h(x,\nabla \mathcal B_h(f)(x))\quad\text{for a.e.\ $x\in\Omega$}.
\end{equation}
Each operator $\mathcal M_h$ is well defined in view of the growth condition~\eqref{eq:A-est}, which ensures that
\[
A_h(\,\cdot\,,\nabla \mathcal B_h(f))\in L^{p'}(\Omega;\R^n).
\]

\begin{definition}\label{def:H-conv}
Let $\Omega\subset\R^n$ be a bounded open set and $A_h\in \M_\Omega (p,\lambda,\Lambda)$ for $h\in\N\cup\{\infty\}$. We say that 
\[
\text{$\mathcal A_h$ $H$-converges to $\mathcal A_\infty$ as $h\to\infty$},
\]
if for every $f\in W^{-1,p'}(\Omega)$ we have:
\begin{itemize}
\item[(H1)] $\mathcal B_h(f)\to \mathcal B_\infty(f)$ weakly in $W^{1,p}_0(\Omega)$ as $h\to\infty$, 
\item[(H2)] $\mathcal M_h(f)\to \mathcal M_\infty(f)$ weakly in $L^{p'}(\Omega;\R^n)$ as $h\to\infty$.
\end{itemize}
\end{definition}

\subsection{Nonlocal \texorpdfstring{$H$}{H}-convergence}\label{sect_nonlocal_H}

The extension of the $H$-convergence to the nonlocal setting requires taking into account the nonlocal nature of the fractional operators $\nabla^s$ and $\operatorname{div}^s$, which are defined on the whole space $\R^n$. 
For this reason, we consider maps $A$ defined on $\R^n\times\R^n$. 
The corresponding classes of monotone maps are introduced below.

\begin{definition}
Let $\Omega\subset\R^n$ be a bounded open set. Let $A_0\in \mathbb M_{\R^n}(p,\lambda,\Lambda)$ satisfy 
\begin{equation}\label{eq:A0}
|A_0(x,\xi)-A_0(x,\eta)|\le C\left[|\xi|^p + |\eta|^p\right]^{\frac{p-2}{p-1}} |\xi - \eta|^\frac{1}{p-1}\quad\text{for a.e.\ $x\in\R^n$ and every $\xi,\eta\in\R^n$},
\end{equation}
for some constant $C>0$. We set
\[
\M_\Omega(p,\lambda,\Lambda,A_0)\coloneqq\{A\in\M_{\R^n}(p,\lambda,\Lambda):A=A_0\text{ in $(\R^n\setminus\Omega)\times\R^n$}\}.
\]
\end{definition}

\begin{remark}\label{rem:A0}
The following remarks are in order.
\begin{itemize}
\item[(a)] The class $\M_\Omega(p,\lambda,\Lambda)$, which is used to treat the local case, consists of maps defined on $\Omega\times\R^n$, whereas $\M_\Omega(p,\lambda,\Lambda,A_0)$, which will be employed in the nonlocal setting, consists of maps defined on $\R^n\times\R^n$ that coincide with a prescribed map $A_0$ on $(\R^n\setminus\Omega)\times\R^n$.
\item[(b)]
The two classes also differ in their growth conditions outside $\Omega$. Indeed, for $A_0$ we require condition~\eqref{eq:A0}. 
As a consequence, taking into account Lemma~\ref{lem:classes}, one readily obtains that every map $A\in\M_\Omega(p,\lambda,\Lambda,A_0)$ satisfies
\begin{equation}\label{eq:A0-con}
|A(x,\xi)-A(x,\eta)|\le C\left[\mathbf{1}_\Omega(x) + |\xi|^p + |\eta|^p\right]^{\frac{p-2}{p-1}} |\xi - \eta|^\frac{1}{p-1}
\end{equation}
for a.e.\ $x\in\R^n$ and for every $\xi,\eta\in\R^n$, and for some constant $C>0$ depending only on $p$, $\lambda$, $\Lambda$, and $A_0$. 
Taking $\eta=0$ in~\eqref{eq:A0-con} and using (M1), we obtain
\begin{equation}\label{eq:A0-est}
|A(x,\xi)|\le C(\mathbf{1}_\Omega(x)+|\xi|^{p-1})\quad\text{for a.e.\ $x\in\R^n$ and for every $\xi\in\R^n$}
\end{equation}
for some constant $C>0$ depending only on $p$, $\lambda$, $\Lambda$, and $A_0$. This stronger condition is required in the nonlocal framework, where one has to deal with integrals of the form
\[
\int_{\R^n} A(x,\nabla^s u(x))\cdot \nabla^s v(x)\,\d x
\quad\text{for every $u,v\in H^{s,p}_0(\Omega)$},
\]
which are well defined provided that
\begin{equation}\label{eq:anablas}
A(\,\cdot\,,\nabla^s u)\in L^{p'}(\R^n;\R^n).
\end{equation}
Notice that if $\nabla^s u\in L^{p}(\R^n;\R^n)$,
then~\eqref{eq:A0-est} directly implies~\eqref{eq:anablas}.
\item[(c)] In view of Remark~\ref{rem:classes}(b), a natural choice of
$A_0\in \M_{\R^n}(p,\lambda,\Lambda)$ satisfying~\eqref{eq:A0} is
\[
A_0(\xi)\coloneqq |\xi|^{p-2}\xi \quad\text{for every $\xi\in\R^n$}.
\]
\end{itemize}
\end{remark}

We can now introduce the $H$-convergence problem in the nonlocal setting. 
Let $\Omega\subset\R^n$ be a bounded open set. 
Let $A_0\in \M_{\R^n}(p,\lambda,\Lambda)$ satisfy~\eqref{eq:A0} and let $A \in \M_\Omega (p,\lambda,\Lambda, A_0)$. 
For every $f\in H^{-s,p'}(\Omega)$ we consider the following Dirichlet problem
\begin{equation}\label{eq:nonlocalproblem}
\begin{cases}
-\div^s(A(\,\cdot\,,\nabla^s u))=f&\text{in }\Omega \\
u=0&\text{in }\R^n\setminus\Omega.
\end{cases}
\end{equation}
By proceeding along the same lines as in Lemma~\ref{lem:existence} and relying on Proposition~\ref{prop:poincare}, we derive the following result.

\begin{lemma}\label{lem:existencefrac}
Let $\Omega\subset\R^n$ be a bounded open set. Let $A_0\in \M_{\R^n}(p,\lambda,\Lambda)$ satisfy~\eqref{eq:A0} and let $A \in \M_\Omega (p,\lambda,\Lambda, A_0)$. 
For every $f\in H^{-s,p'}(\Omega)$ there exists a unique function $u\in H^{s,p}_0(\Omega)$ weak solution to problem~\eqref{eq:nonlocalproblem}, i.e.\ satisfying
\begin{equation*}
\int_{\R^n} A(x,\nabla^s u(x))\cdot \nabla^s v(x)\,\d  x=\langle f,v\rangle_{H^{-s,p'}(\Omega)\times H^{s,p}_0(\Omega)}\quad\text{for every $v\in H^{s,p}_0(\Omega)$}.
\end{equation*}
\end{lemma}

In the previous setting, let $A_h\in \M_\Omega (p,\lambda,\Lambda, A_0)$ for $h\in \N\cup\{\infty\}$. 
We define the operators $\mathcal A_h^s\colon H^{s,p}_0(\Omega)\to H^{-s,p'}(\Omega)$ by
\begin{equation}\label{eq:Ahs}
\langle\mathcal A_h^s(u),v\rangle_{H^{-s,p'}(\Omega)\times H^{s,p}_0(\Omega)}\coloneqq \int_{\R^n} A_h(x,\nabla^s u(x))\cdot \nabla^s v(x)\, \d x\quad\text{for every $u, v \in H^{s,p}_0 (\Omega)$}.
\end{equation}
By Lemma~\ref{lem:existencefrac}, each operator $\mathcal A_h^s$ is invertible, and we can consider
\begin{equation}\label{eq:Bhs}
\mathcal B_h^s\coloneqq (\mathcal A_h^s)^{-1}\colon H^{-s,p'}(\Omega)\to H^{s,p}_0(\Omega).
\end{equation}
We also define the operators $\mathcal M_h^s\colon H^{-s,p'}(\Omega)\to L^{p'}(\R^n;\R^n)$ as 
\begin{equation}\label{eq:Mhs}
\mathcal M_h^s(f)(x)\coloneqq A_h(x,\nabla^s \mathcal B_h^s(f)(x))\quad\text{for a.e.\ $x\in\R^n$},
\end{equation}
which are well defined in virtue of~\eqref{eq:A0-est}.

\begin{definition}\label{def:nonlocal_H-convergence}
Let $\Omega\subset\R^n$ be a bounded open set. 
Let $A_0\in \M_{\R^n}(p,\lambda,\Lambda)$ satisfy~\eqref{eq:A0} and let $A_h\in\M_\Omega (p,\lambda,\Lambda, A_0)$ for $h\in \N\cup\{\infty\}$. 
We say that 
\[
\text{$\mathcal A_h^s$ $H$-converges to $\mathcal A^s_\infty$ as $h\to\infty$}
\]
if for every $f\in H^{-s,p'}(\Omega)$ we have:
\begin{itemize}
\item[(NH1)] $\mathcal B_h^s(f)\to \mathcal B_\infty^s(f)$ weakly in $H^{s,p}_0(\Omega)$ as $h\to\infty$, 
\item[(NH2)] $\mathcal M_h^s(f)\to \mathcal M_\infty^s(f)$ weakly in $L^{p'}(\R^n;\R^n)$ as $h\to\infty$.
\end{itemize}
\end{definition}


\section{Equivalence between \texorpdfstring{$H$}{H}-convergence of local and nonlocal operators}\label{sec:H-equiv}

From now on, given $A_0 \in \M_{\R^n}(p,\lambda,\Lambda)$ satisfying~\eqref{eq:A0} and $A \in \M_\Omega(p,\lambda,\Lambda,A_0)$, we simply write $A \in \M_\Omega(p,\lambda,\Lambda)$ when referring only to the restriction of $A$ to $\Omega \times \R^n$. 
Conversely, if $A \in \M_\Omega(p,\lambda,\Lambda)$, we denote by $A \in \M_\Omega(p,\lambda,\Lambda,A_0)$ its extension to the whole $\R^n \times \R^n$ obtained by setting $A = A_0$ on $(\R^n \setminus \Omega)\times \R^n$.

The main result of this section is the equivalence between local and nonlocal $H$-convergence.

\begin{theorem}\label{thm:Hs}
Let $\Omega\subset\R^n$ be a bounded open set satisfying $|\partial\Omega|=0$. Let $A_0\in \M_{\R^n}(p,\lambda,\Lambda)$ satisfy~\eqref{eq:A0} and $A_h\in\M_\Omega (p,\lambda,\Lambda, A_0)$ for $h\in \N\cup\{\infty\}$. Then 
\[
\text{$\mathcal A_h^s$ $H$-converges to $\mathcal A_\infty^s$ as $h\to\infty$},
\]
if and only if
\[
\text{$\mathcal A_h$ $H$-converges to $\mathcal A_\infty$ as $h\to\infty$}.
\]
\end{theorem}

We divide the proof of Theorem~\ref{thm:Hs} into two steps. 
First, we show that $H$-convergence in the local setting implies the corresponding notion in the nonlocal framework. 
Then, by exploiting well-known properties of local $H$-convergence, we show the converse implication.

To prove the first implication, we need the following technical lemma, which is well-known in the $H$-convergence literature; see, for instance,~\cite[Lemma~7.8]{CPDMDF}.

\begin{lemma}\label{lem:key-con}
Let $\gamma,\delta\in [0,1]$ be such that $\gamma+\delta\le 1$. Let $(a_h)_h\subset L^1(\Omega)$, $(b_h)_h\subset L^1(\Omega)$, and $(c_h)_h\subset L^1(\Omega)$ be such that for every $h\in\N$
\[
a_h,b_h\ge 0\quad\text{a.e.\ in $\Omega$},\qquad |c_h|\le a_h^\gamma b_h^\delta\quad\text{a.e.\ in $\Omega$}.
\]
Assume the existence of $a,b,c\in L^1(\Omega)$ such that as $h\to\infty$
\[
a_h\to a\quad\text{weakly* in $\mathcal D'(\Omega)$},\quad b_h\to b\quad\text{weakly* in $\mathcal D'(\Omega)$},\quad c_h\to c\quad\text{weakly* in $\mathcal D'(\Omega)$}.
\]
Then
\[
|c|\le a^\gamma b^\delta\quad\text{a.e.\ in $\Omega$}.
\]
\end{lemma}

We are in position to prove the following result.

\begin{proposition}\label{prop:loc-nonloc}
Let $\Omega\subset\R^n$ be a bounded open set satisfying $|\partial\Omega|=0$. Let $A_0\in \M_{\R^n}(p,\lambda,\Lambda)$ satisfy~\eqref{eq:A0} and let $A_h\in\M_\Omega (p,\lambda,\Lambda, A_0)$ for $h\in \N\cup\{\infty\}$. 
Assume that
\[
\text{$\mathcal A_h$ $H$-converges to $\mathcal A_\infty$ as $h\to\infty$}.
\]
Then
\[
\text{$\mathcal A_h^s$ $H$-converges to $\mathcal A_\infty^s$ as $h\to\infty$}.
\]
\end{proposition}

\begin{proof}
For every $h\in \N\cup\{\infty\}$, let $\mathcal A_h$, $\mathcal B_h$, $\mathcal M_h$, and $\mathcal A_h^s$, $\mathcal B_h^s$, $\mathcal M_h^s$
denote the differential operators introduced in~\eqref{eq:Ah},~\eqref{eq:Bh},~\eqref{eq:Mh}, and~\eqref{eq:Ahs},~\eqref{eq:Bhs},~\eqref{eq:Mhs}, respectively.

By (M2), for every $f_1,f_2\in H^{-s,p'}(\Omega)$ we obtain
\begin{equation*}
\|\nabla^s \mathcal B_h^s(f_1)-\nabla^s \mathcal B_h^s(f_2)\|_{L^p(\R^n)}^p
\le\lambda^{-1}\|f_1-f_2\|_{H^{-s,p'}(\Omega)}\|\mathcal B_h^s(f_1)-\mathcal B_h^s(f_2)\|_{H^{s,p}_0(\Omega)}.
\end{equation*}
Hence, in view of Proposition~\ref{prop:poincare}, there exists a constant $C>0$, depending only on $n,s,p,\Omega$, and $\lambda$, such that
\begin{equation}\label{eq:Bhs-estimate}
\|\mathcal B_h^s(f_1)-\mathcal B_h^s(f_2)\|_{H^{s,p}_0(\Omega)}
\le
C\|f_1-f_2\|_{H^{-s,p'}(\Omega)}^{\frac{1}{p-1}}
\quad
\text{for every } f_1,f_2\in H^{-s,p'}(\Omega).
\end{equation}
Similarly, by (M1) and (M2), for every $f_1,f_2\in H^{-s,p'}(\Omega)$ we obtain
\begin{equation}\label{eq:Bhs-estimate2}
\|\mathcal B_h^s(f_i)\|_{H^{s,p}_0(\Omega)}
\le
C\|f_i\|_{H^{-s,p'}(\Omega)}^{\frac{1}{p-1}}
\quad
\text{for every } i\in \{1,2\}.
\end{equation}
Moreover, by~\eqref{eq:A0-con} and Hölder’s inequality, with conjugate exponents 
\[
\frac{(p-1)^2}{p(p-2)}\quad\text{and}\quad(p-1)^2,
\]
it follows that, for every $f_1,f_2\in H^{-s,p'}(\Omega)$,
\begin{align*}
&\|\mathcal M_h^s(f_1)-\mathcal M_h^s(f_2)\|_{L^{p'}(\R^n)}^{p'}\\
&\le
C\int_{\R^n}(\mathbf{1}_\Omega(x)+|\nabla^s \mathcal B_h^s(f_1)(x)|^p+|\nabla^s \mathcal B_h^s(f_2)(x)|^p)^{\frac{p(p-2)}{(p-1)^2}}
|\nabla^s \mathcal B_h^s(f_1)(x)- \nabla^s \mathcal B_h^s(f_2)(x)|^{\frac{p}{(p-1)^2}\,}\d x\\
&\le
C\left(|\Omega|+\|\mathcal B_h^s(f_1)\|_{H^{s,p}_0(\R^n)}^p+\|\mathcal B_h^s(f_2)\|_{H^{s,p}_0(\R^n)}^p\right)^{\frac{p(p-2)}{(p-1)^2}}
\|\mathcal B_h^s(f_1)- \mathcal B_h^s(f_2)\|_{H^{s,p}_0(\R^n)}^{\frac{p}{(p-1)^2}},
\end{align*}
where $C>0$ depends only on $p$, $\lambda$, $\Lambda$, and $A_0$. 
Combining this estimate with~\eqref{eq:Bhs-estimate} and~\eqref{eq:Bhs-estimate2}, we conclude that for every $f_1,f_2\in H^{-s,p'}(\Omega)$
\begin{align*}
\|\mathcal M_h^s(f_1)-\mathcal M_h^s(f_2)\|_{L^{p'}(\R^n)}
\le
C\left(1+\|f_1\|_{H^{-s,p'}(\Omega)}^{p'}+\|f_2\|_{H^{-s,p'}(\Omega)}^{p'}\right)^{\frac{p-2}{p-1}}
\|f_1- f_2\|_{H^{-s,p'}(\Omega)}^{\frac{1}{(p-1)^2}},
\end{align*}
for some constant $C>0$ depending only on $n$, $s$, $p$, $\Omega$, $\lambda$, $\Lambda$, and $A_0$.

Consequently, by a standard density argument and up to a not relabeled subsequence, there exist two differential operators
\[
\widetilde{\mathcal B}_\infty^s \colon H^{-s,p'}(\Omega)\to H^{s,p}_0(\Omega),
\qquad
\widetilde{\mathcal M}_\infty^s \colon H^{-s,p'}(\Omega)\to L^{p'}(\R^n;\R^n),
\]
such that for every $f\in H^{-s,p'}(\Omega)$
\begin{align}
\mathcal B_h^s(f)
&\to \widetilde{\mathcal B}_\infty^s(f)
\quad\text{weakly in $H^{s,p}_0(\Omega)$ as $h\to\infty$}, \label{eq:NL1}\\
\mathcal M_h^s(f)
&\to \widetilde{\mathcal M}_\infty^s(f)
\quad\text{weakly in $L^{p'}(\R^n;\R^n)$ as $h\to\infty$}.
\label{eq:NL2}
\end{align}
To conclude the proof, it is sufficient to show that for every
$f\in H^{-s,p'}(\Omega)$
\begin{equation}\label{eq:claim}
\widetilde{\mathcal M}_\infty^s(f)(x) = A_\infty(x,\nabla^s \widetilde{\mathcal B}_\infty^s(f)(x))
\quad\text{for a.e.\ $x\in\R^n$}.
\end{equation}
Indeed, assuming~\eqref{eq:claim}, by passing to the limit as $h\to\infty$ in the equation~\eqref{eq:Ahs}, satisfied by $\mathcal B_h^s(f)$, we obtain
\begin{equation}\label{eq:uniq}
\int_{\R^n}
A_\infty(x,\nabla^s \widetilde{\mathcal B}_\infty^s(f)(x))
\cdot \nabla^s v(x)\,\mathrm d x
=
\langle f,v\rangle_{H^{-s,p'}(\Omega)\times H^{s,p}_0(\Omega)}
\quad
\text{for every $v\in H^{s,p}_0(\Omega)$}.
\end{equation}
By the uniqueness of solutions to~\eqref{eq:uniq}, it follows that
\[
\widetilde{\mathcal B}_\infty^s(f)=\mathcal B_\infty^s(f)\quad\text{for every $f\in H^{-s,p'}(\Omega)$}.
\]
Hence, denoting by $\mathcal M_\infty^s$ the operator introduced in~\eqref{eq:Mhs}, it readily follows that
\[
\widetilde{\mathcal M}_\infty^s(f)(x) = A_\infty(x,\nabla^s\mathcal B_\infty^s(f)(x)) = \mathcal M_\infty^s(f)(x) \quad
\text{for a.e.\ $x\in\R^n$}.
\]
In view of~\eqref{eq:NL1} and~\eqref{eq:NL2}, this proves that
$\mathcal A_h^s$ $H$-converges to $\mathcal A_\infty^s$ as $h\to\infty$.

Let us prove~\eqref{eq:claim}. 
Let $f\in H^{-s,p'}(\Omega)$ be fixed.
By Proposition~\ref{prop:strong-gradient}, we have
\[
\nabla^s \mathcal B_h(f)\to \nabla^s\widetilde{\mathcal B}_\infty(f)\quad\text{strongly in $L^p_{\rm loc}(\R^n\setminus\overline\Omega;\R^n)$ as $h\to\infty$},
\]
which implies that
\[
\mathcal M_h^s(f)=A_0(\,\cdot\,,\nabla^s\mathcal B_h(f))\to A_0(\,\cdot\,,\nabla^s\widetilde{\mathcal B}_\infty^s(f))\quad\text{strongly in $L^{p'}_{\rm loc}(\R^n\setminus\overline\Omega;\R^n)$ as $h\to\infty$}.
\]
Hence,
\begin{equation}\label{eq:id2}
\widetilde{\mathcal M}_\infty^s(f)(x_0)=A_0(x_0,\nabla^s \widetilde{\mathcal B}_\infty^s(f)(x_0))\quad\text{for a.e.\ $x_0\in\R^n\setminus\overline\Omega$}.
\end{equation}

Let $f\in H^{-s,p'}(\Omega)$ and $g\in W^{-1,p'}(\Omega)$ be fixed. For every $h\in\N$, we define the following functions $a_h,b_h,c_h\in L^1(\Omega)$:
\begin{align*}
a_h(x)&\coloneqq 1+\mathcal M_h^s(f)(x)\cdot\nabla^s \mathcal B_h^s(f)(x)+\mathcal M_h(g)(x)\cdot\nabla\mathcal B_h(g)(x)& &\text{for a.e.\ $x\in\Omega$},\\
b_h(x)&\coloneqq (\mathcal M_h^s(f)(x)-\mathcal M_h(g)(x))\cdot(\nabla^s \mathcal B_h^s(f)(x)-\nabla \mathcal B_h(g)(x))& &\text{for a.e.\ $x\in\Omega$},\\
c_h(x)&\coloneqq \mathcal M_h^s(f)(x)-\mathcal M_h(g)(x)& &\text{for a.e.\ $x\in\Omega$}.
\end{align*}
Note that, by (M2) and (M3)
\[
a_h\ge 0\quad\text{a.e.\ in $\Omega$},\qquad 
b_h\ge 0\quad\text{a.e.\ in $\Omega$},\qquad |c_h|
\le \Lambda a_h^{\frac{p-2}{p}} b_h^{\frac{1}{p}}\quad\text{a.e.\ in $\Omega$}.
\]
Since $\mathcal A_h$ $H$-converges to $\mathcal A_\infty$ as $h\to\infty$ by hypotheses, then
\begin{equation}\label{eq:ch}
c_h\to \widetilde{\mathcal M}_\infty^s(f)-\mathcal M_\infty(g)\quad\text{weakly in $L^{p'}(\Omega;\R^n)$ as $h\to\infty$}.
\end{equation}
Let us prove that the following convergences hold:
\begin{align}
\mathcal M_h(g)\cdot \nabla \mathcal B_h(g)&\to\mathcal M_\infty(g)\cdot \nabla \mathcal B_\infty(g)& &\text{weakly* in $\mathcal D'(\Omega)$ as $h\to\infty$},\label{eq:H1}\\
\mathcal M_h(g) \cdot \nabla^s \mathcal B_h^s(f)&\to\mathcal M_\infty(g) \cdot \nabla^s \widetilde{\mathcal B}_\infty^s(f)& &\text{weakly* in $\mathcal D'(\Omega)$ as $h\to\infty$},\label{eq:H2}\\
\mathcal M_h^s(f)\cdot \nabla^s \mathcal B_h^s(f)&\to \widetilde{\mathcal M}_\infty^s(f)\cdot \nabla^s \widetilde{\mathcal B}_\infty^s(f)& &\text{weakly* in $\mathcal D'(\Omega)$ as $h\to\infty$},\label{eq:H3}\\
\mathcal M_h^s(f)\cdot\nabla \mathcal B_h(g)&\to \widetilde{\mathcal M}_\infty^s(f)\cdot\nabla \mathcal B_\infty(g)& &\text{weakly* in $\mathcal D'(\Omega)$ as $h\to\infty$}.\label{eq:H4}
\end{align}
Let $\phi\in C_c^\infty(\Omega)$ be fixed. 
By the Rellich compactness theorem, it holds that
\begin{equation}\label{eq:ffinal}
\mathcal B_h(g)\to \mathcal B_\infty(g)\quad\text{weakly in $W^{1,p}_0(\Omega)$ and strongly in $L^p(\Omega)$ as $h\to\infty$}.
\end{equation}
In particular, letting $h\to\infty$, we get
\begin{align*}
&\int_{\Omega}\mathcal M_h(g)(x)\cdot \nabla \mathcal B_h(g)(x)\phi(x)\,\d x\\
&=\langle g,\mathcal B_h(g)\phi\rangle_{W^{-1,p'}(\Omega)\times W^{1,p}_0(\Omega)}-\int_{\Omega}\mathcal M_h(g)(x)\cdot \nabla \phi(x) \mathcal B_h(g)(x)\,\d x\\
&\to\langle g,\mathcal B_\infty(g)\phi\rangle_{W^{-1,p'}(\Omega)\times W^{1,p}_0(\Omega)}-\int_{\Omega}\mathcal M_\infty(g)(x)\cdot \nabla \phi(x) \mathcal B_\infty(g)(x)\,\d x\\
&=\int_{\Omega}\mathcal M_\infty(g)(x)\cdot \nabla \mathcal B_\infty(g)(x)\phi(x)\,\d x,
\end{align*}
which yields the validity of~\eqref{eq:H1}. 
By Corollary~\ref{coro:equiv}, there exist a not relabeled subsequence and functions $w_h\in W^{1,p}(\Omega)$ for $h\in\N\cup\{\infty\}$ satisfying
\[
\nabla w_h = \nabla^s \mathcal B_h^s(f) \quad \text{in $\Omega$ for every $h \in \N$},\qquad \nabla w_\infty = \nabla^s \widetilde{\mathcal B}^s_\infty(f) \quad \text{in $\Omega$},
\]
and
\[
w_h \to w_\infty \quad \text{weakly in $W^{1,p}(\Omega)$ and strongly in $L^p(\Omega)$ as $h \to \infty$}.
\]
Hence, arguing as in the proof of~\eqref{eq:H1}, we have as $h\to\infty$
\begin{align*}
&\int_{\Omega}\mathcal M_h(g)(x)\cdot \nabla^s \mathcal B_h^s(f)(x)\phi(x)\,\d x=\int_{\Omega}\mathcal M_h(g)(x)\cdot \nabla w_h(x)\phi(x)\,\d x\\
&\to\int_{\Omega}\mathcal M_\infty(g)(x)\cdot \nabla w_\infty(x)\phi(x)\,\d x=\int_{\Omega}\mathcal M_\infty(g)(x)\cdot \nabla^s \widetilde{\mathcal B}_\infty^s(f)(x)\phi(x)\,\d x,
\end{align*}
which shows~\eqref{eq:H2}. 
Let us prove~\eqref{eq:H3}.
By Proposition~\ref{prop:Leibniz} we have
\begin{align*}
&\int_{\Omega}\mathcal M_h^s(f)(x)\cdot \nabla^s \mathcal B_h^s(f)(x)\phi(x)\,\d x=\int_{\R^n}\mathcal M_h^s(f)(x)\cdot \nabla^s \mathcal B_h^s(f)(x)\phi(x)\,\d x\\
&=\langle f,\mathcal B_h^s(f)\phi\rangle_{H^{-s,p'}(\Omega)\times H^{s,p}_0(\Omega)}-\int_{\R^n}\mathcal M_h^s(f)(x)\cdot \nabla^s \phi(x) \mathcal B_h^s(f)(x)\,\d x\\
&\quad-\int_{\R^n}\mathcal M_h^s(f)(x)\cdot \nabla_{\rm NL}^s(\phi, \mathcal B_h^s(f))(x)\,\d x.
\end{align*}
Proposition~\ref{prop:rellich}, combined with~\eqref{eq:nabla_NL} and~\eqref{eq:NL1}, yields
\[
\nabla_{\rm NL}^s(\phi, \mathcal B_h^s(f))\to \nabla_{\rm NL}^s(\phi, \widetilde{\mathcal B}_\infty^s(f))\quad\text{strongly in $L^p(\R^n;\R^n)$ as $h\to\infty$}.
\]
Moreover,
\[
\int_{\R^n}\widetilde{\mathcal M}_\infty^s(f)(x)\cdot \nabla^s (\widetilde{\mathcal B}_\infty^s(f)\phi)(x)\,\d x=\langle f,\widetilde{\mathcal B}_\infty^s(f)\phi\rangle_{H^{-s,p'}(\Omega)\times H^{s,p}_0(\Omega)}.
\]
Therefore, letting $h\to\infty$, we obtain
\begin{align*}
&\int_{\Omega}\mathcal M_h^s(f)(x)\cdot \nabla^s \mathcal B_h^s(f)(x)\phi(x)\,\d x\\
&\to \langle f,\widetilde{\mathcal B}_\infty^s(f)\phi\rangle_{H^{-s,p'}(\Omega)\times H^{s,p}_0(\Omega)}-\int_{\R^n}\widetilde{\mathcal M}_\infty^s(f)(x)\cdot \nabla^s \phi(x) \widetilde{\mathcal B}_\infty^s(f)(x)\,\d x\\
&\quad-\int_{\R^n}\widetilde{\mathcal M}_\infty^s(f)(x)\cdot \nabla_{\rm NL}^s(\phi, \widetilde{\mathcal B}_\infty^s(f))(x)\,\d x\\
&=\int_{\R^n}\widetilde{\mathcal M}_\infty^s(f)(x)\cdot \nabla^s( \widetilde{\mathcal B}_\infty^s(f)\phi)(x)\,\d x-\int_{\R^n}\widetilde{\mathcal M}_\infty^s(f)(x)\cdot \nabla^s \phi(x) \widetilde{\mathcal B}_\infty^s(f)(x)\,\d x\\
&\quad-\int_{\R^n}\widetilde{\mathcal M}_\infty^s(f)(x)\cdot \nabla_{\rm NL}^s(\phi, \widetilde{\mathcal B}_\infty^s(f))(x)\,\d x\\
&=\int_{\R^n}\widetilde{\mathcal M}_\infty^s(f)(x)\cdot \nabla^s \widetilde{\mathcal B}_\infty^s(f)(x)\phi(x)\,\d x=\int_{\Omega}\widetilde{\mathcal M}_\infty^s(f)(x)\cdot \nabla^s \widetilde{\mathcal B}_\infty^s(f)(x)\phi(x)\,\d x,
\end{align*}
where the last equality follows from the fact that $\phi\in C_c^\infty(\Omega)$.
This proves~\eqref{eq:H3}. 
For every $h\in\N\cup\{\infty\}$, we trivially extend $v_h\in W^{1,p}_0(\Omega)$ to a function in $W^{1,p}(\R^n)$ and we define $v_h\coloneqq(-\Delta)^\frac{1-s}{2}\mathcal B_h(g)\in H^{s,p}(\R^n)$.
By Proposition~\ref{prop:equiv}(ii),
\[
\nabla^s v_h = \nabla \mathcal B_h(g) \quad \text{in $\Omega$ for every $h \in \N\cup\{\infty\}$}.
\]
Moreover, combining~\eqref{eq:frac-lap-est} with~\eqref{eq:ffinal}, we obtain
\[
v_h \to v_\infty \quad \text{weakly in $H^{s,p}(\R^n)$ and strongly in $L^p(\R^n)$ as $h \to \infty$}.
\]
Hence, by applying Proposition~\ref{prop:Leibniz}, we can argue as in~\eqref{eq:H3} to deduce, as $h\to\infty$
\begin{align*}
&\int_{\Omega}\mathcal M_h^s(f)(x)\cdot \nabla \mathcal B_h(g)(x)\phi(x)\,\d x=\int_{\R^n}\mathcal M_h^s(f)(x)\cdot \nabla^s v_h(x)\phi(x)\,\d x\\
&\to \int_{\R^n}\widetilde{\mathcal M}_\infty^s(f)(x)\cdot \nabla^s v_\infty(x)\phi(x)\,\d x=\int_{\Omega}\widetilde{\mathcal M}_\infty^s(f)(x)\cdot \nabla \mathcal B_\infty(g)(x)\phi(x)\,\d x,
\end{align*}
which proves~\eqref{eq:H4}.
As a consequence of~\eqref{eq:H1}--\eqref{eq:H4}, we obtain
\begin{align}
a_h&\to 1+\widetilde{\mathcal M}_\infty^s(f)\cdot \nabla^s \widetilde{\mathcal B}_\infty^s(f)+ \mathcal M_\infty(g)\cdot \nabla \mathcal B_\infty(g)\quad\text{weakly* in $\mathcal D'(\Omega)$ as $h\to\infty$},\label{eq:ah}\\
b_h&\to (\widetilde{\mathcal M}_\infty^s(f)-\mathcal M_\infty(g))\cdot(\nabla^s \widetilde{\mathcal B}_\infty^s(f)-\nabla \mathcal B_\infty(g))\quad\text{weakly* in $\mathcal D'(\Omega)$ as $h\to\infty$}.\label{eq:bh}
\end{align}
By~\eqref{eq:ch},~\eqref{eq:ah},~\eqref{eq:bh}, and Lemma~\ref{lem:key-con}, for every $f\in H^{-s,p'}(\Omega)$ and $g\in W^{-1,p'}(\Omega)$ we obtain, for a.e.\ $x\in\Omega$,
\begin{align*}
|\widetilde{\mathcal M}_\infty^s(f)(x)-\mathcal M_\infty(g)(x)|
&\le \Lambda[1+\widetilde{\mathcal M}_\infty^s(f)(x)\cdot\nabla^s \widetilde{\mathcal B}_\infty^s(f)(x)
+\mathcal M_\infty(g)(x)\cdot \nabla\mathcal B_\infty(g)(x)]^{\frac{p-2}{p}}\\
&\quad\times [(\widetilde{\mathcal M}_\infty^s(f)(x)-\mathcal M_\infty(g)(x))
\cdot(\nabla^s \widetilde{\mathcal B}_\infty^s(f)(x)-\nabla\mathcal B_\infty(g)(x))]^{\frac{1}{p}}.
\end{align*}
In particular, since the differential operator $\mathcal B_\infty\colon W^{-1,p'}(\Omega)\to W^{1,p}_0(\Omega)$ is bijective, for every $f\in H^{-s,p'}(\Omega)$ and $\varphi\in C_c^1(\Omega)$ we have, for a.e.\ $x\in\Omega$,
\begin{equation}\label{eq:momenta}
\begin{split}
|\widetilde{\mathcal M}_\infty^s(f)(x)-A_\infty(x,\nabla\varphi(x))|
&\le \Lambda[1+\widetilde{\mathcal M}_\infty^s(f)(x)\cdot\nabla^s \widetilde{\mathcal B}_\infty^s(f)(x)
+A_\infty(x,\nabla\varphi(x))\cdot \nabla \varphi(x)]^{\frac{p-2}{p}}\\
&\quad\times [(\widetilde{\mathcal M}_\infty^s(f)(x)-A_\infty(x,\nabla\varphi(x)))
\cdot(\nabla^s \widetilde{\mathcal B}_\infty^s(f)(x)-\nabla \varphi(x))]^{\frac{1}{p}}.
\end{split}
\end{equation}
In view of assumption (M3), the set of points in $\Omega$ where the previous inequality holds can be chosen independently of $\varphi\in C_c^1(\Omega)$. 
Let $x_0\in\Omega$ be a point where~\eqref{eq:momenta} holds, let $\eta\in C_c^1(\Omega)$ be such that $\eta=1$ in a neighbourhood of $x_0$, and define
\[
\varphi(x)\coloneqq(\nabla^s\widetilde{\mathcal B}_\infty^s(f)(x_0)\cdot x)\eta(x)\quad\text{for $x\in\R^n$}.
\]
In particular, 
\[
\nabla\varphi(x)=\nabla^s\widetilde{\mathcal B}_\infty^s(f)(x_0)
\]
in a neighbourhood of $x_0$.
Then, by~\eqref{eq:momenta},
\begin{equation}\label{eq:id1}
\widetilde{\mathcal M}_\infty^s(f)(x_0)=A_\infty(x_0,\nabla^s\widetilde{\mathcal B}_\infty^s(f)(x_0)).
\end{equation}
By the arbitrariness of $x_0$,~\eqref{eq:id1} holds for a.e.\ $x_0\in\Omega$.

Combining~\eqref{eq:id2} and ~\eqref{eq:id1}, and recalling that $|\partial\Omega|=0$, we obtain~\eqref{eq:claim}.
This completes the proof.
\end{proof}

Let us prove the converse implication. To this end, we first prove the following preliminary result, which shows that the nonlocal $H$-limit is unique.

\begin{lemma}\label{lem:uniqueness}
Let $\Omega\subset\R^n$ be a bounded open set. Let $A_0\in \M_{\R^n}(p,\lambda,\Lambda)$ satisfy~\eqref{eq:A0} and $A_h\in\M_\Omega (p,\lambda,\Lambda, A_0)$ for $h\in \N$. Assume that
$A_\infty,\widehat A_\infty \in \M_\Omega (p,\lambda,\Lambda, A_0)$ satisfy
\begin{equation}\label{eq:H-A}
\text{$\mathcal A_h^s$ $H$-converges to $\mathcal A^s_\infty$ as $h\to\infty$}
\end{equation}
and
\begin{equation}\label{eq:H-hatA}
\text{$\mathcal A_h^s$ $H$-converges to $\widehat{\mathcal A}_\infty^s$ as $h\to\infty$},
\end{equation}
where $\mathcal A_\infty^s$ and $\widehat{\mathcal A}_\infty^s$ denote the differential operators defined in~\eqref{eq:Ahs} associated with $A_\infty$ and $\widehat A_\infty$, respectively. Then
\begin{equation}\label{eq:H-uniq}
A_\infty(x,\xi)=\widehat A_\infty(x,\xi)\quad\text{for a.e.\ $x\in\R^n$ and every $\xi\in\R^n$}.
\end{equation}
\end{lemma}

\begin{proof}
First, we note that by the definition of the class $\M_\Omega(p,\lambda,\Lambda,A_0)$, we have that 
\[
A_\infty(x,\xi)=A_0(x,\xi)=\widehat A_\infty(x,\xi)\quad\text{for a.e.\ $x\in\R^n\setminus\Omega$ and every $\xi\in\R^n$}.
\]
To conclude it is enough to prove~\eqref{eq:H-uniq} for a.e. $x\in\Omega$ and every $\xi\in\R^n$.

By Lemma~\ref{lem:existencefrac}, the differential operators $\mathcal A_\infty^s,\widehat{\mathcal A}_\infty^s\colon H^{s,p}_0(\Omega)\to H^{-s,p'}(\Omega)$ are invertible.
Let $\mathcal B_\infty^s\coloneqq (\mathcal A_\infty^s)^{-1}\colon H^{-s,p'}(\Omega)\to H^{s,p}_0(\Omega)$ and $\widehat{\mathcal B}_\infty^s\coloneqq (\widehat{\mathcal A}_\infty^s)^{-1}\colon H^{-s,p'}(\Omega)\to H^{s,p}_0(\Omega)$, and let $f\in H^{-s,p'}(\Omega)$ be fixed. By (NH1) in~\eqref{eq:H-A} and~\eqref{eq:H-hatA}, we deduce that
\[
\mathcal B_\infty^s(f)=\widehat{\mathcal B}_\infty^s(f).
\]
Moreover, by (NH2), for every $\Phi\in L^p(\R^n;\R^n)$ we have
\[
\int_{\R^n}A_\infty(x,\nabla^s \mathcal B_\infty^s(f)(x))\cdot\Phi(x)\,\d  x=\int_{\R^n}\widehat A_\infty(x,\nabla^s \mathcal B_\infty^s(f)(x))\cdot \Phi(x)\,\d  x.
\]
Since $f\in H^{-s,p'}(\Omega)$ can be arbitrarily chosen, we derive that for every $\varphi\in C_c^\infty(\Omega)$ and every $\Phi\in L^p(\R^n;\R^n)$
\[
\int_{\R^n}A_\infty(x,\nabla^s \varphi(x))\cdot\Phi(x)\,\d  x=\int_{\R^n}\widehat A_\infty(x,\nabla^s \varphi(x))\cdot \Phi(x)\,\d  x.
\]
Hence, 
\begin{equation}\label{eq:H-identity}
A_\infty(x,\nabla^s \varphi(x))=\widehat A_\infty(x,\nabla^s \varphi(x))\quad\text{for a.e.\ $x\in\Omega$ and every $\varphi\in C_c^\infty(\Omega)$},
\end{equation}
and, in view of (M3), the collections of points of $\R^n$ where~\eqref{eq:H-identity} fails can be chosen independent of $\varphi\in C_c^\infty(\Omega)$. 
Let $\xi\in\R^n$ be fixed. By~\cite[Lemma~4.3]{KrSc22}, for every $x_0\in\Omega$ we can find a function $\varphi_{x_0,\xi} \in C_c^\infty(\Omega)$ such that 
\[
\nabla^s\varphi_{x_0,\xi} (x_0)= \xi.
\]
By choosing $\varphi=\varphi_{x_0,\xi}$ in~\eqref{eq:H-identity}, we derive that 
\[
A_\infty(x_0,\xi)=\widehat A_\infty(x_0,\xi)\quad\text{for a.e.\ $x_0\in \Omega$}.
\]
By the arbitrariness of $\xi$, the conclusion follows.
\end{proof}

\begin{remark}\label{rem:uniqueness}
In the local case, the uniqueness result is well-known; see~\cite[Proposition~2.9]{DASC} for a proof. However, the strategy in~\cite[Proposition~2.9]{DASC} cannot be applied straightforwardly in the nonlocal setting. 
Indeed, in the local framework the authors construct a test function $\varphi_{x_0,\xi}$ such that $\nabla \varphi_{x_0,\xi} = \xi$ in a neighbourhood of a point $x_0$. In the nonlocal case, constructing such functions is much more delicate, since one cannot impose $\nabla^s \varphi_{x_0,\xi} = \xi$ in a neighbourhood of $x_0$, but only at the point $x_0$ itself.
\end{remark}

Moreover, we recall the following compactness result concerning local $H$-convergence. 

\begin{proposition}\label{prop:Hcomploc}
Let $\Omega\subset\R^n$ be a bounded open set. Let $A_h\in \M_\Omega (p,\lambda,\Lambda)$ for $h\in\N$. Then, there exist a subsequence $A_{h_k}\in \M_\Omega (p,\lambda,\Lambda)$ for $k\in\N$ and a function $A_\infty\in \M_\Omega (p,\lambda,\Lambda)$ such that
\[
\text{$\mathcal A_{h_k}$ $H$-converges to $\mathcal A_\infty$ as $k\to\infty$}.
\]
\end{proposition}

\begin{remark}
The proof of Proposition~\ref{prop:Hcomploc} is standard. We refer to~\cite[Theorem~11.2]{Tartar} for the case $p=2$, and to~\cite[Theorem~3.5]{DASC} and~\cite[Corollary~7.12]{CPDMDF} for the general case $p\ge 2$. Note that in~\cite[Theorem~3.5]{DASC} and~\cite[Corollary~7.12]{CPDMDF} the authors show that the class $\M_\Omega'(p,\lambda,\Lambda)$ is precompact in $\M_\Omega''(p,\lambda,\Lambda')$ for a suitable $\Lambda'\ge \Lambda$. This follows from the fact that, to show the continuity estimate for the limit map $A_\infty$, one passes from (M3') to (M3) for $A_h$, applies the div--curl lemma to show that the limit $A_\infty$ satisfies (M3), and finally passes to (M3''). If, instead, one starts directly from $\M_\Omega(p,\lambda,\Lambda)$, then the limit $A_\infty$ belongs to the same class. In particular, the class $\M_\Omega(p,\lambda,\Lambda)$ is compact with respect to local $H$-convergence.
\end{remark}

We can now prove the following result.

\begin{proposition}\label{prop:nonloc-loc}
Let $\Omega\subset\R^n$ be a bounded open set satisfying $|\partial\Omega|=0$. Let $A_0\in \M_{\R^n}(p,\lambda,\Lambda)$ satisfy~\eqref{eq:A0} and $A_h\in\M_\Omega (p,\lambda,\Lambda, A_0)$ for $h\in \N\cup\{\infty\}$. Assume that
\begin{equation*}
\text{$\mathcal A^s_h$ $H$-converges to ${\mathcal A}^s_\infty$ as $h\to\infty$}.
\end{equation*}
Then
\begin{equation*}
\text{$\mathcal A_h$ $H$-converges to $\mathcal A_\infty$ as $h\to\infty$}.
\end{equation*}
\end{proposition}

\begin{proof}
By Proposition~\ref{prop:Hcomploc}, there exist a subsequence $A_{h_k}\in \M_\Omega(p,\lambda,\Lambda)$ for $k\in\N$ and a function
$\widehat A_\infty \in \M_\Omega(p,\lambda,\Lambda)$ such that
\[
\text{$\mathcal A_{h_k}$ $H$-converges to $\widehat{\mathcal A}_\infty$ as $k\to\infty$},
\]
where $\widehat{\mathcal A}_\infty$ is the differential operator defined in~\eqref{eq:Ah} associated with $\widehat A_\infty$.
By Proposition~\ref{prop:loc-nonloc}, it follows that
\[
\text{$\mathcal A_{h_k}^s$ $H$-converges to $\widehat{\mathcal A}_\infty^s$ as $k\to\infty$},
\]
where $\widehat{\mathcal A}_\infty^s$ is the differential operator defined in~\eqref{eq:Ahs} associated with $\widehat A_\infty$.
By Lemma~\ref{lem:uniqueness}, we deduce that
\[
A_\infty(x,\xi)=\widehat A_\infty(x,\xi)\quad\text{for a.e.\ $x\in\R^n$ and for every $\xi\in\R^n$}.
\]
Hence, $\mathcal A_{h_k}$ $H$-converges to $\mathcal A_\infty$ as $k\to\infty$. Since this argument applies to every subsequence, the Urysohn property of $H$-convergence implies that
$\mathcal A_h$ $H$-converges to $\mathcal A_\infty$ as $h\to\infty$.
\end{proof}

Finally, we prove the main result of this section.

\begin{proof}[Proof of Theorem~\ref{thm:Hs}]
It is enough to combine Proposition~\ref{prop:loc-nonloc} and Proposition~\ref{prop:nonloc-loc}.
\end{proof}

We conclude this section with the following corollary, which is an immediate consequence of Theorem~\ref{thm:Hs} and Proposition~\ref{prop:Hcomploc}.

\begin{corollary}\label{cor:nonlocal_H-compactness}
Let $\Omega\subset\R^n$ be a bounded open set satisfying $|\partial\Omega|=0$. Let $A_0\in \M_{\R^n}(p,\lambda,\Lambda)$ satisfy~\eqref{eq:A0} and let $A_h\in\M_\Omega (p,\lambda,\Lambda, A_0)$ for every $h\in \N$. 
Then, there exist a subsequence $A_{h_k}\in\M_\Omega (p,\lambda,\Lambda, A_0)$ for $k\in\N$ and a function $A_\infty\in \M_\Omega (p,\lambda,\Lambda,A_0)$ such that
\begin{equation}\label{eq:nonlocal-compactness}
\text{$\mathcal A_{h_k}^s$ $H$-converges to $\mathcal A_\infty^s$ as $k\to\infty$}.
\end{equation}
\end{corollary}

\begin{proof}
By Proposition~\ref{prop:Hcomploc}, there exist a subsequence $A_{h_k}\in\M_\Omega (p,\lambda,\Lambda)$ for $k\in\N$ and a function $A_\infty\in \M_\Omega (p,\lambda,\Lambda)$ such that
\begin{equation*}
\text{$\mathcal A_{h_k}$ $H$-converges to $\mathcal A_\infty$ as $k\to\infty$}.
\end{equation*}
Hence, if we consider $A_{h_k}\in\M_\Omega (p,\lambda,\Lambda,A_0)$ for $k\in\N\cup\{\infty\}$, by Theorem~\ref{thm:Hs} we conclude that~\eqref{eq:nonlocal-compactness} holds.
\end{proof}


\section{The \texorpdfstring{$\Gamma$}{Gamma}-convergence problem}\label{sec_4}

In this section, we introduce the $\Gamma$-convergence problem for both local and nonlocal integral functionals.
This is the variational counterpart of the $H$-convergence problem for the differential operators studied in the previous sections. 
Let us first recall the notion of $\Gamma$-convergence for functionals defined on Banach spaces. 
For a comprehensive treatment of the subject, we refer the interested reader to the monographs~\cite{Bra,DalMaso}.

\begin{definition}\label{def:Gamma}
Let $X$ be a Banach space and let $E_j\colon X\to \mathbb{R}\cup\{\infty\}$ for $j\in\N\cup\{\infty\}$. We say that
\[
\text{$E_j$ $\Gamma$-converges to $E_\infty$ strongly in $X$ as $j\to\infty$}
\]
if for every $x_\infty\in X$ the following conditions hold:
\begin{itemize}
\item[($\Gamma$1)] for every sequence $(x_j)_j\subset X$ such that $x_j\to x_\infty$ strongly in $X$ as $j\to\infty$, one has
\[
E_\infty(x_\infty)\le\liminf_{j\to\infty} E_j(x_j),
\]
\item[($\Gamma$2)] there exists a sequence $(y_j)_j\subset X$ such that $y_j\to x_\infty$ strongly in $X$ as $j\to\infty$ and
\[
E_\infty(x_\infty)\ge\limsup_{j\to\infty} E_j(y_j).
\]
\end{itemize}
\end{definition}

We restrict our analysis to the subclass of \textit{conservative} monotone maps.

\begin{definition}
A map $A\colon\R^n\to\R^n$ is called {\em conservative} if there exists a differentiable function $\varphi\colon \R^n\to\R$ such that
\[
A(\xi)=\nabla \varphi(\xi)\quad\text{for every }\xi\in\R^n.
\]
The function $\varphi$ is called a \emph{potential} of $A$.
\end{definition}

Let us recall the following characterization of conservative maps; see~\cite[Proposition~41.5]{Zeidler3}.

\begin{proposition}\label{prop:potential}
Let $A\colon\R^n\to\R^n$ be continuous, and let us define the function $\varphi_A\colon\R^n\to \R$ as
\begin{equation}\label{eq:fA}
\varphi_A(\xi)\coloneqq \int_0^1 A(t\xi)\cdot \xi\,\d t\quad\text{for $\xi\in\R^n$}.
\end{equation}
Then, $A$ is conservative if and only if
\[
\varphi_A(\xi)-\varphi_A(\eta)=\int_0^1 A(\eta+t(\xi-\eta))\cdot (\xi-\eta)\,\d t\quad\text{for every $\xi,\eta\in\R^n$}.
\]
In such case, $\varphi_A$ is a potential of $A$.
\end{proposition}

\subsection{The class of local functionals} 

Let $\Omega \subset \R^n$ be an open set. Starting from the class $\M_\Omega(p,\lambda,\Lambda)$ (see Definition~\ref{def:classes}), we introduce the subclass of conservative monotone maps that is relevant for our analysis in the local setting.

\begin{definition}
Let $\Omega \subset \R^n$ be an open set. We define
\begin{align*}
\M_\Omega^{\rm grad}(p,\lambda,\Lambda)\coloneqq\{A\in\M_\Omega (p,\lambda,\Lambda) : \text{$A(x,\,\cdot\,)$ is a conservative map for a.e.\ $x\in\Omega$}\}.
\end{align*}
\end{definition}

\begin{remark}\label{rem:conservative_symmetric}
If the map $A\colon \R^n\to\R^n$ is differentiable, then $A$ is conservative if and only if 
\[
\nabla A(\xi)=\nabla A(\xi)^T\quad\text{for every $\xi\in\R^n$}.
\]
In particular, if 
\[
A(\xi)\coloneqq B\xi\quad\text{for every $\xi\in\R^n$},
\]
for some $B\in\R^{n\times n}$, then $A$ is a conservative map if and only if $B=B^T$, i.e. if $B\in\R^{n\times n}_{\rm sym}$.
As a consequence, the class $\M_\Omega^{\rm grad}(2,\lambda,\Lambda)$ coincides with the class $\mathcal{M}^{\rm sym}(\lambda,\Lambda,\Omega)$, introduced in~\cite{CCM25}. 
\end{remark}

We can now introduce the class of potentials associated with maps in $\M_\Omega^{\rm grad}(p,\lambda,\Lambda)$. 

\begin{definition}
Let $\Omega \subset \R^n$ be an open set.
Let $A\in\mathcal M^{\rm grad}_\Omega(p,\lambda,\Lambda)$. 
For a.e. $x\in\Omega$, let $\varphi_{A(x,\cdot)}$ denote the potential associated with the map $A(x,\cdot)\colon \R^n\to\R^n$ given by~\eqref{eq:fA}.
We define the function $\varphi_A\colon \Omega\times\R^n\to\R$ as
\begin{equation}\label{eq:varphiA}
\varphi_A(x,\xi)\coloneqq \varphi_{A(x,\,\cdot\,)}(\xi)\quad\text{for a.e.\ $x\in\Omega$ and for every $\xi\in\R^n$},
\end{equation} 
and the class
\begin{align*}
\Po_\Omega(p,\lambda,\Lambda)\coloneqq\{\varphi\colon\Omega\times \R^n\to \R \text{ Borel} :\varphi=\varphi_A\text{ for some } A\in \M^{\rm grad}_\Omega(p,\lambda,\Lambda)\}.
\end{align*}
\end{definition}

\begin{remark}
Since $A(x,\cdot)$ is continuous by (M3) and conservative by definition, Proposition~\ref{prop:potential} implies that $\varphi_A(x,\,\cdot\,)$ is a potential of $A(x,\,\cdot\,)$ for a.e.\ $x\in\Omega$. 
Moreover, by construction, $\varphi_A$ is a Carathéodory function and, in virtue of~\cite[Proposition~3.3]{BDF}, we may assume without loss of generality that $\varphi_A$ is also Borel measurable. 
We emphasise this point, since some of the results stated below require the Borel measurability of $\varphi_A$.
\end{remark}


According to the following proposition, a potential $\varphi=\varphi_A\in\Po_\Omega(p,\lambda,\Lambda)$ satisfies standard growth conditions and inherits certain regularity properties from the corresponding map $A$.

\begin{proposition}\label{prop:density-loc}
Let $\Omega \subset \R^n$ be an open set. Then, there exist positive constants $a_0,a_1,a_2,a_3$ and $\alpha\in(0,1)$, depending only on $p$, $\lambda$, and $\Lambda$, such that for every $\varphi \in\Po_\Omega(p,\lambda,\Lambda)$, for a.e.\ $x\in\Omega$, and for every $\xi,\eta\in\R^n$ the following hold:
\begin{itemize}
\item[\rm{(P1)}] the function $\varphi(x,\,\cdot\,)$ is of class $C^1(\R^n)$ and convex, $\varphi(x,0)=0$, and $\nabla_\xi\varphi(x,0)=0$,
\item[\rm{(P2)}] $a_0|\xi|^p\le \varphi(x,\xi)\le a_1(1+|\xi|^p)$,
\item[\rm{(P3)}] $|\varphi(x,\xi)-\varphi(x,\eta)|\le a_2(1+|\xi|+|\eta|)^{p-1}|\xi-\eta|$,
\item[\rm{(P4)}] $|\nabla_\xi \varphi(x,\xi)-\nabla_\xi \varphi(x,\eta)|\le a_3(1+|\xi|+|\eta|)^{p-1-\alpha}|\xi-\eta|^\alpha$.
\end{itemize} 
\end{proposition}

\begin{proof}
Let $\varphi \in\Po_\Omega(p,\lambda,\Lambda)$.
By definition, there exists $A\in \M_\Omega^{\rm grad}(p,\lambda,\Lambda)$ such that $\varphi=\varphi_A$.
In view of~\eqref{eq:fA} and~\eqref{eq:varphiA}, $\varphi(x,0)=0$ for a.e.\ $x\in\Omega$ and, by Proposition~\ref{prop:potential},
\begin{equation}\label{eq:nablavarphi-A}
\nabla_\xi \varphi(x,\xi)=A(x,\xi)\quad\text{for a.e.\ $x\in\Omega$ and for every $\xi\in\R^n$}.
\end{equation}
Moreover, by (M3) in Definition~\ref{def:classes}, $A(x,\cdot)$ is a continuous monotone map. 
Hence, $\varphi(x,\cdot)$ is convex and of class $C^1(\R^n)$.
In addition, by (M1), $\nabla_\xi \varphi(x,0)=0$ and (P1) follows. 

By (M1), (M2) and~\eqref{eq:A-est}, there exists a positive constant $C$, depending only on $p$, $\lambda$ and $\Lambda$, such that for a.e.\ $x\in\Omega$ and every $\xi\in\R^n$
\begin{align*}
&\varphi(x,\xi)\ge \lambda|\xi|^p\int_0^1 t^{p-1}\,\d t=\frac{\lambda}{p}|\xi|^p, \\ &\varphi(x,\xi)\le C\int_0^1(1+|t\xi|^{p-1})|\xi|\,\d t\le C(|\xi|+|\xi|^p)\le 2C(1+|\xi|^p),
\end{align*}
which gives (P2). 

By~\eqref{eq:A-est} and Proposition~\ref{prop:potential} for a.e.\ $x\in\Omega$ and every $\xi\in\R^n$
\begin{align*}
    |\varphi(x,\xi)-\varphi(x,\eta)|&\le C\int_0^1(1+|\eta+t(\xi-\eta)|^{p-1})|\xi-\eta|\,{\rm d}t\le C(1+|\xi|^{p-1}+|\eta|^{p-1})|\xi-\eta|\\
    &\le C(1+|\xi|+|\eta|)^{p-1}|\xi-\eta|,
\end{align*}
which implies (P3). 

Finally, by (M3) and (M3)$''$, whose validity is guaranteed by Lemma~\ref{lem:classes}, and by~\eqref{eq:nablavarphi-A}
\begin{align*}
    |\nabla_\xi \varphi(x,\xi)-\nabla_\xi \varphi(x,\eta)|&=|A(x,\xi)-A(x,\eta)|\le C(1+|\xi|^p+|\eta|^p)^\frac{p-2}{p-1}|\xi-\eta|^\frac{1}{p-1}\\
    &\le C(1+|\xi|+|\eta|)^\frac{p(p-2)}{p-1}|\xi-\eta|^\frac{1}{p-1}= C(1+|\xi|+|\eta|)^\frac{(p-1)^2-1}{p-1}|\xi-\eta|^\frac{1}{p-1},
\end{align*}
which yields the validity of (P4) for the value $\alpha=\frac{1}{p-1}$.
\end{proof}

Let us define the local integral functionals associated with the class $\M_\Omega^{\rm grad}(p,\lambda,\Lambda)$. 
Let $\Omega\subset\R^n$ be a bounded open set and let $\varphi_h\in \Po_\Omega(p,\lambda,\Lambda)$ for $h\in\N\cup\{\infty\}$. 
We define the functionals $F_h\colon L^p(\Omega)\to [0,\infty]$ by
\begin{equation}\label{eq:Fh}
F_h(u)\coloneqq 
\begin{cases}
\displaystyle \int_\Omega\varphi_h(x,\nabla u(x))\,\d x & \quad \text{if } u\in W_0^{1,p}(\Omega),\\
\infty & \quad \text{otherwise}.
\end{cases}
\end{equation}

\subsection{The class of nonlocal functionals} 

In analogy with the case of nonlocal $H$-convergence (see Section~\ref{sect_nonlocal_H}), we introduce a further class of conservative monotone maps $A$ defined on the whole $\R^n\times\R^n$, in order to formulate the $\Gamma$-convergence problem in the nonlocal setting.

\begin{definition}\label{def:class_conservative_fields}
Let $\Omega\subset\R^n$ be a bounded open set and let $A_0\in\M_{\R^n}^{\rm grad}(p,\lambda,\Lambda)$ satisfy~\eqref{eq:A0}. We set
\[
\M_\Omega^{\rm grad}(p,\lambda,\Lambda,A_0)\coloneqq\{A \in \M_{\R^n}^{\rm grad}(p,\lambda,\Lambda): A=A_0\text{ in $(\R^n\setminus\Omega)\times\R^n$}\}.
\]
\end{definition}

\begin{remark}
As already observed in Remark~\ref{rem:A0}, a possible choice of $A_0\in \M_{\R^n}^{\rm grad}(p,\lambda,\Lambda)$ satisfying~\eqref{eq:A0} is
\[
A_0(\xi)\coloneqq|\xi|^{p-2}\xi\quad\text{for every }\xi\in\R^n.
\]
\end{remark}

As in the local case, for every $A\in\M_\Omega^{\rm grad}(p,\lambda,\Lambda,A_0)$, we define $\varphi_A\colon \R^n\times\R^n\to \R$ as in~\eqref{eq:varphiA}. 
This leads to the following class of potentials.

\begin{definition}
Let $\Omega\subset\R^n$ be a bounded open set and let $A_0\in\M_{\R^n}^{\rm grad}(p,\lambda,\Lambda)$ satisfy~\eqref{eq:A0}. We set
\begin{align*}
\Po_\Omega(p,\lambda,\Lambda,A_0)\coloneqq\{\varphi\colon\R^n\times \R^n\to \R \text{ Borel}:\varphi=\varphi_A\text{ for some }A\in \M_\Omega^{\rm grad}(p,\lambda,\Lambda,A_0)\}.
\end{align*}
\end{definition}

Arguing as in Proposition~\ref{prop:density-loc}, and taking into account Remark~\ref{rem:A0}(b) we deduce that a potential $\varphi\in\Po_\Omega(p,\lambda,\Lambda,A_0)$ satisfies the following growth and continuity properties.

\begin{proposition}\label{prop:density-nonlocal}
Let $\Omega \subset \R^n$ be a bounded open set and let $A_0\in\M_{\R^n}^{\rm grad}(p,\lambda,\Lambda)$ satisfy~\eqref{eq:A0}. Then, there exist positive constants $a_0,a_1,a_2,a_3$ and $\alpha\in(0,1)$, depending only on $p$, $\lambda$, $\Lambda$, and $A_0$, such that for every $\varphi\in\Po_\Omega(p,\lambda,\Lambda,A_0)$, for a.e.\ $x\in\mathbb{R}^n$, and for every $\xi,\eta\in\R^n$ the following hold:
\begin{itemize}
\item[\rm{(NP1)}] the function $\varphi(x,\,\cdot\,)$ is of class $C^1(\R^n)$ and convex, $\varphi(x,0)=0$, and $\nabla_\xi\varphi(x,0)=0$,
\item[\rm{(NP2)}] $a_0|\xi|^p\le \varphi(x,\xi)\le a_1(\mathbf{1}_\Omega(x)+|\xi|^p)$,
\item[\rm{(NP3)}] $|\varphi(x,\xi)-\varphi(x,\eta)|\le a_2(\mathbf{1}_\Omega(x)+|\xi|+|\eta|)^{p-1}|\xi-\eta|$,
\item[\rm{(NP4)}] $|\nabla_\xi \varphi(x,\xi)-\nabla_\xi \varphi(x,\eta)|\le a_3(\mathbf{1}_\Omega(x)+|\xi|+|\eta|)^{p-1-\alpha}|\xi-\eta|^\alpha$.
\end{itemize} 
\end{proposition}

We now introduce the nonlocal integral functionals, which are the nonlocal counterparts of the local functionals defined in~\eqref{eq:Fh}. 
Let $\Omega\subset\R^n$ be a bounded open set and let $A_0\in\M_{\R^n}^{\rm grad}(p,\lambda,\Lambda)$ satisfy~\eqref{eq:A0}. Let $\varphi_h\in \Po_\Omega(p,\lambda,\Lambda,A_0)$ for $h\in\N\cup\{\infty\}$. We define
$F^s_h\colon L^p(\R^n)\to [0,\infty]$ by
\begin{equation}\label{eq:Fhs}
F^s_h(u)\coloneqq 
\begin{cases}
\displaystyle \int_{\R^n}\varphi_h(x,\nabla^s u(x))\, \d x&\quad\text{if $u\in H^{s,p}_0(\Omega)$},\\
\infty&\quad\text{otherwise}.
\end{cases}
\end{equation}

\begin{remark}
Let us emphasise the difference between the local functionals $F_h$ and their nonlocal counterparts $F_h^s$. The former are defined on $L^p(\Omega)$ and involve integration over $\Omega$, whereas the latter are defined on $L^p(\mathbb{R}^n)$, with integration over $\mathbb{R}^n$. 
Moreover, by properties (P2) and (NP2), the functionals $F_h$ and $F_h^s$ are finite on $W_0^{1,p}(\Omega)$ and $H_0^{s,p}(\Omega)$, respectively.
\end{remark}


\section{Equivalence between \texorpdfstring{$\Gamma$}{Gamma}-convergence of local and nonlocal energies}\label{sec:Gamm-equiv}

As in Section~\ref{sec:H-equiv}, given $A_0 \in \M_{\R^n}^{\rm grad}(p,\lambda,\Lambda)$ satisfying~\eqref{eq:A0} and $\varphi \in \Po_\Omega(p,\lambda,\Lambda,A_0)$, we simply write $\varphi \in \Po_\Omega(p,\lambda,\Lambda)$ when referring to its restriction to $\Omega \times \R^n$. Conversely, if $\varphi \in \Po_\Omega(p,\lambda,\Lambda)$, we write $\varphi \in \Po_\Omega(p,\lambda,\Lambda,A_0)$ to indicate that $\varphi$ is extended to the whole $\R^n \times \R^n$ by setting $\varphi=\varphi_{A_0}$ on $(\R^n\setminus \Omega)\times \R^n$; see~\eqref{eq:fA}.

In this section we show the equivalence between the $\Gamma$-convergence of the class of local and nonlocal energy functionals associated with conservative monotone operators, respectively introduced in~\eqref{eq:Fh} and~\eqref{eq:Fhs}.
The main result of this section is the following theorem. 
\begin{theorem}\label{thm:Gammas} 
Let $\Omega\subset\R^n$ be a bounded open set. Let $A_0\in \M_{\R^n}^{\rm grad}(p,\lambda,\Lambda)$ satisfy~\eqref{eq:A0}, and for every for $h\in \N\cup\{\infty\}$ let $\varphi_h\in\Po_\Omega (p,\lambda,\Lambda,A_0)$, and let $F_h$ and $F_h^s$ be the functionals defined in~\eqref{eq:Fh} and~\eqref{eq:Fhs}, respectively. Then
\[
\text{$F_h^s$ $\Gamma$-converges to $F^s_\infty$ strongly in $L^p(\R^n)$ as $h\to\infty$},
\]
if and only if 
\[
\text{$F_h$ $\Gamma$-converges to $F_\infty$ strongly in $L^p(\Omega)$ as $h\to\infty$}.
\]
\end{theorem}

The proof of Theorem~\ref{thm:Gammas} follows the same scheme adopted in Section~\ref{sec:H-equiv} for $H$-convergence. 
We first show that the $\Gamma$-convergence in the local setting implies $\Gamma$-convergence in the nonlocal one. We then prove the converse implication by means of a compactness argument combined with a result on unique integral representation. 
To prove the first implication, we need some preliminary results on $\Gamma$-convergence in the local setting. 
We begin with the following lemma.

\begin{lemma}\label{lem:uniq_rappr}
Let $\phi_1,\phi_2\colon\R^n\to\R$ be two convex functions satisfying
\[
\phi_1(0)=\phi_2(0)=0.
\]
Assume that
\begin{equation}\label{eq:integral_identity}
\int_{\R^n}\phi_1(\nabla u(x))\,\d x = \int_{\R^n}\phi_2(\nabla u(x))\,\d x \quad \text{for every $u\in C_c^\infty(\R^n)$}.
\end{equation}
Then there exists $a\in \mathbb{R}^n$ such that
\begin{equation*}
\phi_1(\xi)-\phi_2(\xi)=a\cdot\xi\quad\text{for every $\xi\in\R^n$}.
\end{equation*}
\end{lemma}

\begin{proof}
Let $\psi\coloneqq \phi_1-\phi_2\colon \mathbb{R}^n\to \mathbb{R}$. 
Since $\phi_1,\phi_2$ are convex functions, and therefore locally Lipschitz, $\psi$ is a locally Lipschitz function.
Moreover, by~\eqref{eq:integral_identity} we have
\begin{equation}\label{eq:unica_1}
\int_{\R^n}\frac{\psi(\nabla u(x)+\varepsilon\nabla v(x))-\psi(\nabla u(x))}{\varepsilon}\,\d x = 0
\end{equation}
for every $u,v\in C_c^\infty(\mathbb{R}^n)$ and $\varepsilon \in (0,1)$. Given $u,v\in C_c^\infty(\mathbb{R}^n)$, we set
\[
R\coloneqq\sup_{\varepsilon\in(0,1)} \|\nabla u+\varepsilon\nabla v\|_{L^\infty(\R^n)}.
\]
Since $\psi$ is locally Lipschitz, there exists a constant $L>0$, depending only on $R$, such that
\[
|\psi(\xi)-\psi(\eta)|\le L|\xi-\eta|
\quad\text{for every $\xi,\eta\in B_R$}.
\]
Therefore, for every $\varepsilon\in(0,1)$ it holds that
\[
\left|\frac{\psi(\nabla u(x)+\varepsilon\nabla v(x))-\psi(\nabla u(x))}{\varepsilon}\right|\le L|\nabla v(x)|\quad\text{for every $x\in \R^n$}.
\]
Let $v\in C_c^\infty(\mathbb{R}^n)$, and let $w\in C_c^\infty(\mathbb{R}^n)$ be such that $w=1$ in a neighbourhood of the support of $v$.
Moreover, let us choose
\[
u(x)\coloneqq \frac{|x|^2w(x)}{2}.
\]
Then, 
$u\in C_c^\infty(\mathbb{R}^n)$ and $\nabla u(x)=x$ in the support of $v$.
Hence, substituting $u$ in~\eqref{eq:unica_1}, we get
\begin{equation}\label{eq:unica_2}
\int_{\R^n}\frac{\psi(\nabla u(x)+\varepsilon\nabla v(x))-\psi(\nabla u(x))}{\varepsilon}\,\d x =\int_{\operatorname{supp}(v)}\frac{\psi(x+\varepsilon\nabla v(x))-\psi(x)}{\varepsilon}\,\d x = 0.
\end{equation}
By Rademacher's theorem, $\psi$ is differentiable a.e. in $\mathbb{R}^n$.
Hence, as $\varepsilon\to 0$
\[
\frac{\psi(x+\varepsilon\nabla v(x))-\psi(x)}{\varepsilon}\to \nabla\psi(x)\cdot \nabla v(x)\quad\text{for a.e.\ $x\in \R^n$}.
\]
Since $L|\nabla v|\in L^1(\R^n)$, by passing to the limit in~\eqref{eq:unica_2} as $\varepsilon\to 0$ and applying the dominated convergence theorem, we obtain
\[
\int_{\R^n}\nabla\psi(x) \cdot \nabla v(x)\,\d x =0.
\]
By the arbitrariness of $v\in C^\infty_c(\R^n)$, we infer that $\psi$ is harmonic (in the distributional sense) on $\R^n$. Hence, by Weyl's lemma, $\psi$ is an analytic harmonic function on $\R^n$. In particular, $\psi\in C^\infty(\R^n)$.
As a consequence, we can now pass to the limit in~\eqref{eq:unica_1} as $\varepsilon\to 0$ to get
\begin{equation}\label{eq:Weyl}
\int_{\R^n}\nabla\psi(\nabla u(x)) \cdot \nabla v(x)\,\d x =0\quad\text{for every $u,v\in C_c^\infty(\mathbb{R}^n)$}.
\end{equation}
Let $u,v,w \in C^\infty_c(\R^n)$ and  $\varepsilon\in(0,1)$ be fixed. 
By applying~\eqref{eq:Weyl} to both $u(x)+\varepsilon w(x)\in C^\infty_c(\R^n)$ and $u(x)$, we get
\begin{equation}\label{eq:unica_3}
\int_{\R^n}\frac{\nabla\psi(\nabla u(x)+\varepsilon\nabla w(x))\cdot \nabla v(x)-\nabla\psi(\nabla u(x))\cdot \nabla v(x)}{\varepsilon}
\,\d x=0.
\end{equation}
Since $\psi\in C^\infty(\R^n)$, if we set 
\[
R'\coloneqq \sup_{\varepsilon\in(0,1)} \|\nabla u+\varepsilon\nabla w \|_{L^\infty(\R^n)},
\]
there exists a positive $L'>0$, depending only on $R'$, such that
\[
|\nabla\psi(\xi)-\nabla \psi(\eta)|\le L'|\xi-\eta|\quad\text{for every $\xi,\eta\in B_{R'}$}.
\]
Hence, for every $\varepsilon\in(0,1)$ we have
\begin{equation*}
\left|\frac{\nabla\psi(\nabla u(x)+\varepsilon\nabla w(x))\cdot \nabla v(x) -\nabla\psi(\nabla u(x))\cdot \nabla v(x)}{\varepsilon}
\right| \le L'|\nabla w(x)||\nabla v(x)|\quad\text{for every $x\in\R^n$}.
\end{equation*}
Since $L'|\nabla w||\nabla v|\in L^1(\R^n)$, by passing to the limit in~\eqref{eq:unica_3} as $\varepsilon\to 0$, and by the dominated convergence theorem, we obtain
\begin{equation}\label{eq:unica_4}
\int_{\R^n}
\nabla^2\psi(\nabla u(x))\nabla w(x)\cdot \nabla v(x)
\,\d x=0\quad\text{for every $u,v,w\in C^\infty_c(\R^n)$},
\end{equation}
where $\nabla^2\psi$ is the Hessian matrix of $\psi$.
Let us take $w=v\in C^\infty_c(\R^n)$ and consider $u\in C^\infty_c(\R^n)$ such that $u(x)=\frac{|x|^2}{2}$ on the support of $v$. Substituting into~\eqref{eq:unica_4}, we obtain
\begin{equation}\label{eq:identita_quadratica}
\int_{\R^n}
\nabla^2\psi(x)\nabla v(x)\cdot \nabla v(x)
\,\d x = 0
\quad \text{for every $v\in C^\infty_c(\R^n)$}.
\end{equation}
We recall that, for every $i\in\{1,2\}$, the second-order distributional derivative $D^2\phi_i$ of the convex function $\phi_i$ is a signed Radon measure (see~\cite[Theorem~6.8]{EG}). In particular, we have
\begin{equation*}
\nabla^2\psi
=\nabla^2\phi_1-\nabla^2\phi_2\quad\text{a.e.\ in $\R^n$},
\end{equation*}
where $\nabla^2\phi_i\in L^1_{\rm loc}(\R^n;\R^{n\times n}_{\rm sym})$ denotes the absolutely continuous part of the measure $D^2\phi_i$ for $i\in\{1,2\}$. We claim that
\begin{align}\label{eq:der-sec}
\nabla^2\phi_i(x)\xi\cdot\xi\ge 0
\quad\text{for a.e.\ $x\in\R^n$ and every $\xi\in\R^n$}.
\end{align}
To see this, fix $\xi\in\R^n$. By the construction made in~\cite[Theorem~6.8]{EG}, the measure $D^2\phi_i\xi\cdot\xi$ is nonnegative. Hence, its absolutely continuous part is a nonnegative function, which coincides a.e.\ with $\nabla^2\phi_i\xi\cdot\xi$.

By combining~\eqref{eq:identita_quadratica} and~\eqref{eq:der-sec}, and applying~\cite[Lemma~22.5]{DalMaso} with $A=\nabla^2\phi_1$ and $B=\nabla^2\phi_2$, we deduce that $\nabla^2\phi_1=\nabla^2\phi_2\quad\text{a.e. in }\mathbb{R}^n$. In particular, 
\[
\nabla^2\psi=0\quad\text{a.e. in $\R^n$}.
\]
Hence, the function $\psi$ is affine, i.e.,
\[
\psi(\xi)=a\cdot\xi+b\quad\text{for every $\xi\in\R^n$}.
\]
As $\psi(0)=0$, we have $b=0$, and the thesis follows.
\end{proof}

As a consequence of Lemma~\ref{lem:uniq_rappr} and a blow-up argument, we deduce the uniqueness of the integral representation of both the local and nonlocal functionals $F_h$ and $F^s_h$ in~\eqref{eq:Fh} and~\eqref{eq:Fhs}.

\begin{proposition}\label{prop:uniq_rapp}
Let $\Omega \subset \R^n$ be an open set and let $\varphi_1, \varphi_2 \colon \Omega \times \R^n \to \R$ be two Borel functions satisfying for every $i \in \{1,2\}$ and a.e.\ $x\in\Omega$:
\begin{itemize}
\item[\emph{(i)}] the function $\varphi_i(x,\,\cdot\,)$ is of class $C^1(\R^n)$ and convex, $\varphi_i(x,0) = 0$, and $\nabla_\xi \varphi_i(x,0) = 0$,
\item[\emph{(ii)}] there exist $q \in [1,\infty)$ and a positive constant $C$, independent of $x$ and $\xi$, such that 
\[
|\varphi_i(x,\xi)| \le C(|\xi| + |\xi|^q) \quad \text{for every $i\in\{1,2\}$, a.e.\ $x \in \Omega$, and every $\xi \in \R^n$}.
\]
\end{itemize}
Let $\sigma \in (0,1]$ be fixed and assume that
\begin{equation}\label{eq:integral_identity_2}
\int_\Omega \varphi_1(x, \nabla^\sigma u(x))\,\d x
=
\int_\Omega \varphi_2(x, \nabla^\sigma u(x))\,\d x
\quad \text{for every $u \in C^\infty_c(\Omega)$},
\end{equation}
with the convention $\nabla^1 = \nabla$. Then,
\begin{equation}\label{eq:varphi_indentity}
\varphi_1(x,\xi) = \varphi_2(x,\xi) \quad \text{for a.e.\ $x \in \Omega$ and for every $\xi \in \R^n$}.
\end{equation}
\end{proposition}

\begin{proof}
Let us introduce the Banach space
\begin{equation*}
    \mathcal X \coloneqq H^{\sigma,1}(\R^n) \cap H^{\sigma,q}(\R^n),
\end{equation*}
endowed with the norm
\begin{equation*}
    \|u\|_{\mathcal X} \coloneqq \|u\|_{H^{\sigma,1}(\R^n)} + \|u\|_{H^{\sigma,q}(\R^n)} \quad \text{for every } u \in \mathcal X,
\end{equation*}
with the convention that $H^{1,q}(\R^n) = W^{1,q}(\R^n)$ for $q\in [1,\infty)$. Let $\mathcal{D} \subset C^\infty_c(\R^n)$ be a countable dense subset of $\mathcal X$.
We use a blow-up argument. 
Let $x_0\in\Omega$ be such that assumptions (i)–(ii) hold, and such that $x_0$ is a Lebesgue point of the functions $x \mapsto \varphi_i(x,\nabla^\sigma u(x))$ for every $i \in \{1,2\}$ and every $u \in \mathcal{D}$. Since $\mathcal D$ is countable, almost every point of $\Omega$ enjoys this property.
To conclude, we show that~\eqref{eq:varphi_indentity} holds at $x_0$.
Let $u \in \mathcal{D}$ and let $r>0$ be fixed. 
We define
\begin{equation*}
    u_{x_0,r}(x) \coloneqq r^{-\sigma} u(r(x-x_0)) \quad \text{for every $x \in \R^n$}.
\end{equation*}
By performing a change of variables, for every $i \in \{1,2\}$ we get
\begin{align*}
r^n \int_\Omega \varphi_i(x, \nabla^\sigma u_{x_0,r}(x))\,\d x
=\int_{\Omega_{x_0,r}} \varphi_i\left(x_0 + \frac{y}{r}, \nabla^\sigma u(y)\right) \d y,
\end{align*}
where
\[
\Omega_{x_0,r} \coloneqq \left\{y \in \R^n : x_0 + \frac{y}{r} \in \Omega\right\}.
\]
Note that $u_{x_0,r} \in C^\infty_c(\Omega)$ for $r>0$ sufficiently large. Therefore, by~\eqref{eq:integral_identity_2}, for every $u \in \mathcal{D}$ and for every $r>0$ sufficiently large we have
\begin{equation}\label{eq:blow-up_0}
    \int_{\Omega_{x_0,r}} \varphi_1\left(x_0+\frac{y}{r}, \nabla^\sigma u(y)\right) \d y
    = \int_{\Omega_{x_0,r}} \varphi_2\left(x_0+\frac{y}{r}, \nabla^\sigma u(y)\right) \d y.
\end{equation}
In order to pass to the limit in~\eqref{eq:blow-up_0}, we observe that, for every $r>0$, by (ii),
\begin{equation*}
    \left|\varphi_i\left(x_0 + \frac{y}{r}, \nabla^\sigma u(y)\right)\right| \le C \left(|\nabla^\sigma u(y)| + |\nabla^\sigma u(y)|^q \right) \quad \text{for a.e.\ $y\in \R^n$}.
\end{equation*}
Since $C(|\nabla^\sigma u| + |\nabla^\sigma u|^q)\in L^1(\R^n)$ because $u\in\mathcal D\subset\mathcal X$, by letting $r \to \infty$ in~\eqref{eq:blow-up_0}, we obtain
\begin{equation*}
    \int_{\R^n} \varphi_1(x_0, \nabla^\sigma u(y))\,\d y
    = \int_{\R^n} \varphi_2(x_0, \nabla^\sigma u(y))\,\d y
    \quad \text{for every } u \in \mathcal{D}.
\end{equation*}
Owing to (i)–(ii), the mappings $u\mapsto\int_{\mathbb R^n}\varphi_i(x_0,\nabla^\sigma u)$ are continuous on $\mathcal X$. 
Hence, the identity extends from $\mathcal D$ to $\mathcal X$, 
and we conclude that
\begin{equation}\label{eq:blow-up_1}
    \int_{\R^n} \varphi_1(x_0, \nabla^\sigma u(y))\,\d y
    = \int_{\R^n} \varphi_2(x_0, \nabla^\sigma u(y))\,\d y
    \quad \text{for every } u \in \mathcal X.
\end{equation}

If $\sigma = 1$, by Lemma~\ref{lem:uniq_rappr}, applied to $\phi_i = \varphi_i(x_0,\,\cdot\,)$ for $i\in\{1,2\}$, there exists $a\in\mathbb{R}^n$ such that
\begin{equation*}
\varphi_1(x_0,\xi)-\varphi_2(x_0,\xi)=a\cdot\xi\quad\text{for every $\xi\in\R^n$}.
\end{equation*}
In particular, by (i)
\[
a=\nabla_\xi\varphi_1(x_0,0)-\nabla_\xi\varphi_2(x_0,0)=0,
\]
and the thesis follows.

Let $\sigma \in (0,1)$, and let $v\in C_c^\infty(\R^n)$. 
By Proposition~\ref{prop:equiv}(ii), the function $u \coloneqq (-\Delta)^\frac{1-\sigma}{2}v$ belongs to $\mathcal X$ and satisfies
\begin{equation*}
    \nabla^\sigma u = \nabla v\quad\text{in $\R^n$}.
\end{equation*}
Hence, by recalling~\eqref{eq:blow-up_1}, we obtain
\begin{equation*}
    \int_{\R^n} \varphi_1(x_0,\nabla v(y))\,\d y
    = \int_{\R^n} \varphi_2(x_0,\nabla v(y))\,\d y
    \quad \text{for every $v \in C^\infty_c(\R^n)$}.
\end{equation*}
Since $\nabla_\xi\varphi_i(x_0,0)=0$ for $i\in\{1,2\}$ by (i), we can apply again Lemma~\ref{lem:uniq_rappr} to conclude that 
\[
\varphi_1(x_0,\xi) = \varphi_2(x_0,\xi)\quad\text{for every $\xi \in \R^n$}.
\]
Hence,~\eqref{eq:varphi_indentity} holds at $x_0$.
Since $x_0$ was arbitrary outside a negligible set, the thesis follows.
\end{proof}

For the next result, we introduce a localised version of the functionals $F_h$ defined in~\eqref{eq:Fh}. 
Let $\Omega\subset\R^n$ be a bounded open set and let $\varphi_h\in \Po_\Omega(p,\lambda,\Lambda)$ for $h\in\N\cup\{\infty\}$. 
We define the functionals $G_h\colon L^p(\Omega)\times\mathscr A(\Omega)\to [0,\infty]$ by
\begin{equation}\label{eq:GhU}
G_h(u,U)\coloneqq 
\begin{cases}
\displaystyle \int_U\varphi_h(x,\nabla u(x))\,\d x & \quad \text{if $U\in \mathscr A(\Omega)$ and $u\in W^{1,p}(U)$},\\
\infty & \quad \text{otherwise}.
\end{cases}
\end{equation}

\begin{remark}
The local functionals $F_h$ defined in~\eqref{eq:Fh} incorporate the homogeneous boundary condition into their domain, whereas the functionals $G_h$ do not. Introducing energies without prescribed boundary conditions is technically advantageous when passing from the nonlocal to the local setting. Indeed, given $u\in H^{s,p}_0(\Omega)$, the function $v$ provided by Proposition~\ref{prop:equiv}, belongs to $W^{1,p}_{\rm loc}(\mathbb R^n)$ but it does not satisfy any prescribed boundary condition on $\Omega$. Furthermore, by considering the localized functionals $G_h$ on open subsets of $\Omega$, we can directly apply the local results established in~\cite{ADMZ}. For these reasons, it will be more convenient to work with the more flexible family of functionals $G_h$.
\end{remark}

We recall the following compactness result, which can be deduced by~\cite{ADMZ}.

\begin{proposition}\label{prop:compactness}
Let $\Omega\subset\R^n$ be a bounded open set. 
Let $\varphi_h\colon \Omega\times\R^n \to [0,\infty)$, for $h\in\N$, be a sequence of Borel functions satisfying \emph{(P1)}, \emph{(P2)}, \emph{(P3)}, \emph{(P4)} of Proposition~\ref{prop:density-loc} for some positive constants $a_0,a_1,a_2,a_3$ and $\alpha\in(0,1)$ independent of $h$. 
Then, there exist a subsequence $(\varphi_{h_k})_k$ and a Borel function $\varphi_\infty\colon \Omega\times\R^n \to [0,\infty)$ satisfying \emph{(P1)}, \emph{(P2)}, \emph{(P3)} of Proposition~\ref{prop:density-loc} for the same constants $a_0,a_1,a_2$, such that
\begin{equation}\label{eq:Ansini_Zeppieri}
    \text{$G_{h_k}(\,\cdot\,,U)$ $\Gamma$-converges to $G_\infty(\,\cdot\,,U)$ strongly in $L^p(\Omega)$ as $k\to\infty$ for every $U\in\mathscr{A}(\Omega)$},
\end{equation}
where $G_h\colon L^p(\Omega)\times\mathscr{A}(\Omega)\to [0,\infty]$ denotes the functional associated with $\varphi_h$ for $h\in\N\cup\{\infty\}$, defined as in~\eqref{eq:GhU}.
\end{proposition}

\begin{proof}
Since $\varphi_h$ satisfies (P2) and (P3) for every $h\in\N$,~\cite[Theorem~20.4]{DalMaso} yields the existence of a subsequence $(\varphi_{h_k})_k$ and a Borel function $\varphi_\infty\colon \Omega\times\R^n\to[0,\infty)$ satisfying (P2) and (P3), such that~\eqref{eq:Ansini_Zeppieri} holds.
In particular, the function $\varphi_\infty(x,\,\cdot\,)$ is convex for a.e.\ $x\in\Omega$. 

Moreover, since $\varphi_{h}$ satisfies (P1) and (P4) for every $h\in\N$, it follows from~\cite[Theorem~2.8]{ADMZ} that $\varphi_\infty(x,\,\cdot\,)\in C^1(\R^n)$ for a.e.\ $x\in\Omega$. 
Furthermore, since $\varphi_{h}$ satisfies (P1), we have $G_{h_k}(0,\Omega)=0$ for every $k\in\mathbb{N}$. 
Hence, by ($\Gamma$1) in Definition~\ref{def:Gamma},
\[
G_\infty(0,\Omega)\le0.
\]
On the other hand, since $\varphi_\infty$ satisfies (P2), we have $G_\infty(0,\Omega)\ge0$. Therefore,
\[
G_\infty(0,\Omega)=0.
\]
Hence, $\varphi_\infty(x,0)=0$ for a.e. $x\in\Omega$.
Finally, by (4.10a) in~\cite[Theorem~4.5]{ADMZ}, we have 
\[
0=\nabla_\xi \varphi_{h_k}(\,\cdot\,,0) \to \nabla_\xi \varphi_\infty(\,\cdot\,,0) \quad \text{weakly in $L^{p'}(\Omega;\R^n)$}.
\]
Hence, $\varphi_\infty$ satisfies (P1).
\end{proof}

As a consequence of the compactness result Proposition~\ref{prop:compactness}, by exploiting~\cite[Theorem~4.5]{ADMZ}, we can deduce that the class of localised functionals $G_h$, associated with $\Po_\Omega(p,\lambda,\Lambda)$, is compact with respect to the $\Gamma$-convergence.

\begin{corollary}\label{coro:compactness}
Let $\Omega\subset\R^n$ be a bounded open set and for every $h\in\N$ let 
$\varphi_h\in\Po_\Omega(p,\lambda,\Lambda)$.
Then, there exist a subsequence $(\varphi_{h_k})_k\subset\Po_\Omega(p,\lambda,\Lambda)$ and a Borel function 
$\varphi_\infty\in \Po_\Omega(p,\lambda,\Lambda)$ such that~\eqref{eq:Ansini_Zeppieri} holds.
\end{corollary}

\begin{proof}
Let $\varphi_h\in\Po_\Omega(p,\lambda,\Lambda)$ for every $h\in\mathbb N$.
By Proposition~\ref{prop:density-loc} and 
Proposition~\ref{prop:compactness}, there exist a subsequence $(\varphi_{h_k})_k$ and a Borel function $\varphi_\infty$ satisfying (P1), (P2), (P3) of Proposition~\ref{prop:density-loc} with the same constants $a_0,a_1,a_2$, and such that~\eqref{eq:Ansini_Zeppieri} holds.
To conclude, it remains to prove that $\varphi_\infty\in\Po_\Omega(p,\lambda,\Lambda)$. 
To this end, define
\begin{equation}\label{eq:def_A_infty}
    A_\infty\coloneqq\nabla_\xi \varphi_\infty.
\end{equation}
It is therefore enough to show that
\[
A_\infty\in \M^{\rm grad}_\Omega(p,\lambda,\Lambda).
\]
Let $A_h\coloneqq \nabla_\xi \varphi_h\in \M_\Omega^{\rm grad}(p,\lambda,\Lambda)$ for every $h\in\N$.
Without loss of generality, we may assume that the $\Gamma$-convergence in~\eqref{eq:Ansini_Zeppieri} holds for the whole sequence $G_h$, namely
\begin{equation}\label{eq:Ansini_Zeppieri_Omega}
    \text{$G_h(\,\cdot\,,U)$ $\Gamma$-converges to $G_\infty(\,\cdot\,,U)$ strongly in $L^p(\Omega)$ as $h\to\infty$ for every $U\in\mathscr{A}(\Omega)$}.
\end{equation}
In particular,~\eqref{eq:Ansini_Zeppieri_Omega} holds for $U=\Omega$.
Then, by~\cite[Theorem~21.1]{DalMaso},
\begin{equation}\label{eq:Fh-gamma}
\text{$F_h$ $\Gamma$-converges to $F_\infty$ strongly in $L^p(\Omega)$ as $h\to\infty$},
\end{equation}
where $F_h$ is defined as in~\eqref{eq:Fh} for every $h\in \N\cup\{\infty\}$. 

We claim that $(F_h)_h$ are uniformly convex. Indeed, thanks to Proposition~\ref{prop:potential} and (M2), for a.e.\ $x\in\Omega$ and every $\xi,\eta\in\R^n$ we have
\begin{align*}
\varphi_h(x,\xi)- \varphi_h(x,\eta)-A_h(x,\eta)\cdot(\xi-\eta)&=\int_0^1 (A_h(x,\eta+t(\xi-\eta))-A_h(x,\eta))\cdot (\xi-\eta)\,{\rm d}t\\
&\ge \frac{\lambda}{p}|\xi-\eta|^p.
\end{align*}
Hence, for a.e.\ $x\in\Omega$, for every $\xi,\eta,\zeta\in\R^n$, and for every $t\in [0,1]$ we have
\begin{align*}
t\varphi_h(x,\xi)- t\varphi_h(x,\zeta)&\ge tA_h(x,\zeta)\cdot(\xi-\zeta)+\frac{\lambda}{p}t|\xi-\zeta|^p,\\
(1-t)\varphi_h(x,\eta)- (1-t)\varphi_h(x,\zeta)&\ge (1-t)A_h(x,\zeta)\cdot(\eta-\zeta)+\frac{\lambda}{p}(1-t)|\eta-\zeta|^p.
\end{align*}
By summing the two inequalities and choosing $\zeta=t\xi+(1-t)\eta$, we get
\[
t\varphi_h(x,\xi)+(1-t)\varphi_h(x,\eta)\ge \varphi_h(x,t\xi+(1-t)\eta) + \frac{\lambda}{p}t(1-t)[t^{p-1}+(1-t)^{p-1}]|\xi-\eta|^p.
\]
For every $u,v\in W^{1,p}(\Omega)$, by taking $\xi=\nabla u(x)$ and $\eta=\nabla v(x)$ and integrating over $x\in \Omega$, for every $t\in [0,1]$ we get
\[
F_h(tu+(1-t)v) + \frac{\lambda}{p}[t^p(1-t)+t(1-t)^p]\|\nabla u-\nabla v\|_{L^p(\Omega)}^p\le tF_h(u)+(1-t)F_h(v).
\]
Thus, by Poincaré's inequality, for every $u,v\in L^p(\Omega)$ and $t\in [0,1]$ we deduce
\begin{equation}\label{eq:Fh-convexity}
F_h(tu+(1-t)v) + C[t^p(1-t)+t(1-t)^p]\|u-v\|_{L^p(\Omega)}^p\le tF_h(u)+(1-t)F_h(v),
\end{equation}
where $C$ is a positive constant depending on $n$, $p$, $\Omega$, and $\lambda$. 

By exploiting~\eqref{eq:Fh-gamma} and passing to the limit as $h\to\infty$ in~\eqref{eq:Fh-convexity}, we conclude that for every $u,v\in L^p(\Omega)$ and $t\in [0,1]$
\[
F_\infty(tu+(1-t)v) + C[t^p(1-t)+t(1-t)^p]\|u-v\|_{L^p(\Omega)}^p\le tF_\infty(u)+(1-t)F_\infty(v),
\]
that is, $F_\infty$ is uniformly convex. 
Let $f\in W^{-1,p'}(\Omega)$ be fixed, and set
\begin{equation*}
F_h^f(u)\coloneqq 
\begin{cases}
F_h(u)-\langle f,u\rangle_{W^{-1,p'}(\Omega)\times W^{1,p}(\Omega)}&\text{if $u\in W^{1,p}_0(\Omega)$},\\
\infty &\text{otherwise}.
\end{cases}
\end{equation*}
By the lower bound in (P2) we obtain that
\[
\text{$F_h^f$ $\Gamma$-converges to $F_\infty^f$ strongly in $L^p(\Omega)$ as $h\to\infty$}.
\]
Since the functionals $(F_h^f)_h$ are equi-coercive on $L^p(\Omega)$, denoting by $u_h^f$ the unique minimiser of $F_h^f$ for every $h\in\mathbb{N}\cup\{\infty\}$, by~\cite[Corollary~7.24]{DalMaso} we have
\begin{equation}\label{eq:convergenza_minimi}
u^f_h\to u^f_\infty\quad\text{weakly in $W^{1,p}_0(\Omega)$ and strongly in $L^p(\Omega)$ as $h\to\infty$}.
\end{equation}
Moreover,
\begin{equation*}
F_h^f(u_h^f) \to F^f_\infty(u^f_\infty)\quad\text{as $h\to\infty$},
\end{equation*}
which implies that
\[
G_h(u_h^f,\Omega)\to G_\infty(u^f_\infty,\Omega)\quad\text{as $h\to\infty$}.
\]
Hence, by~\cite[Theorem~4.5]{ADMZ}, it holds that
\begin{equation}\label{eq:convergenza_momenti}
A_h(\,\cdot\,,\nabla u_h^f)\to A_\infty(\,\cdot\,,\nabla u_\infty^f)\quad\text{weakly in $L^{p'}(\Omega;\R^n)$ as $h\to\infty$}.
\end{equation}
Note that for every $h\in\N\cup\{\infty\}$ the function $u_h^f$ is the unique weak solution of
\begin{equation*}
\begin{cases}
-\div(A_h(\,\cdot\,,\nabla u_h^f))=f&\text{in $\Omega$},\\
u_h^f=0&\text{on }\partial\Omega,
\end{cases}
\end{equation*}
which is the Euler–Lagrange equation associated with the functional $F_h^f$ (the uniqueness follows from the uniform convexity of $F_h^f$ for every $h\in\N\cup\{\infty\}$). 
Therefore,~\eqref{eq:convergenza_minimi} and~\eqref{eq:convergenza_momenti} imply that
\[
\mathcal A_h\coloneqq-\operatorname{div}(A_h(\,\cdot\,,\nabla\,))\text{ $H$-converges to }\mathcal A_\infty\coloneqq-\operatorname{div}(A_\infty(\,\cdot\,,\nabla\,))\text{ as $h\to\infty$}.
\]
Hence, Proposition~\ref{prop:Hcomploc} yields $A_\infty\in\M_\Omega(p,\lambda,\Lambda)$. Since $A_\infty=\nabla_\xi\varphi_\infty$ by~\eqref{eq:def_A_infty}, we conclude that $A_\infty\in\M_\Omega^{\rm grad}(p,\lambda,\Lambda)$. Therefore, by definition of $\Po_\Omega(p,\lambda,\Lambda)$, $\varphi_\infty=\varphi_{A_\infty}\in\Po_\Omega(p,\lambda,\Lambda)$. 
\end{proof}

As a consequence of Proposition~\ref{prop:uniq_rapp} and Corollary~\ref{coro:compactness}, we obtain the following nonlinear counterpart of \cite[Theorem~22.4]{DalMaso}, originally proved for classes of linear operators and the corresponding quadratic functionals.

\begin{proposition}\label{prop:gamma-equivalence}
Let $\Omega\subset\R^n$ be a bounded open set and let $\varphi_h\in\Po_\Omega (p,\lambda,\Lambda)$ for $h\in \N\cup\{\infty\}$. The following are equivalent:
\begin{itemize}
\item[\emph{(a)}] $G_h(\,\cdot\,,U)$ $\Gamma$-converges to $G_\infty(\,\cdot\,,U)$ strongly in $L^p(\Omega)$ as $h\to\infty$ for every $U\in \mathscr A(\Omega)$,
\item[\emph{(b)}] $G_h(\,\cdot\,,\Omega)$ $\Gamma$-converges to $G_\infty(\,\cdot\,,\Omega)$ strongly in $L^p(\Omega)$ as $h\to\infty$,
\item[\emph{(c)}] $F_h$ $\Gamma$-converges to $F_\infty$ strongly in $L^p(\Omega)$ as $h\to\infty$.
\end{itemize}
\end{proposition}

\begin{proof}
{\bf (a) $\Longrightarrow$ (b)}. This implication is trivial.

{\bf (b) $\Longrightarrow$ (c)}. This implication is a consequence of~\cite[Theorem~21.1]{DalMaso}. 

{\bf (c) $\Longrightarrow$ (a)}. 
By Corollary~\ref{coro:compactness}, there exist a subsequence $(\varphi_{h_k})_k\subset \Po_\Omega(p,\lambda,\Lambda)$ and a Borel function $\widehat{\varphi}_\infty\in \Po_\Omega(p,\lambda,\Lambda)$ such that 
\[
\text{$G_{h_k}(\,\cdot\,,U)$ $\Gamma$-converges to $\widehat{G}_\infty(\,\cdot\,,U)$ strongly in $L^p(\Omega)$ as $k\to\infty$ for every $U\in \mathscr A(\Omega)$},
\]
where $\widehat{G}_\infty\colon L^p(\Omega)\times\mathscr{A}(\Omega)\to [0,\infty]$ is the functional associated with $\widehat{\varphi}_\infty$, defined as in~\eqref{eq:GhU}. 
Again, by~\cite[Theorem~21.1]{DalMaso},
\[
\text{$F_{h_k}$ $\Gamma$-converges to $\widehat{F}_\infty$ strongly in $L^p(\Omega)$ as $k\to\infty$},
\]
where $\widehat{F}_\infty\colon L^p(\Omega)\to [0,\infty]$ is the functional associated with $\widehat{\varphi}_\infty$, defined as in~\eqref{eq:Fh}. 
By assumption (c) and the uniqueness of the $\Gamma$-limit, we have
\[
F_\infty(u)=\widehat{F}_\infty(u)\quad\text{for every $u\in L^p(\Omega)$}.
\]
Therefore, by Proposition~\ref{prop:uniq_rapp}, we conclude that 
\[
\varphi_\infty(x,\xi)=\widehat\varphi_\infty(x,\xi)\quad\text{for a.e.\ $x\in\Omega$ and for every $\xi\in\R^n$}.
\]
Thus
\[
\text{$G_{h_k}(\,\cdot\,,U)$ $\Gamma$-converges to $G_\infty(\,\cdot\,,U)$ strongly in $L^p(\Omega)$ as $k\to\infty$ for every $U\in \mathscr A(\Omega)$}.
\]
The conclusion now follows from the Urysohn property of $\Gamma$-convergence~\cite[Proposition 8.3]{DalMaso}.
\end{proof}

As a consequence of Corollary~\ref{coro:compactness} and Proposition~\ref{prop:gamma-equivalence}, we obtain the following $\Gamma$-compactness result for the class of local functionals $(F_h)_h$.

\begin{corollary}\label{cor:local_Gamma-compactness}
Let $\Omega\subset\R^n$ be a bounded open set and let $\varphi_h\in\Po_\Omega (p,\lambda,\Lambda)$ for $h\in \N$. Then, there exist a subsequence $(\varphi_{h_k})_k\subset\Po_\Omega (p,\lambda,\Lambda)$ and a Borel function $\varphi_\infty\in \Po_\Omega (p,\lambda,\Lambda)$ such that
\begin{equation}\label{eq:Gamma-loc-comp}
\text{$F_{h_k}$ $\Gamma$-converges to $F_\infty$ strongly in $L^p(\Omega)$ as $k\to\infty$},
\end{equation}
where, for every $h\in\N\cup\{\infty\}$, $F_h$ are the functionals defined in~\eqref{eq:Fh}.
\end{corollary}

\begin{proof}
The conclusion follows immediately from Corollary~\ref{coro:compactness} and Proposition~\ref{prop:gamma-equivalence}.
\end{proof}

We can now prove the first implication in Theorem~\ref{thm:Gammas}. The proof is similar to that of~\cite[Proposition~4.3]{CCM25} in the linear case and relies on the strategies developed in~\cite{CuKrSc23,KrSc22}.

\begin{proposition}\label{prop:gamma1}
Let $\Omega\subset\R^n$ be a bounded open set. Let $A_0\in \M_{\R^n}^{\rm grad}(p,\lambda,\Lambda)$ satisfy~\eqref{eq:A0} and $\varphi_h\in\Po_\Omega (p,\lambda,\Lambda,A_0)$ for $h\in \N\cup\{\infty\}$. Let $F_h$ and $F_h^s$ for $h\in\N\cup\{\infty\}$ be the functionals defined in~\eqref{eq:Fh} and~\eqref{eq:Fhs}, respectively. Assume that 
\[
\text{$F_h$ $\Gamma$-converges to $F_\infty$ strongly in $L^p(\Omega)$ as $h\to\infty$}.
\]
Then
\[
\text{$F_h^s$ $\Gamma$-converges to $F_\infty^s$ strongly in $L^p(\R^n)$ as $h\to\infty$}.
\]
\end{proposition}

\begin{proof}
For the sake of exposition, we divide our proof in two parts.

{\bf $\Gamma$-liminf inequality}. Let $u_h\in L^p(\R^n)$ for $h\in\N\cup\{\infty\}$, be such that $u_h\to u_\infty$ strongly in $L^p(\R^n)$ as $h\to\infty$. We claim that
\begin{equation}\label{eq:Fh-claim}
F_\infty^s (u_\infty)\le\liminf_{h\to\infty}F_h^s (u_h).
\end{equation}
Without loss of generality, we may assume that
\[
\liminf_{h\to\infty}F_h^s (u_h)=\lim_{h\to\infty}F_h^s (u_h)<\infty.
\]
Therefore, property (NP2) in Proposition~\ref{prop:density-nonlocal} implies that 
\[
\sup_{h\in\mathbb{N}} \|\nabla^s u_h\|_{L^p(\mathbb{R}^n)} < \infty.
\]
Together with Proposition~\ref{prop:poincare}, this yields that $(u_h)_h$ is uniformly bounded in $H^{s,p}_0(\Omega)$. 
Hence, $u_\infty\in H^{s,p}_0(\Omega)$ and
\begin{align}\label{eq:weak-u}
u_h\to u_\infty\quad\text{weakly in $H^{s,p}_0(\Omega)$ as $h\to\infty$}.
\end{align}
By Corollary~\ref{coro:equiv}, up to a not relabeled subsequence, there exist $w_h\in W^{1,p}(\Omega)$ for $h\in\N\cup\{\infty\}$ such that
\begin{equation}\label{eq:local_nonlocal_gradients}
\nabla w_h=\nabla^s u_h\quad\text{in $\Omega$ for every $h\in\N\cup\{\infty\}$},
\end{equation}
and
\[
w_h\to w_\infty\quad\text{weakly in $W^{1,p}(\Omega)$ and strongly in $L^p(\Omega)$ as $h\to\infty$}.
\]
On the one hand, by \eqref{eq:local_nonlocal_gradients} and Proposition~\ref{prop:gamma-equivalence} we obtain
\begin{align*}
\int_\Omega \varphi_{A_\infty}(x,\nabla^s u_\infty(x))\, \d x=G_\infty(w_\infty,\Omega) \le \liminf_{h\to\infty} G_h (w_h,\Omega)=\liminf_{h\to\infty}\int_\Omega \varphi_{A_h}(x,\nabla^s u_h(x))\, \d x,
\end{align*}
where $G_h\colon L^p(\Omega)\times\mathscr A(\Omega)\to [0,\infty]$ is the functional defined in~\eqref{eq:GhU} for $h\in\N\cup\{\infty\}$. On the other hand, since $\varphi_h\in \Po_\Omega(p,\lambda,\Lambda,A_0)$, by~\eqref{eq:weak-u} and (NP1) of Proposition~\ref{prop:density-nonlocal}, we get
\begin{align*}
\int_{\R^n\setminus\Omega} \varphi_{A_0}(x,\nabla^s u_\infty(x))\, \d x&\le \liminf_{h\to\infty}\int_{\R^n\setminus \Omega} \varphi_{A_0}(x,\nabla^s u_h(x))\, \d x\\
&=\liminf_{h\to\infty}\int_{\R^n\setminus \Omega} \varphi_h(x,\nabla^s u_h(x))\, \d x.
\end{align*}
Combining the above inequalities, we conclude that
\begin{align*}
F_\infty^s (u_\infty)& =\int_\Omega \varphi_{A_\infty}(x,\nabla^s u_\infty(x))\, \d x + \int_{\R^n\setminus\Omega} \varphi_{A_0}(x,\nabla^s u_\infty(x))\, \d x \\
& \le \liminf_{h\to\infty}\int_\Omega \varphi_{A_h}(x,\nabla^s u_h(x))\, \d x+\liminf_{h\to\infty}\int_{\R^n\setminus \Omega} \varphi_{A_h}(x,\nabla^s u_h(x))\, \d x\\
& \le \liminf_{h\to\infty} F_h^s (u_h),
\end{align*}
which proves~\eqref{eq:Fh-claim}.

{\bf $\Gamma$-limsup inequality}.  
Let $u_\infty\in L^p(\R^n)$ be fixed.
We want to show the existence of a sequence $(u_h)_h\subset L^p(\R^n)$ such that $u_h\to u_\infty$ strongly in $L^p(\R^n)$ as $h\to\infty$, and
\begin{equation}\label{eq:Fh-Gammalimsup}
F_\infty^s (u_\infty) \ge \limsup_{h\to\infty} F_h^s (u_h).
\end{equation}
It suffices to consider the case $u_\infty \in H^{s,p}_0(\Omega)$,~\eqref{eq:Fh-Gammalimsup} being otherwise trivial in virtue of~\eqref{eq:Fhs}.
By Proposition~\ref{prop:equiv}(i), there exists $v_\infty\in W^{1,p}_{\rm loc}(\R^n)$ such that
\begin{equation}\label{eq:nablav}
\nabla v_\infty=\nabla^s u_\infty\quad\text{in $\R^n$}.
\end{equation}
By Proposition~\ref{prop:gamma-equivalence} and, in particular, by the $\Gamma$-limsup inequality for $(G_h(\cdot,\Omega))_h$, there exists a recovery sequence $(v_h)_h\subset W^{1,p}(\Omega)$ strongly converging to $v_\infty$ in $L^p(\Omega)$ as $h\to\infty$ such that
\begin{equation}\label{eq:new-recovery}
\lim_{h\to\infty} G_h (v_h,\Omega) = G_\infty (v_\infty,\Omega).
\end{equation}
Thus, by the definition of $G_h$ in \eqref{eq:GhU} and the lower bound in (P2) of Proposition~\ref{prop:density-loc}, the sequence $(v_h)_h$ is bounded in $W^{1,p}(\Omega)$. Hence
\[
v_h\to v_\infty\quad\text{weakly in $W^{1,p}(\Omega)$ as $h\to\infty$}.
\]
Let $\varepsilon\in (0,1)$ be fixed and let $K_\varepsilon\subset \Omega$ be a compact set such that 
\begin{equation}\label{eq:Ke}
\int_{\Omega\setminus K_\varepsilon}(1+|\nabla v_\infty(x)|^p)\,\d x<\varepsilon.
\end{equation}
Let $U_\varepsilon',U_\varepsilon''$ be open sets with $K_\varepsilon\subset\subset U_\varepsilon'\subset\subset U_\varepsilon''\subset\subset \Omega$ and let $V_\varepsilon\coloneqq \Omega\setminus K_\varepsilon$. 
By the fundamental estimates (see~\cite[Definition 18.2 and Theorem~19.6]{DalMaso}), for every $h\in\N$ there exist a function $\varphi_h^\varepsilon\in C_c^1(U_\varepsilon'')$ with $0\le \varphi_h^\varepsilon\le 1$ in $U_\varepsilon''$ and $\varphi_h^\varepsilon=1$ on $\overline{U_\varepsilon'}$, and a constant $M_\varepsilon>0$ such that
\begin{equation*}
\begin{aligned}
G_h (\varphi_h^\varepsilon v_h+(1-\varphi_h^\varepsilon)v_\infty, U_\varepsilon'\cup V_\varepsilon) & \le (1+\varepsilon) G_h(v_h,U_\varepsilon'') + G_h(v_\infty,V_\varepsilon)	\\
&\quad+\varepsilon(\|v_h\|_{L^p((U_\varepsilon''\setminus U_\varepsilon')\cap V_\varepsilon)}^p+\|v_\infty\|_{L^p((U_\varepsilon''\setminus U_\varepsilon')\cap V_\varepsilon)}^p+1)\\
&\quad+M_\varepsilon\|v_h-v_\infty\|_{L^p((U_\varepsilon''\setminus U_\varepsilon')\cap V_\varepsilon)}.
\end{aligned}
\end{equation*}
For every $h\in\N$, define
\begin{equation}\label{eq:vhe}
v_h^\varepsilon\coloneqq \varphi_h^\varepsilon v_h+(1-\varphi_h^\varepsilon)v_\infty\in W^{1,p}(\Omega).
\end{equation}
Since $U_\varepsilon'\cup V_\varepsilon=\Omega$, we have
\begin{align*}
G_h(v_h^\varepsilon,\Omega)&\le (1+\varepsilon)G_h(v_h,\Omega)+G_h(v_\infty,\Omega\setminus K_\varepsilon)\\
&\quad+\varepsilon(\|v_h\|_{L^p(\Omega)}^p+\|v_\infty\|_{L^p(\Omega)}^p+1)+M_\varepsilon\|v_h-v_\infty\|_{L^p(\Omega)}.
\end{align*}
Thus, by~\eqref{eq:new-recovery},~\eqref{eq:Ke}, and (P2) of Proposition~\ref{prop:density-loc} we derive
\begin{equation}\label{eq:fund-est}
\limsup_{h\to\infty} G_h(v_h^\varepsilon,\Omega)\le (1+\varepsilon) G_\infty (v_\infty,\Omega)+\varepsilon (a_1+2\|v_\infty\|_{L^p(\Omega)}^p+1).
\end{equation}
In particular, by applying again (P2), from~\eqref{eq:fund-est} we deduce
\begin{equation}\label{eq:vheps-est}
\sup_{h\in\N}\|\nabla v_h^\varepsilon\|_{L^p(\Omega)}<\infty.
\end{equation}
Moreover, by \eqref{eq:vhe}
\[
\|v_h^\varepsilon -v_\infty\|_{L^p(\Omega)}=\|\varphi_h^\varepsilon(v_h-v_\infty)\|_{L^p(\Omega)}\le \|v_h-v_\infty\|_{L^p(\Omega)}\to 0\quad\text{as $h\to\infty$}.
\]
Thus 
\begin{equation}\label{eq:vheps-v}
v_h^\varepsilon\to v_\infty\quad\text{strongly in $L^p(\Omega)$ as $h\to\infty$}.
\end{equation}
Let us trivially extend $v_h^\varepsilon-v_\infty\in W^{1,p}_0(\Omega)$ to a function in $W^{1,p}(\R^n)$ and define
\[
w_h^\varepsilon\coloneqq(-\Delta)^{\frac{1-s}{2}}(v_h^\varepsilon-v_\infty)\quad\text{for every $h\in\N$}.
\]
By Proposition~\ref{prop:equiv}(ii), for every $h\in\N$ we have
\begin{equation}\label{eq:nablaswh}
w_h^\varepsilon \in H^{s,p}(\R^n),\qquad\nabla^s w_h^\varepsilon=\nabla (v_h^\varepsilon-v_\infty)\quad\text{in $\R^n$}. 
\end{equation}
Moreover, by~\eqref{eq:frac-lap-est}, there exists a constant $C>0$, depending only on $n$ and $s$, such that
\[
\|w_h^\varepsilon\|_{L^p(\R^n)}\le C\|v_h^\varepsilon-v_\infty\|_{L^p(\Omega)}^s\|\nabla v_h^\varepsilon-\nabla v_\infty\|_{L^p(\Omega)}^{1-s}.
\]
Together with~\eqref{eq:vheps-est} and~\eqref{eq:vheps-v}, the above inequality implies that 
\begin{equation}\label{eq:wh-strong}
w_h^\varepsilon\to 0\quad\text{strongly in $L^p(\R^n)$ as $h\to\infty$}.
\end{equation}
Moreover, \eqref{eq:vheps-est} and \eqref{eq:nablaswh} imply
\begin{equation}\label{eq:wh-weak}
w_h^\varepsilon\to 0\quad\text{weakly in $H^{s,p}(\R^n)$ as $h\to\infty$}.
\end{equation}
Let $\chi^\varepsilon\in C_c^\infty(\Omega)$ satisfy $0\le \chi^\varepsilon\le 1$ on $\Omega$ and $\chi^\varepsilon=1$ on $\overline{U_\varepsilon''}$.
We define
\[
u_h^\varepsilon\coloneqq u_\infty+\chi^\varepsilon w_h^\varepsilon\in H^{s,p}_0(\Omega)\quad\text{for every $h\in\N$}.
\]
By~\eqref{eq:wh-strong},~\eqref{eq:wh-weak}, and Proposition~\ref{prop:Leibniz} we derive
\begin{equation}\label{eq:uhepsilon1}
u_h^\varepsilon\to u_\infty\quad\text{strongly in $L^p(\R^n)$ and weakly in $H^{s,p}_0(\Omega)$ as $h\to\infty$}.
\end{equation}
We also set
\[
R_h^\varepsilon\coloneqq\nabla^s(\chi^\varepsilon w_h^\varepsilon)-\chi^\varepsilon\nabla^sw_h^\varepsilon\quad\text{for $h\in\N$}.
\]
By~\eqref{eq:nabla_NL} and~\cite[Lemma~2.3]{ComiStefani23}, there exists a positive constant $C$ such that
\begin{align*}
\left\|R_h^\varepsilon\right\|_{L^p(\R^n)}\le C\left\|\chi^\varepsilon\right\|_{W^{1,\infty}(\R^n)}\left\|w_h^\varepsilon\right\|_{L^p(\R^n)}\quad\text{for every $h\in\N$}.
\end{align*}
Then, by~\eqref{eq:wh-strong},
\begin{align}\label{eq:Rh-conv}
R_h^\varepsilon\to 0\quad\text{strongly in $L^p(\R^n;\R^n)$ as $h\to\infty$}.
\end{align}
Thanks to~\eqref{eq:nablaswh}, we can write
\begin{align}\label{eq:nablauh-for}
\nabla^s u_h^\varepsilon&=\nabla^s u_\infty+\chi^\varepsilon\nabla^sw_h^\varepsilon+R_h^\varepsilon=\nabla^s u_\infty+\chi^\varepsilon\nabla(v_h^\varepsilon-v_\infty)+R_h^\varepsilon\quad\text{in $\R^n$}.
\end{align}
We consider the following decomposition
\begin{align}\label{eq:Fhuh}
F_h^s (u_h^\varepsilon)&=\int_{U_\varepsilon''}\varphi_{A_h}(x,\nabla^s u_h^\varepsilon(x))\, \d x+\int_{\Omega\setminus U_\varepsilon''}\varphi_{A_h}(x,\nabla^s u_h^\varepsilon(x))\, \d x\\
&\quad+\int_{\R^n\setminus\Omega}\varphi_{A_0}(x,\nabla^s u_h^\varepsilon(x))\, \d x\nonumber.
\end{align}
By~\eqref{eq:Rh-conv},~\eqref{eq:nablauh-for}, and (NP1)--(NP2) in Proposition~\ref{prop:density-nonlocal}, the last integral in~\eqref{eq:Fhuh} satisfies
\begin{align}\label{eq:Fh-1}
\lim_{h\to\infty}\int_{\mathbb{R}^n\setminus\Omega}\varphi_{A_0}(x,\nabla^s u_h^\varepsilon(x))\, \d x&=\lim_{h\to\infty}\int_{\mathbb{R}^n\setminus\Omega}\varphi_{A_0}(x,\nabla^s u_\infty(x)+R_h^\varepsilon(x))\, \d x\nonumber\\
&=\int_{\mathbb{R}^n\setminus\Omega}\varphi_{A_0}(x,\nabla^s u_\infty(x))\, \d x.
\end{align}
As for the second integral in~\eqref{eq:Fhuh}, we observe that, by~\eqref{eq:vhe} and~\eqref{eq:nablauh-for},
\[
\nabla^s u_h^\varepsilon = \nabla^s u_\infty + R_h^\varepsilon \quad \text{in } \Omega\setminus U_\varepsilon''.
\]
Therefore, by~\eqref{eq:nablav},~\eqref{eq:Ke},~\eqref{eq:Rh-conv}, and (NP2) in Proposition~\ref{prop:density-nonlocal}, we obtain
\begin{align}\label{eq:Fh-2}
&\limsup_{h\to \infty}\int_{\Omega\setminus U_\varepsilon''}\varphi_{A_h}(x,\nabla^s u_h^\varepsilon(x))\, \d x \nonumber\\
&\le a_1 \lim_{h\to \infty}\int_{\Omega\setminus U_\varepsilon''}\left(1 + 2^{p-1}|\nabla^s u_\infty(x)|^p + 2^{p-1}|R_h^\varepsilon(x)|^p\right)\,\d x \nonumber\\
&= a_1 \int_{\Omega\setminus U_\varepsilon''}\left(1 + 2^{p-1}|\nabla^s u_\infty(x)|^p\right)\,\d x \le a_1 2^{p-1}\varepsilon.
\end{align}
Finally, for the first integral in~\eqref{eq:Fhuh}, we observe that, since $\chi^\varepsilon = 1$ in $U_\varepsilon''$, by~\eqref{eq:nablav} and~\eqref{eq:nablauh-for} we have
\[
\nabla^s u_h^\varepsilon = \nabla v_\infty + \nabla(v_h^\varepsilon - v_\infty) + R_h^\varepsilon = \nabla v_h^\varepsilon + R_h^\varepsilon \quad \text{in } U_\varepsilon''.
\]
Hence, by (NP3) in Proposition~\ref{prop:density-nonlocal} and Hölder's inequality, we obtain
\begin{align*}
&\int_{U_\varepsilon''}|\varphi_{A_h}(x,\nabla^s u_h^\varepsilon(x))-\varphi_{A_h}(x,\nabla v_h^\varepsilon(x))|\,\d x\\
&\le a_2 3^{p-2}(|\Omega|+\|\nabla^s u_h^\varepsilon\|_{L^p(\Omega)}^{p-1}+\|\nabla v_h^\varepsilon\|_{L^p(\Omega)}^{p-1})\|R^\varepsilon_h\|_{L^p(\R^n)}.
\end{align*}
Since \eqref{eq:vheps-est} and \eqref{eq:uhepsilon1} imply that $(v_h^\varepsilon)_h$ and $(u_h^\varepsilon)_h$ are respectively uniformly bounded in $W^{1,p}(\Omega)$ and $H^{s,p}_0(\Omega)$, in view of~\eqref{eq:Rh-conv} we deduce
\begin{align*}
\lim_{h\to \infty}\int_{U_\varepsilon''}|\varphi_{A_h}(x,\nabla^s u_h^\varepsilon(x))-\varphi_{A_h}(x,\nabla v_h^\varepsilon(x))|\,\d x=0.
\end{align*}
Therefore, by~\eqref{eq:fund-est} we get
\begin{align}\label{eq:Fh-3}
\limsup_{h\to\infty}\int_{U_\varepsilon''}\varphi_{A_h}(x,\nabla^s u_h^\varepsilon(x))\, \d x &=\limsup_{h\to\infty}\int_{U_\varepsilon''}\varphi_{A_h}(x,\nabla v_h^\varepsilon(x))\, \d x\le \limsup_{h\to\infty} G_h(v_h^\varepsilon,\Omega)\nonumber\\
&\le (1+\varepsilon) G_\infty (v_\infty,\Omega)+\varepsilon(a_1+2\|v_\infty\|_{L^p(\Omega)}^p+1).
\end{align}
By combining~\eqref{eq:Fhuh},~\eqref{eq:Fh-1},~\eqref{eq:Fh-2} and~\eqref{eq:Fh-3}, for every $\varepsilon\in (0,1)$ we obtain
\begin{equation}\label{eq:uhepsilon2}
\limsup_{h\to\infty}F_h^s (u_h^\varepsilon)\le (1+\varepsilon) F_\infty^s (u_\infty)+\varepsilon(a_1+2\|v_\infty\|_{L^p(\Omega)}^p+1)+a_12^{p-1}\varepsilon.
\end{equation}
To deduce~\eqref{eq:Fh-Gammalimsup} from~\eqref{eq:uhepsilon2}, it suffices to apply a diagonal argument, similarly to that used in~\cite[Proposition~4.3]{CCM25}.
\end{proof}

Let us prove the converse implication.

\begin{proposition}\label{prop:gamma2}
Let $\Omega\subset\R^n$ be a bounded open set. Let $A_0\in \M_{\R^n}^{\rm grad}(p,\lambda,\Lambda)$ satisfy~\eqref{eq:A0} and $\varphi_h\in\Po_\Omega (p,\lambda,\Lambda,A_0)$ for $h\in \N\cup\{\infty\}$. Let $F_h$ and $F_h^s$ for $h\in\N\cup\{\infty\}$ be the functionals defined in~\eqref{eq:Fh} and~\eqref{eq:Fhs}, respectively. Assume that 
\[
\text{$F_h^s$ $\Gamma$-converges to $F_\infty^s$ strongly in $L^p(\R^n)$ as $h\to\infty$}.
\]
Then
\[
\text{$F_h$ $\Gamma$-converges to $F_\infty$ strongly in $L^p(\Omega)$ as $h\to\infty$}.
\]
\end{proposition}

\begin{proof}
By Corollary~\ref{cor:local_Gamma-compactness}, there exist a subsequence $(\varphi_{h_k})_k\subset \Po_\Omega(p,\lambda,\Lambda)$ and a Borel function $\widehat{\varphi}_\infty\in \Po_\Omega(p,\lambda,\Lambda)$ such that 
\begin{equation}\label{eq:Gamma_tilde_F}
    \text{$F_{h_k}$ $\Gamma$-converges to $\widehat{F}_\infty$ strongly in $L^p(\Omega)$ as $k\to\infty$},
\end{equation}
where $\widehat{F}_\infty\colon L^p(\Omega)\to [0,\infty]$ is the functional associated with $\widehat{\varphi}_\infty$, defined as in~\eqref{eq:Fh}. 
If we consider $\widehat{\varphi}_\infty\in \Po_\Omega(p,\lambda,\Lambda,A_0)$, by applying Proposition~\ref{prop:gamma1} to the subsequence $(F_{h_k}^s)_k$ we get that 
\[
\text{$F_{h_k}^s$ $\Gamma$-converges to $\widehat{F}_\infty^s$ strongly in $L^p(\R^n)$ as $k\to\infty$},
\]
where $\widehat{F}_\infty^s\colon L^p(\R^n)\to [0,\infty]$ is the functional associated with $\widehat{\varphi}_\infty$, defined as in~\eqref{eq:Fhs}. 
By the uniqueness of the $\Gamma$-limit, we obtain that 
\[
\widehat{F}_\infty^s(u) = F_\infty^s(u)\quad\text{for every $u\in L^p(\R^n)$}.
\]
Since both $\varphi_\infty$ and $\widehat{\varphi}_\infty$ coincide with $\varphi_{A_0}$ on $(\mathbb R^n\setminus\Omega)\times\mathbb R^n$, we obtain
\begin{equation*}
\int_\Omega \varphi_\infty(x, \nabla^s u(x))\,\d x
=
\int_\Omega \widehat{\varphi}_\infty(x, \nabla^s u(x))\,\d x
\quad \text{for every } u \in C^\infty_c(\Omega).
\end{equation*}
Hence, by Proposition~\ref{prop:uniq_rapp},
\[
\varphi_\infty(x,\xi) = \widehat{\varphi}_\infty(x,\xi) \quad \text{for a.e.\ $x \in \R^n$ and every $\xi \in \R^n$},
\] 
and by \eqref{eq:Gamma_tilde_F}, we get
\begin{equation*}
    \text{$F_{h_k}$ $\Gamma$-converges to $F_\infty$ strongly in $L^p(\Omega)$ as $k\to\infty$}.
\end{equation*}
The conclusion then follows from the Urysohn property of $\Gamma$-convergence and the arbitrariness of the subsequence.
\end{proof}

We can finally prove Theorem~\ref{thm:Gammas}.

\begin{proof}[Proof of Theorem~\ref{thm:Gammas}]
It is enough to combine Proposition~\ref{prop:gamma1} and Proposition~\ref{prop:gamma2}.
\end{proof}

We conclude this section with the following corollary, which is an immediate consequence of Theorem~\ref{thm:Gammas} and Corollary~\ref{cor:local_Gamma-compactness}.

\begin{corollary}\label{cor:Gamma-compactness}
Let $\Omega\subset\R^n$ be a bounded open set and let $A_0\in \M_{\R^n}^{\rm grad}(p,\lambda,\Lambda)$ satisfy~\eqref{eq:A0}. Let $\varphi_h\in\Po_\Omega (p,\lambda,\Lambda,A_0)$ for $h\in \N$. Then, there exist a subsequence $(\varphi_{h_k})_k\subset\Po_\Omega (p,\lambda,\Lambda,A_0)$ and a Borel function $\varphi_\infty\in \Po_\Omega (p,\lambda,\Lambda,A_0)$ such that
\begin{equation}\label{eq:Gamma-nonloc-comp}
\text{$F_{h_k}^s$ $\Gamma$-converges to $F_\infty^s$ strongly in $L^p(\R^n)$ as $k\to\infty$},
\end{equation}
where, for every $h\in\N\cup\{\infty\}$, $F_h^s$ are the functionals defined in~\eqref{eq:Fhs}.
\end{corollary}

\begin{proof}
By Corollary~\ref{cor:local_Gamma-compactness}, there exist a subsequence $(\varphi_{h_k})_k\subset\Po_\Omega (p,\lambda,\Lambda)$ and a function $\varphi_\infty\in \Po_\Omega (p,\lambda,\Lambda)$ such that
\[
\text{$F_{h_k}$ $\Gamma$-converges to $F_\infty$ strongly in $L^p(\Omega)$ as $k\to\infty$},
\]
where $F_h$ for $h\in\N\cup\{\infty\}$ are the functionals defined in~\eqref{eq:Fh}.
Hence, by extending $\varphi_\infty$ as an element of $\Po_\Omega(p,\lambda,\Lambda,A_0)$, Theorem~\ref{thm:Gammas} yields~\eqref{eq:Gamma-nonloc-comp}.
\end{proof}

\section{Equivalence between \texorpdfstring{$H$}{H}-convergence and \texorpdfstring{$\Gamma$}{Gamma}-convergence}\label{sect:equivGammaH}

In this final section, we prove the equivalence between $H$-convergence and $\Gamma$-convergence. More precisely, we have the following result.

\begin{theorem}\label{thm:Hs-Gammas}
Let $\Omega\subset\R^n$ be a bounded open set satisfying $|\partial\Omega|=0$ and let $A_0\in \M_{\R^n}^{\rm grad}(p,\lambda,\Lambda)$ satisfy~\eqref{eq:A0}. Let $A_h\in \M^{\rm grad}_\Omega (p,\lambda,\Lambda,A_0)$ for $h\in\N\cup\{\infty\}$ and let $\varphi_h\coloneqq \varphi_{A_h}\in\Po_\Omega(p,\lambda,\Lambda,A_0)$ for $h\in \N\cup\{\infty\}$ be the corresponding potential defined as in~\eqref{eq:varphiA}. Then 
\[
\text{$\mathcal A_h^s$ $H$-converges to $\mathcal A_\infty^s$ as $h\to\infty$},
\]
if and only if
\[
\text{$F_h^s$ $\Gamma$-converges to $F_\infty^s$ strongly in $L^p(\mathbb{R}^n)$ as $h\to\infty$}.
\]
\end{theorem}

In light of our previous equivalence results obtained in Sections~\ref{sec:H-equiv} and~\ref{sec:Gamm-equiv}, it suffices to prove that such convergences are equivalent in the local framework in order to complete the picture. As a consequence of~\cite[Theorem~4.5]{ADMZ} we can prove the following result.

\begin{proposition}\label{prop:H-Gamma}
Let $\Omega\subset\R^n$ be a bounded open set. Let $A_h\in \M^{\rm grad}_\Omega (p,\lambda,\Lambda)$ for $h\in\N\cup\{\infty\}$ and let $\varphi_h\coloneqq \varphi_{A_h}\in\Po_\Omega(p,\lambda,\Lambda)$ for $h\in \N\cup\{\infty\}$ be the corresponding potential defined as in~\eqref{eq:varphiA}. Then 
\begin{equation}\label{eq:H-Gamma}
\text{$\mathcal A_h$ $H$-converges to $\mathcal A_\infty$ as $h\to\infty$},
\end{equation}
if and only if
\begin{equation}\label{eq:Gamma-H}
\text{$F_h$ $\Gamma$-converges to $F_\infty$ strongly in $L^p(\Omega)$ as $h\to\infty$}.
\end{equation}
\end{proposition}

\begin{proof}

{\bf Step 1.} We prove that~\eqref{eq:Gamma-H} implies~\eqref{eq:H-Gamma}. To this aim, we proceed as in the proof of Corollary~\ref{coro:compactness}. Let us fix $f\in W^{-1,p}(\Omega)$ and set for every $h\in\N\cup\{\infty\}$
\begin{equation*}
F_h^f(u)\coloneqq\begin{cases}
     F_h(u)-\langle f,u\rangle_{W^{-1,p'}(\Omega)\times W^{1,p}_0(\Omega)}&\quad\text{if $u\in W^{1,p}_0(\Omega)$},\\
    \infty&\quad\text{otherwise}.
\end{cases}
\end{equation*}
By~\eqref{eq:Gamma-H} and the lower bound in (P2), we obtain that
\[
\text{$F_h^f$ $\Gamma$-converges to $F_\infty^f$ strongly in $L^p(\Omega)$ as $h\to\infty$}.
\]
Moreover, the functionals $(F_h^f)_h$ are equi-coercive and uniformly convex on $L^p(\Omega)$. Hence, letting $u_h^f$ be the unique minimisers of $F_h^f$ for $h\in\N\cup\{\infty\}$, by~\cite[Corollary~7.24]{DalMaso} we have
\begin{equation}\label{eq:conv_min}
u^f_h\to u^f_\infty\quad\text{weakly in $W^{1,p}_0(\Omega)$ and strongly in $L^p(\Omega)$ as $h\to\infty$}.
\end{equation}
Moreover,
\begin{equation}\label{eq:conv_en}
F_h^f(u_h^f) \to F^f_\infty(u^f_\infty)\quad\text{as $h\to\infty$}.
\end{equation}
Recalling that $\nabla_\xi\varphi_{A_h}=A_h$ for every $h\in\N\cup\{\infty\}$, we pass to the Euler-Lagrange equations of $F_h^f$ to obtain that $u_h^f$ satisfies for every $h\in\N\cup\{\infty\}$
\begin{equation*}
\begin{cases}
-\div(A_h(\,\cdot\,,\nabla u_h^f))=f &\text{in }\Omega, \\
u_h^f=0&\text{on }\partial\Omega.
\end{cases}
\end{equation*}
According to the notation introduced in Section~\ref{sec:H-def}, the convergence~\eqref{eq:conv_min} reads as
\begin{equation*}
    \mathcal B_h(f)\to\mathcal{B}_\infty(f)\quad\text{weakly in $W^{1,p}_0(\Omega)$ as $h\to\infty$},
\end{equation*}
which is (H1) in Definition~\ref{def:H-conv}. 
Therefore, to prove~\eqref{eq:H-Gamma}, it remains to prove the convergence of the momenta, that is (H2) in Definition~\ref{def:H-conv}. 
By~\eqref{eq:conv_min} and~\eqref{eq:conv_en}, we have
\[
G_h(u_h^f,\Omega)\to G_\infty(u^f_\infty,\Omega)\quad\text{as $h\to\infty$},
\]
where, for every $h\in\mathbb{N}\cup\{\infty\}$, $G_h$ are the functionals defined in \eqref{eq:GhU}.
Hence, by~\cite[Theorem~4.5]{ADMZ},
\[
A_h(\,\cdot\,,\nabla u_h^f)\to A_\infty(\,\cdot\,,\nabla u_\infty^f)\quad\text{weakly in $L^{p'}(\Omega;\R^n)$ as $h\to\infty$},
\]
that is, according to the notation introduced in Section~\ref{sec:H-def},
\begin{equation*}
    \mathcal{M}_h(f)\to \mathcal{M}_\infty(f)\quad\text{weakly in $L^{p'}(\Omega;\R^n)$ as $h\to\infty$}.
\end{equation*}
This gives (H2) in Definition~\ref{def:H-conv} and concludes the proof of~\eqref{eq:H-Gamma}.

{\bf Step 2.} Assume~\eqref{eq:H-Gamma}. 
By Corollary~\ref{cor:local_Gamma-compactness}, there exists a subsequence $(\varphi_{h_k})_k\subset \Po_\Omega(p,\lambda,\Lambda)$ and a Borel function $\widehat{\varphi}_\infty\in \Po_\Omega(p,\lambda,\Lambda)$ such that 
\[
\text{$F_{h_k}$ $\Gamma$-converges to $\widehat{F}_\infty$ strongly in $L^p(\Omega)$ as $k\to\infty$},
\]
where $\widehat{F}_\infty\colon L^p(\Omega)\to [0,\infty]$ is the functional associated with $\widehat{\varphi}_\infty$ defined as in~\eqref{eq:Fh}. 
By Step 1, we conclude that 
\[
\text{$\mathcal A_{h_k}$ $H$-converges to $\widehat{\mathcal A}_\infty$ as $k\to\infty$},
\]
where $\widehat{\mathcal A}_\infty$ is the differential operator associated with $\widehat{\varphi}_\infty$ defined as in~\eqref{eq:Ah}. Since the $H$-limit is unique, by~\eqref{eq:H-Gamma} we deduce that 
\[
A_\infty(x,\xi)=\widehat A_\infty(x,\xi)\quad\text{for a.e.\ $x\in\Omega$ and for every $\xi\in\R^n$}.
\]
Since, by Proposition~\ref{prop:density-loc},
\[
\varphi_\infty(x,0)=\widehat\varphi_\infty(x,0)=0\quad\text{for a.e.\ $x\in\Omega$},
\]
the above identity implies that 
\[
\varphi_\infty(x,\xi)=\widehat \varphi_\infty(x,\xi)\quad\text{for a.e.\ $x\in\Omega$ and for every $\xi\in\R^n$},
\]
and the conclusion~\eqref{eq:Gamma-H} follows by the Urysohn property of $\Gamma$-convergence.
\end{proof}

We can finally prove Theorem~\ref{thm:Hs-Gammas}.
\begin{proof}[Proof of Theorem~\ref{thm:Hs-Gammas}]
It suffices to combine Theorem~\ref{thm:Hs}, Theorem~\ref{thm:Gammas}, and Proposition~\ref{prop:H-Gamma}.
\end{proof}


\section*{Declarations}


\smallskip

\noindent{\it \textbf{Conflict of interest.}} On behalf of all authors, the corresponding author states that there is no conflict of interest.

\smallskip

\noindent{\it \textbf{Funding information.}} The authors G.~C.\ Brusca, M.\ Caponi, A.\ Carbotti and F.\ Paronetto are members of GNAMPA - Istituto Nazionale di Alta Matematica (INdAM) - and acknowledge the support of the INdAM - GNAMPA 2026 Project ``Analisi variazionale per operatori locali e nonlocali possibilmente singolari o degeneri'' (CUP: E53C25002010001).

M.\ Caponi has been founded by the European Union - NextGenerationEU under the Italian Ministry of University and Research (MUR) National Centre for HPC, Big Data and Quantum Computing (CN\_00000013 – CUP: E13C22001000006).






This work was initiated while G.~C. Brusca was visiting the Department of Mathematics of the Universitat Politècnica de Catalunya.
This research was partially supported by the Centre de Recerca Matemàtica of Barcelona (CRM), under the International Programme for Research in Groups (IP4RG), entitled: ``A Mathematical Approach to the Homogenization of Nonlocal Composite Materials: $H$-convergence, $G$-convergence, and $\Gamma$-convergence''.

\smallskip

\noindent{\it \textbf{Data availability statement.}} Data sharing not applicable to this article as no datasets were generated or analysed during the current study.



\begin{thebibliography}{99}

\bibitem{ABSS}
{\sc R.~Alicandro, A.~Braides, M.~Solci, and G.~Stefani},
{\em Topological singularities arising from fractional-gradient energies},
Math.\ Ann.\ {\bf 393} (2025), 71--111.

\bibitem{ACFS}
{\sc S.~Almi, M.~Caponi, M.~Friedrich, and F.~Solombrino}, 
{\em Riesz fractional gradient functionals defined on partitions: nonlocal-to-local variational limits}, Preprint (2025), arXiv: \href{https://arxiv.org/abs/2510.04881}{2510.04881}

\bibitem{arroyo}
{\sc A.~Arroyo-Rabasa},
{\em Functional and variational aspects of nonlocal operators associated with linear PDEs},
Nonlinear Anal. {\bf 251} (2025), Paper No.\ 113683


\bibitem{ADMZ}
{\sc N.~Ansini, G.~Dal Maso, and C.~I.~Zeppieri},
{\em New results on $\Gamma$-limits of integral functionals},
Ann.\ Inst.\ H.\ Poincaré C Anal.\ Non Linéaire {\bf 31} (2014), 185--202


\bibitem{BCM20}
{\sc J.\ C.\ Bellido, J.\ Cueto, and C.\ Mora-Corral},
{\em Fractional Piola identity and polyconvexity in fractional spaces},
Ann.\ Inst.\ H.\ Poincaré C Anal.\ Non Linéaire {\bf 37} (2020), 955--981

\bibitem{BCM21}
{\sc J.\ C.\ Bellido, J.\ Cueto, and C.\ Mora-Corral},
{\em $\Gamma$-convergence of polyconvex functionals involving $s$-fractional gradients to their local counterparts},
Calc.\ Var.\ Partial Differential Equations {\bf 60} (2021), Paper No.\ 7

\bibitem{Bra}
{\sc A.~Braides}, 
$\Gamma$-convergence for beginners,
Oxford Lecture Ser.\ Math.\ Appl., 22
Oxford University Press, Oxford, 2002



\bibitem{BDF}
{\sc A.~Braides and A.~Defranceschi},
Homogenization of multiple integrals,
Oxford Lecture Ser.\ Math.\ Appl., 12
The Clarendon Press, Oxford University Press, New York, 1998


\bibitem{BCCS22}
{\sc E.~Bruè, M.~Calzi, G.~E.~Comi, and G.~Stefani},
{\em A distributional approach to fractional Sobolev spaces and fractional variation: asymptotics II},
C.\ R.\ Math.\ Acad.\ Sci.\ Paris {\bf 360} (2022), 589--626


\bibitem{bbcejw}
{\sc A.~Buchinger, K.~Burazin, I.~Crnjac, M.~Erceg, M.~Jolić, and M.~Waurick},
{\em Characterisation of homogenisation for nonlocal diffusion by local topologies}, 
Preprint 2026, arXiv: \href{https://arxiv.org/abs/2601.17579}{2601.17579}

\bibitem{bpw}
{\sc A.~Buchinger, M.~Porfido, and M.~Waurick},
{\em The Free Lunch Theorem~of Homogenisation}, 
Preprint 2026, arXiv: \href{https://arxiv.org/abs/2606.04498}{2606.04498}

\bibitem{Caponi26}
{\sc M.~Caponi},
{\em A Survey on the Div-Curl Lemma~and Some Extensions to Fractional Sobolev Spaces}, 
In: J. Duran, A. Maione (eds) {\em New Frontiers in Homogenization and Fractional Calculus},
Trends in Mathematics series, Springer-Birkhäuser (2026)

\bibitem{CCM25}
{\sc M.~Caponi, A.~Carbotti, and A.~Maione},
{\em $H$-compactness for nonlocal linear operators in fractional divergence form},
Calc.\ Var.\ Partial Differential Equations {\bf 64} (2025), Paper No.\ 290



\bibitem{CPDMDF}
{\sc V.~Chiadò Piat, G.~Dal Maso, and A.~Defranceschi},
{\em $G$-convergence of monotone operators},
Ann.\ Inst.\ H.\ Poincaré C Anal.\ Non Linéaire {\bf 7} (1990), 123--160

\bibitem{ComiStefani19}
{\sc G.~E.~Comi and G.~Stefani},
{\em A distributional approach to fractional Sobolev spaces and fractional variation: existence of blow-up}, 
J.\ Funct.\ Anal.\ {\bf 277} (2019), 3373--3435

\bibitem{ComiStefani22}
{\sc G.~E.~Comi and G.~Stefani},
{\em Leibniz rules and Gauss-Green formulas in distributional fractional spaces},
J.\ Math.\ Anal.\ Appl.\ {\bf 514} (2022), Paper No.\ 126312

\bibitem{ComiStefani23}
{\sc G.~E.~Comi and G.~Stefani},
{\em A distributional approach to fractional Sobolev spaces and fractional variation: asymptotics I}, 
Rev.\ Mat.\ Complut.\ {\bf 36} (2023), 491--569

\bibitem{CuKrSc23}
{\sc J.~Cueto, C.~Kreisbeck, and H.~Schönberger},
{\em A variational theory for integral functionals involving finite-horizon fractional gradients},
Fract.\ Calc.\ Appl.\ Anal.\ {\bf 26} (2023), 2001--2056


\bibitem{DalMaso}
{\sc G.~Dal Maso},
An introduction to $\Gamma$-convergence, 
Progr.\ Nonlinear Differential Equations Appl., 8
Birkhäuser Boston, Inc., Boston, MA, 1993

\bibitem{DASC}
{\sc R.~De Arcangelis and F.~Serra Cassano},
{\em On the convergence of solutions of degenerate elliptic equations in divergence form},
Ann.\ Mat.\ Pura Appl. {\bf 167} (1994), 1--23

\bibitem{degiorgi2}
{\sc E.~De Giorgi},
{\em $\Gamma$-convergenza e $G$-convergenza},
Boll.\ Un.\ Mat.\ Ital.\ A {\bf 14} (1977), 213--220

\bibitem{degiorgi-fran} 
{\sc E.~De Giorgi and T.~Franzoni},
{\em Su un tipo di convergenza variazionale}, 
Atti Accad.\ Naz.\ Lincei Rend.\ Cl.\ Sci.\ Fis.\ Mat.\ Nat.\ {\bf 58} (1975), 842--850

\bibitem{DGS}
{\sc E.~De Giorgi and S.~Spagnolo},
{\em Sulla convergenza degli integrali dell’energia per operatori ellittici del secondo ordine},
Boll.\ Un.\ Mat.\ Ital.\ {\bf 8} (1973), 391--411

\bibitem{EG}
{\sc L.~C.~Evans and R.~F.~Gariepy},
Measure theory and fine properties of functions,
Revised edition.
Textb.\ Math.\
CRC Press, Boca Raton, FL, 2015

\bibitem{FBRS17}
{\sc J.~Fernández Bonder, A.~Ritorto, and A.~M.~Salort}, 
{\em $H$-convergence result for nonlocal elliptic-type problems via Tartar’s method},
SIAM J.\ Math.\ Anal.\ {\bf 49} (2017), 2387--2408

\bibitem{FM86}
{\sc N.~Fusco and G.~Moscariello},
{\em Further results on the homogenization of quasilinear operators}
Ricerche Mat.\ {\bf 35} (1986), 231--246

\bibitem{FM87}
{\sc N.~Fusco and G.~Moscariello},
{\em On the homogenization of quasilinear divergence structure operators},
Ann.\ Mat.\ Pura Appl.\ {\bf 146} (1987), 1--13


\bibitem{KrSc22}
{\sc C.~Kreisbeck and H.~Schönberger},
{\em Quasiconvexity in the fractional calculus of variations: characterization of lower semicontinuity and relaxation},
Nonlinear Anal.\ {\bf 215} (2022), Paper No.\ 112625



\bibitem{Maione}
{\sc A.~Maione},
{\em $H$-convergence for equations depending on monotone operators in Carnot groups},
Electron.\ J.\ Differential Equations {\bf 2021} (2021), Paper No.\ 13

\bibitem{MPV}
{\sc A.~Maione, F.~Paronetto, and E.~Vecchi},
{\em $G$-convergence of elliptic and parabolic operators depending on vector fields},
ESAIM Control Optim.\ Calc.\ Var.\ {\bf 29} (2023), Paper No.\ 8

\bibitem{MuratI}
{\sc F.~Murat},
{\em Compacité par compensation},
Ann.\ Scuola Norm.\ Sup.\ Pisa Cl.\ Sci.\ {\bf 5} (1978), 489--507



\bibitem{ShiehSpector15}
{\sc T.-T.~Shieh and D.~E.~Spector},
{\em On a new class of fractional partial differential equations},
Adv.\ Calc.\ Var.\ {\bf 8} (2015), 321--336

\bibitem{ShiehSpector18}
{\sc T.-T.~Shieh and D.~E.~Spector},
{\em On a new class of fractional partial differential equations II},
Adv.\ Calc.\ Var.\ {\bf 11} (2018), 289--307

\bibitem{simon}
{\sc J.~Simon},
{\em Régularité de la solution d'une équation non linéaire dans $\R^N$},
In: P. Bénilan, J. Robert (eds) {\em Journées d’Analyse Non Linéaire}. Lecture Notes in Mathematics, vol 665,
Springer, Berlin, Heidelberg


\bibitem{Spagnolo}
{\sc S.~Spagnolo},
{\em Sul limite delle soluzioni di problemi di Cauchy relativi all’equazione del calore},
Ann.\ Scuola Norm.\ Sup.\ Pisa Cl.\ Sci.\ {\bf 21} (1967), 657--699

\bibitem{Spagnolo67}
{\sc S.~Spagnolo},
{\em Una caratterizzazione degli operatori differenziali autoaggiunti del $2\sp{\circ }$ ordine a coefficienti misurabili e limitati},
Rend.\ Sem.\ Mat.\ Univ.\ Padova {\bf 39} (1967), 56--64

\bibitem{Spagnolo68}
{\sc S.~Spagnolo},
{\em Sulla convergenza di soluzioni di equazioni paraboliche ed ellittiche},
Ann.\ Scuola Norm.\ Sup.\ Pisa {\bf 22} (1968), 571--597; errata: {\bf 22} (1968), 673

\bibitem{Stefani}
{\sc G.~Stefani},
{\em On blow-ups of sets with finite fractional variation}, 
J.\ Geom.\ Anal.\ {\bf 36} (2026), Paper No.\ 223



\bibitem{Tartar0}
{\sc L.~Tartar},
{\em Convergence d'operateurs defferentiels},
Proceedings of the Meeting ``Analisi convessa e Applicazioni'', Roma (1974).

\bibitem{tartar-murat}
{\sc L.~Tartar},
{\em Cours Peccot}, {Coll\`ege de France}, (1977).
Partially written in: {\sc F.~Murat},
{\em $H$-convergence - S\'eminaire d'Analyse Fonctionnelle et Num\'erique},
Universit\'e\ d'Alger, 1977-78. English translation:
{\sc F.~Murat and L.~Tartar}, {\it H-Convergence}, in {\it Topics in the Mathematical Modelling of Composite Materials} (A. Cherkaev, R. Kohn, editors), Birkh\"auser, Boston, 1997, 21--43.

\bibitem{Tartar}
{\sc L.~Tartar},
The general theory of homogenization. A personalised introduction. 
Lecture Notes of the Unione Matematica Italiana 7. 
Springer, Berlin, 2009.

\bibitem{Waurick}
{\sc M.~Waurick},
{\em Nonlocal $H$-convergence},
Calc.\ Var.\ Partial Differential Equations {\bf 57} (2018), Paper No.\ 159

 \bibitem{Waurick2}
{\sc M.~Waurick},
{\em Nonlocal $H$-convergence for topologically nontrivial domains}, 
J.\ Funct.\ Anal.\ {\bf 288} (2025), Paper No.\ 110710

\bibitem{Zeidler3}
{\sc E.~Zeidler},
Nonlinear functional analysis and its applications. III,
Springer-Verlag, New York, 1985


\bibitem{Silhavy20}
{\sc M.~Šilhavý},
{\em Fractional  vector analysis  based on invariance requirements (critique of coordinate approaches)},
Contin.\ Mech.\ Thermodyn.\ {\bf 32} (2020), 207--228



\end{thebibliography}
\end{document}